\documentclass[a4paper,11pt]{amsart}

\usepackage{fancyhdr}
\usepackage{amssymb,latexsym, amsfonts, amsmath, amsthm}
\usepackage{graphicx}
\usepackage{setspace}
\usepackage{enumerate}
\usepackage{wrapfig}
\usepackage{mathdots}
\usepackage{lscape}
\usepackage[headheight=13.0pt,paperwidth=210mm,paperheight=297mm,textwidth=420pt,textheight=615pt ]{geometry}
\usepackage{rotating}
\usepackage{epstopdf}
\usepackage[backgroundcolor=green!40, linecolor=green!40]{todonotes}
\usepackage{tikz}
\usepackage{todonotes}
\usepackage{MnSymbol}
\DeclareFontFamily{U}{matha}{\hyphenchar\font45}
\DeclareFontShape{U}{matha}{m}{n}{
      <5> <6> <7> <8> <9> <10> gen * matha
      <10.95> matha10 <12> <14.4> <17.28> <20.74> <24.88> matha12
      }{}
\DeclareSymbolFont{matha}{U}{matha}{m}{n}
\DeclareFontFamily{U}{mathx}{\hyphenchar\font45}
\DeclareFontShape{U}{mathx}{m}{n}{
      <5> <6> <7> <8> <9> <10>
      <10.95> <12> <14.4> <17.28> <20.74> <24.88>
      mathx10
      }{}
\DeclareSymbolFont{mathx}{U}{mathx}{m}{n}

\DeclareMathSymbol{\obot}         {2}{matha}{"6B}
\DeclareMathSymbol{\bigobot}       {1}{mathx}{"CB}

\newcommand{\Lie}{\text{Lie}}

\usetikzlibrary{arrows,calc,matrix,fit}

\theoremstyle{definition}
\newtheorem{theorem}[equation]{Theorem}
\newtheorem{proposition}[equation]{Proposition}

\newtheorem{notation}[equation]{Notation}

\newtheorem{definition}[equation]{Definition}
\newtheorem{lemma}[equation]{Lemma}
\newtheorem{corollary}[equation]{Corollary}
\newtheorem{remark}[equation]{Remark}

\date{\today}

\usepackage[utf8]{inputenc}

\newcommand{\mfa}{\mf{a}}

\numberwithin{equation}{section}

\newcommand{\be}{\begin{enumerate}}
\newcommand{\ee}{\end{enumerate}}
\newcommand{\bi}{\begin{itemize}}
\newcommand{\ei}{\end{itemize}}
\newcommand{\beq}{\begin{equation}}
\newcommand{\eeq}{\end{equation}}
\newcommand{\mf}{\mathfrak}
\newcommand{\ra}{\rightarrow }

\def\End{\operatorname{End}}
\def\Aut{\operatorname{Aut}}

\def\U{\operatorname{U}}
\def\SU{\operatorname{SU}}
\def\Hom{\operatorname{Hom}}
\def\Irr{\operatorname{Irr}}

\def\dim{\operatorname{dim}}

\def\id{\operatorname{id}}

\def\Gal{\operatorname{Gal}}
\def\GL{\operatorname{GL}}

\def\SU{\operatorname{SU}}

\def\Sp{\operatorname{Sp}}
\def\Res{\operatorname{Res}}

\def\ind{\operatorname{ind}}
\def\I{\operatorname{I}}
\def\R{\operatorname{R}}

\def\Nrd{\operatorname{Nrd}}
\def\SO{\text{SO}}

\def\CC{\mathbb{C}}

\def\NN{\mathbb{N}}

\def\ZZ{\mathbb{Z}}
\def\bbU{\mathbb{U}}
\def\bbQ{\mathbb{Q}}
\def\bbZ{\mathbb{Z}}
\def\bbR{\mathbb{R}}
\def\bbG{\mathbb{G}}
\def\bbN{\mathbb{N}}
\def\bbP{\mathbb{P}}
\def\bbC{\mathbb{C}}

\def\cY{{\rm Y}}
\def\A{{\rm A}}
\def\B{{\rm B}}
\def\C{{\mathcal{C}}}
\def\D{{\rm D}}
\def\E{{\rm E}}
\def\F{{\rm F}}
\def\G{{\rm G}}
\def\H{{\rm H}}
\def\J{{\rm J}}
\def\K{{\rm K}}
\def\L{{\rm L}}
\def\M{{\rm M}}
\def\N{{\rm N}}
\def\O{{\rm O}}
\def\P{{\rm P}}
\def\Q{{\rm Q}}
\def\SS{{\rm S}}

\def\V{{\rm V}}
\def\W{{\rm W}}

\def\Vv{\mathcal{V}}
\def\Ww{\mathcal{W}}

\def\Yy{\mathcal{Y}}

\def\b{\beta}
\def\d{\delta}
\def\e{\varepsilon}
\def\g{\gamma}

\def\k{\text{k}}
\def\l{\lambda}

\def\o{\mathfrak{o}}
\def\p{\mathfrak{p}}

\def\th{\theta}

\def\w{\varpi}

\def\Ga{\Gamma}
\def\La{\Lambda}

\def\Y{\Upsilon}

\def\>{\geqslant}
\def\<{\leqslant}

\def\daniel#1{\textcolor{black}{#1}}

\def\tG{\widetilde{G}}
\def\tH{\widetilde{H}}

\def\ov#1{\overline{#1}}
\def\({\left(}
\def\){\right)}

\def\presuper#1#2%
  {\mathop{}%
   \mathopen{\vphantom{#2}}^{#1}%
   \kern-\scriptspace%
   #2}
\makeatletter
\tikzstyle{notestyleraw}=[
  draw=\@todonotes@currentbordercolor,
  fill=\@todonotes@currentbackgroundcolor,
  line width=0.5pt,
  text width=\@todonotes@textwidth+9ex,
  inner sep=0.8ex,
  rounded corners=4pt
]
\makeatother

\newcommand{\CoefField}{\mathbf{C}}

\def\C{{\rm C}}
\newcommand{\ti}[1]{\tilde{#1}}
\newcommand{\trd}{\text{trd}}
\newcommand{\tr}{\text{tr}}

\newcommand{\mc}{\mathcal}

\newcommand{\tiG}{\ti{\G}}

\newcommand{\tiH}{\ti{\H}}
\newcommand{\tiC}{\ti{\C}}
\newcommand{\tiP}{\ti{\P}}
\newcommand{\mfg}{\mf{g}}
\newcommand{\timfg}{\ti{\mf{g}}}
\newcommand{\timfb}{\ti{\mf{b}}}
\newcommand{\tipsi}{\ti{\psi}}
\newcommand{\timfa}{\ti{\mf{a}}}

\newcommand{\tiJ}{\ti{\J}}
\newcommand{\tiK}{\ti{\K}}
\newcommand{\disc}{\text{disc}}

\newcommand{\loccit}{{\it loc.cit.}}
\newcommand{\bfI}{{\textbf{I}}}
\newcommand{\bfII}{{\textbf{II}}}
\newcommand{\pD}{\mf{p}_\D} 
\newcommand{\depth}{\text{d}}
\newcommand{\Cay}{\text{Cay}}

\newcommand{\Latt}{\text{Latt}}
\newcommand{\ext}{\text{ext}}
\newcommand{\bext}{\mbox{$\b$-$\ext$}}
\def\ainv{(\bar{\ })}
\newcommand{\into}{\hookrightarrow}
\newcommand{\mult}{\text{mult}}
\newcommand{\gl}{\textbf{gl}}
\newcommand{\simarrow}{\stackrel{\sim}{\longrightarrow}}
\newcommand{\1}{\mathbf{1}}
\newcommand{\Br}{\text{Br}}

\makeatletter
\providecommand*{\cupdot}{%
  \mathbin{%
    \mathpalette\@cupdot{}%
  }%
}
\newcommand*{\@cupdot}[2]{%
  \ooalign{%
    $\m@th#1\cup$\cr
    \hidewidth$\m@th#1\cdot$\hidewidth
  }%
}
\makeatother

\title{Cuspidal irreducible complex or~$l$-modular representations of quaternionic forms of~$p$-adic classical groups for 
odd~$p$}
\date{\today}
\author{Daniel Skodlerack}
\makeatother
\begin{document}

\begin{abstract}
Given a quaternionic form~$\G$ of a~$p$-adic classical group ($p$ odd) we classify all cuspidal irreducible
representations of~$\G$ with coefficients in an algebraically closed field of characteristic different from~$p$. 
We prove two theorems: At first: Every irreducible cuspidal representation of~$\G$ is induced from a cuspidal type, i.e. from a certain irreducible representation
of a compact open subgroup of~$\G$, constructed from a~$\b$-extension and a cuspidal representation of a finite group.
Secondly we show that two intertwining cuspidal types of~$\G$ are up to equivalence conjugate under some element of~$\G$. [11E57][11E95][20G05][22E50]
%
\end{abstract}

\maketitle

\section{Introduction}
This work is the third part in a series of three papers, the first two being~\cite{skodlerack:17-1} and~\cite{skodlerack:20}. 
Let~$\F$ be a non-Archimedean local field with odd residue characteristic~$p$. 
The construction and classification of cuspidal irreducible representation complex or~$l$-modular of the set of rational points~$\bbG(\F)$ of a reductive group~$\bbG$ 
defined over~$\F$ has already been successfully studied for general linear 
groups (complex case: \cite{bushnellKutzko:93} Bushnell--Kutzko,~\cite{secherreStevensIV:08},~\cite{broussousSecherreStevens:12},~\cite{secherreStevensVI:12} Broussous--Secherre--Stevens; modular case:~\cite{vigneras:96} Vigneras,~\cite{minguezSecherre:14} Minguez--Sech\'erre) 
and for~$p$-adic classical groups (\cite{stevens:08} Stevens,~\cite{kurinczukStevens:19} Kurinczuk--Stevens, \cite{kurinczukSkodlerackStevens:20} Kurinczuk--Stevens joint with the author). 
In this paper we are generalizing from~$p$-adic classical groups to their quaternionic forms. 
Let us mention~\cite{YuJK:01} Yu,~\cite{fintzen:18},~\cite{fintzen:19} Fintzen and~\cite{kimJL:07} Kim for results over reductive~$p$-adic 
groups in general. 

We need to introduce notation to describe the result. 
We fix a skew-field~$\D$ of index~$2$ over~$\F$ together with an anti-involution~$(\bar{\ })$ on~$\D$ and 
an~$\epsilon$-hermitian form
\[h:\V\times\V\ra\D\] 
on a finite dimensional~$\D$-vector space~$\V$. 
Let~$\G$ be the group of isometries of~$h$. Then~$\G$ is the set of rational points of the connected reductive group~$\bbG$ defined by~$h$ and~$\Nrd=1$, see \S\ref{subsecSemiChar}.
Let~$\CoefField$ be an algebraically closed field of characteristic~$p_\CoefField$ different from $p$. We only consider smooth representations with coefficients in~$\CoefField$. 
At first we describe the construction of the cuspidal types (imitating the Bushnell--Kutzko--Stevens framework): 
A cuspidal type is a certain irreducible representation~$\lambda$ of a certain compact open subgroup~$\J$ of~$\G.$
The arithmetic core of~$\lambda$ is given by a skew-semisimple stratum~$\Delta=[\La,n,0,\b]$. It provides the following data (see~\cite{skodlerack:20} 
for more information): 
\begin{itemize}
 \item An element~$\b$ of the Lie algebra of~$\G$ which generates over~$\F$ a product~$\E$ of fields in~$A:=\End_\D\V$. We denote the centralizer of~$\b$ in~$\G$ by~$\G_\b$. 
 \item A self-dual~$\o_\E$-$\o_\D$-lattice sequence~$\La$ of~$\V$ which can be interpreted as a point of the Bruhat-Tits building~$\mf{B}(\G)$ and as the image of a 
 point~$\La_\b$ of the Bruhat-Tits building~$\mf{B}(\G_\b)$ under a canonical map (see~\cite{skodlerack:13})
 \[j_\b:\mf{B}(\G_\b)\ra\mf{B}(\G).\] . 
 \item An integer~$n>0$ which is related to the depth of the stratum. 
 \item Compact open  subgroups of~$\G$:~$\H^1(\b,\La)\subseteq\J^1(\b,\La)\subseteq\J(\b,\La)$, here abbreviated by~$\H^1,\J^1$ and~$\J$. 
 \item A set~$\C(\Delta)$ of characters of~$\H^1$. (so-called self-dual semisimple characters)
\end{itemize}
The representation~$\lambda$ consists of two parts: 

Part 1 is the arithmetic part: One chooses a self-dual semisimple character~$\theta\in\C(\Delta)$, which admits a 
Heisenberg representation~$\eta$ on~$\J^1$ (see~\cite[\S8]{bushnellFroehlich:83} for these extensions) and then
constructs a certain extension~$\kappa$ of~$\eta$  to~$\J$. ($\kappa$ having the same degree as~$\eta$)
Not every extension is allowed for~$\kappa$. For example if~$\La_\b$ corresponds to a vertex in~$\mf{B}(\G_\b)$ 
(which is the case for cuspidal types) we impose that the restriction of~$\kappa$ to a pro-$p$-Sylow subgroup of~$\J$ is intertwined by~$\G_\b$. 

Part 2 is a representation of a finite group (This is the so called level zero part). Let~$\k_\F$ be the residue field of~$\F$. 
The group~$\J/\J^1$ is the set of~$\k_\F$-rational point of a reductive group, here denoted by~$\bbP(\La_\b)$. 
It is also the reductive quotient of the stabilizer~$\P(\La_\b)$ of~$\La_\b$ in~$\G_\b$. The pre-image~$\P^0(\La_\b)$  of~$\bbP(\La_\b)^0(\k_\F)$
(connected component) in~$\P(\La_\b)$ is the parahoric subgroup of~$\G_\b$ corresponding to~$\La_\b$. 
We choose an irreducible representation~$\rho$ of~$\bbP(\La_\b)(\k_\F)$ whose restriction to~$\bbP(\La_\b)^0(\k_\F)$ is a direct sum of 
cuspidal irreducible representations, and we inflate~$\rho$ to~$\J$, still called~$\rho$, and define~$\lambda:=\kappa\otimes\rho$. 
Then~$\lambda$ is called a cuspidal type if~$\P^0(\La_\b)$ is a maximal parahoric subgroup in~$\G_\b$. 
(see. \S\ref{secCuspType})

Then, we obtain the following classification theorem: 
\begin{theorem}[Main Theorem]\label{thmMain}
 \begin{enumerate}
  \item Every irreducible cuspidal $\CoefField$-representation of $\G$ is induced by a cuspidal type. (Theorem~\ref{thmExhaustionG})\label{thmMain-i}
  \item If~$(\lambda,\J)$ is a cuspidal type, then~$\ind_\J^\G\lambda$ is irreducible cuspidal.~\label{thmMain-ii} (Theorem~\ref{thmCuspType})
  \item Two intertwining cuspidal types~$(\lambda,\J)$ and~$(\lambda',\J')$ are up to equivalence conjugate in~$\G$ if and only if they intertwine in~$\G$. (Theorem~\ref{thmIntConG})\label{thmMain-iii}
 \end{enumerate}
\end{theorem}

We need to consider representations~$(\lambda,\J)$ for lattice sequences~$\La$ such that~$\La_\b$ does not correspond to a vertex of~$\mf{B}(\G_\b)$, because the proof of Theorem~\ref{thmMain}\ref{thmMain-i} is done by contradiction in proving:

\begin{proposition}\label{propExhaustionIndirect}
 Let~$\pi$ be a cuspidal irreducible representation of~$\G$ containing~$\theta$. Then~$\P^0(\Lambda_\b)$ is a maximal parahoric subgroup of~$\G_\b$.  
\end{proposition}

This is analogous to the non-quaternionic case, see~\cite[Theorem 12.2]{kurinczukStevens:19} and~\cite[Appendix A]{MiSt}. Let us recall the outline of the construction of~$\b$-extensions for the non-vertex case, see~\cite[\S 4]{stevens:08}, because it is important for what follows in the introduction. Given~$\theta\in\C(\Delta)$ we choose a lattice sequence~$\La_\M$ such that~$(\La_\M)_\b$ corresponds to a vertex of the facet containing~$\La_\b$. Then we choose a path
\[(\La_\M)_\b=\La_\b^{(0)},\La_\b^{(1)},\ldots,\La_\b^{(t)}=\La_\b\]
in the closure of the facet of~$\La_\b$ such that~$\La^{(i)}$ and~$\La^{(i+1)}$ are close enough, i.e. the closure of the facet in~$\mf{B}(\G)$ of one of them contains both lattice sequences. A~$\b$-extension of~$\th$ with respect to~$\La_\M$  is constructed by a sequence of irreducible representations 
\begin{equation}\label{eqKappaSeq}
 \kappa^{(0)}, \kappa^{(1)},\ldots, \kappa^{(t)}=\kappa
\end{equation}
 such that~$\kappa^{(i)}$ is attached to~$\La^{(i)}$,~$\kappa^{(0)}$ is a restriction of a~$\b$-extension~$\kappa_\M\in\b-\ext(\La_\M)$ (of the transfer of~$\theta$),~$\kappa^{(i)}$ and~$\kappa^{(i+1)}$ satisfy a compatibility condition, see Lemma~\ref{lemMutMutStev4.3}, and such that~$\kappa$ is a representation of~$\J(\b,\La)$. Representations~$\kappa$ constructed this way are called~$\b$-extension of~$\th$ on~$\J(\b,\La)$ with respect to~$\La_\M$. See~\S\ref{secBetaExt} for details.
 
The proof of the theorem needs several steps. We need a quadratic unramified field extension~$\L|\F$ and~$\G_\L:=\G\otimes\L$ with its Bruhat-Tits 
building~$\mf{B}(\G_\L)$, and further the building~$\mf{B}(\tiG_\F)$ of the general linear group~$\tiG_\F=\Aut_\F(\V)$. 

Step 1: At first we show that every irreducible representation of~$\G$ contains a self-dual semisimple character. (This is the most difficult part of the theory.),
 see Theorem~\ref{thmStDuke5.1ForQuatCase}. Mainly we use the canonical embeddings
 \[\mf{B}(\G)\hookrightarrow\mf{B}(\G_\L)\hookrightarrow\mf{B}(\tiG_\F)\]
 together with unramified, here~$\Gal(\L|\F)$, and ramified, here~$\Gal(\F[\varpi_\D]|\F)$,  Galois restriction to results of~\cite{stevens:05} and~\cite{stevens:02}. (cf.~\cite[\S8.9]{dat:09})

 Step 2: (The proof of Theorem~\ref{thmMain}\ref{thmMain-ii}) We show that the set of self-intertwininers of~$\lambda$ is equal to~$\J$. This is done using Morris theory, analogous to~\cite[Proposition 6.18]{stevens:08} (without using~\cite[Corollary 6.16]{stevens:08}) and an irreducibility criterion for the modular case
 as done in~\cite[Theorem 12.1]{kurinczukStevens:19}. See~\S\ref{secCuspType} and Theorem~\ref{thmCuspType}.
 
 Step 3: (The proof of Theorem~\ref{thmMain}\ref{thmMain-i}) 
 We prove that a self-dual semisimple character contained in a cuspidal irreducible representation needs to be skew, see \S\ref{secSkew}. 
 Then we prove Proposition~\ref{propExhaustionIndirect}, see~\S\ref{secExhaust}: Starting with a cuspidal irreducibility representation~$\pi$ of~$\G$, it must contain a skew-semisimple character~$\theta\in\C(\Delta)$, for some~$\Delta$, and therefore some representation~$\lambda=\kappa\otimes\rho$ as above. An argument using covers, see~\cite{bushnellKutzko:98}, similar to~\cite[\S7]{stevens:08} shows that~$\P^0(\La_\b)$ is maximal parahoric. More precisely, assuming that~$\P^0(\La_\b)$ is not a maximal 
 parahoric, 
 we find a parabolic subgroup~$\P$ with a Levi~$\M$ and an irreducible representation~$(\lambda_\P,\J_\P)$ ($\J_\P\subseteq\J)$ which induces to~$\lambda$ such that there is a proper Levi~$\M'$ of~$\G$ such that~$\M'\supseteq\M$ and such that~$\lambda_\P$ is a cover of~$(\lambda_\P)|_{\M'}$ (in the sense of strongly positive elements of~$\M'$.).
 
Here we needed to generalize the notion of subordinate decompositions, see \S\ref{secIwahori} to get the parabolic subgroup~$\P$. To prove that~$\lambda_\P$ is a cover of~$(\lambda_\P)_{\M'}$ we need a bound for the set  of self-intertwiners of~$\lambda$, see~\cite[Corollary 6.16]{stevens:08}. 
The proofs in~\cite[\S 6]{stevens:08} do not work anymore for the quaternionic case if we do not use a finer choice of standard~$\b$-extensions. This is the main difference between the quaternionic and the non-quaternionic case. A level zero~$\b$-extension for~$\GL_m(\D)$ on~$\GL_m(\o_\D)$ does not need to be intertwined by~$\D^\times$, where~$\D^\times$ is embedded diagonally into~$\GL_m(\D)$, see Remark~\ref{remtiKappanotnormalizedbyDtimes}. 
The choice of standard~$\b$-extensions is as follows: Suppose we are given~$\th\in\C(\Delta)$ such that~$\La_\b$ does not correspond to a vertex in~$\mf{B}(\G_\b)$. We choose a vertex~$x$ of~$\La_\b$ such that its stabilizer in~$\G_\b$ contains enough Weyl-group elements. We consider~$\b$-extensions~$(\kappa,\J)$ with respect to~$x$. This is analogue to the non-quaternionic case. But further in the quaternionic case we have to impose an extra condition on the chosen~$\b$-extension~$\kappa$, more precisely that the order of the determinant of~$\kappa$ divides~$2p^s$ for some non-negative  integer~$s$. We denote this property by~\textbf{(ORD)}. We say that~$\kappa$ is standard if in~\eqref{eqKappaSeq} the representation~$\kappa_\M$ satisfies~\textbf{(ORD)}. This implies that~$\kappa$ also satisfies~\textbf{(ORD)}. See~\S\ref{secStandardbetaExt}. 

Step 4: For the intertwining implies conjugacy part of Theorem~\ref{thmMain} we use~\cite{skodlerack:20} and~\cite{kurinczukStevens:19} and follow~\cite[\S11]{kurinczukSkodlerackStevens:20}, 
 see \S\ref{secConjCuspTypes}.

In the appendix, see~\S\ref{appErratum4p2}, we have added an erratum on a proposition in~\cite{stevens:02} which was used to show in~\loccit, for~$p$-adic classical groups~$\G'$,  that the coset of any non-$\G'$-split fundamental stratum is  contained in the coset of a skew-semisimple stratum. This was necessary because main statements in~\loccit\ are used in the proof of existence of semisimple characters in irreducible representations.    
The proofs in the erratum were written by S. Stevens, the author of~\cite{stevens:08},
 in 2012, but not published yet.
 
In Appendix~\ref{appMaxSelfdualOrderNonParahoric} we prove a Lemma which is very important for the exhaustion part Theorem~\ref{thmMain}\ref{thmMain-i}. It roughly says, that if in the search of a type in an irreducible representation of~$\G$ with maximal parahoric, one has landed at a vertex (in the weak simplicial structure of~$\B(\G_\b)$) which does not support a maximal parahoric subgroup of~$\G_\b$, then one can move along an edge to resume the search. This idea is disguised in~\cite[\S7]{stevens:08} and~\cite[Appendix]{MiSt}, so that we found that it is worth to give a proof of this lemma, see Lemma~\ref{lemCompNonParahoric}, in lattice theoretic terms. 

This work was supported by a two month stay at the University of East Anglia in summer 2018 and afterwards by my position at ShanghaiTech University.

Daniel Skodlerack\\
Institute of Mathematical Sciences, \\
ShanghaiTech University\\
No. 393 Huaxia Middle Road,\\
Pudong New Area, \\
Shanghai 201210\\
www.skodleracks.co.uk\\
http://ims.shanghaitech.edu.cn/ \\
email: dskodlerack@shanghaitech.edu.cn

\section{Notation}
\subsection{Semisimple characters}\label{subsecSemiChar}
This article is a continuation of~\cite{skodlerack:20} and~\cite{skodlerack:17-1} which we call \bfI\ and \bfII. We mainly follow their notation, but there is a major change, see the remark below, and there are slight changes to adapt the notation to~\cite{stevens:08}. 
Let~$\F$ be a non-Archimedean local field of odd residue characteristic~$p$ with valuation~$\nu_\F: \F\ra\bbZ$, valuation ring~$\o_\F$, valuation ideal~$\mf{p}_\F$, residue field~$\k_\F$ and we fix a uniformizer~$\varpi_\F$ of~$\F$. 
We fix an additive character~$\psi_\F$ of~$\F$ of level~$1$. 
We consider a quaternionic form~$\G$ of a~$p$-adic classical group as in~\bfII, i.e.~$\G=\U(h)$ for an~$\epsilon$-hermitian form
\[h:\V\times \V\ra (\D,\bar{(\ )}),\]
where~$\D$ is a skew-field of index~$2$ and central over~$\F$ together with an anti-involution~$\bar{(\ )}:\D\rightarrow\D$ of~$\D$. The form~$h$ defines via its adjoint anti-involution an algebraic group~$\bbU(h)$ defined over~$\F$. We denote with~$\bbG$ the unital component of~$\bbU(h)$, given by the additional equation~$\Nrd=1$. 
 By~\bfII.2.9 (see~\cite[1.III.1]{moeglinVignerasWaldspurger:87})
the sets~$\bbG(\F)$ and~$\bbU(h)(\F)$ coincide with~$\G$, and we will consider~$\bbG$ as the algebraic  group associated to~$\G$.
The ambient general linear group for~$\G$:~$\Aut_\D(\V)$, is denoted by~$\tiG$. 
Let us recall that a stratum has 
 the standard notation~$\Delta=[\Lambda,n,r,\beta]$, i.e. the entries for~$\Delta'$ are~$\Lambda',n',r',\beta'$ and for~$\Delta_i$ are~$\Lambda^i,n_i,r_i$ and~$\beta_i$. 
A semisimple stratum has a unique coarsest decomposition as a direct sum of  simple strata: $\Delta=\oplus_{i\in I}\Delta_i,$
in particular it decomposes~$\E=\F[\beta]$ into a product of fields~$\E_i=\F[\beta_i]$, provides idempotents via~$1=\sum_i 1^i$ and further decompositions
\[\V=\oplus_{i\in \I}\V^i,\ \A=\End_\D(\V)=\oplus_{i,j\in \I}\A^{ij},\ \A^{ij}:=\Hom_\D(\V^i,\V^j).\]
The form~$h$ comes along with an adjoint anti-involution~$\sigma_h$ on~$\A$ and an adjoint involution~$\sigma$ on~$\tilde{\G}$, defined via~$\sigma(g):=\sigma_h(g)^{-1}$.
We denote by~$\C_?(!)$ the centralizer of~$!$ in~$?$,~$\B:=\C_\A(\beta)$ decomposes into~$\B=\oplus_i \B^i$,~$\B^i=\C_{\A^{ii}}(\beta_i)$. 
We write~$\tiG_i$ for~$(\A^{ii})^\times$. 
The adjoint anti-involution~$\sigma_h$ of~$h$ induces a map on the set of strata~$\Delta\mapsto \Delta^\#$.~$\Delta$ is called self-dual if~$\Delta$ and~$\Delta^\#$ coincide up to a translation of~$\Lambda$, i.e.~$n=n^\#,r=r^\#,\beta=\beta^\#$, and there is an integer~$k$ such that~$\Lambda-k$, which is~$(\Lambda_{j+k})_{j\in\bbZ}$, is equal to~$\Lambda^\#$, i.e.~$\Lambda$ is~\emph{self-dual}. \daniel{A self-dual~$o_\D$-lattice sequence is called~\emph{standard self-dual} if the~$o_\D$-period~$e(\Lambda|o_\D)$ is even and~$\Lambda^\#(z)=\Lambda(1-z)$ for all integer~$z$. Further, in the self-dual semisimple case, the anti-involution~$\sigma_h$ induces an action of~$\langle \sigma_h\rangle$ on the index set~$\I$ of the stratum, and decomposes it as~$\I=\I_+\cupdot\I_0\cupdot\I_-$, with the fixed point set~$\I_0$ and a section~$\I_+$ through all the orbits of length~$2$. We write~$\I_{0,+}$ for~$\I_0\cup\I_+$.}
To a semisimple stratum~$\Delta$ is attached a compact open subgroup~$\tiH(\Delta)$ of~$\tiG$ and a finite set of complex characters~$\tiC(\Delta)$ defined on~$\tiH(\Delta)$. If~$\Delta$ is self-dual semisimple we define~$\C(\Delta)$ as the set of the restriction of the elements of~$\tiC(\Delta)$ to~$\H(\Delta):=\tiH(\Delta)\cap \G$.  
Given a stratum~$\Delta$ we denote by~$\Delta(j-)$ the stratum~$[\Lambda,n,r-j,\beta]$, if~$n\geq r-j\geq 0$, for~$j\in\bbZ$ and analogously we 
have~$\Delta(j+)$. There is a major change of notation to \bfI\ and~\bfII: 
\begin{remark}
 We make the following convention for the notation. (Caution this is then different form the notation in~\bfI\ and~\bfII.) 
 Every object which corresponds to the general linear group~$\tiG$ is going to get a~$\tilde{(\ )}$ on top. Instead of~$\C_-(\Delta)$ in \bfII\ we write~$\C(\Delta)$, and instead 
 of~$\C(\Delta)$ in~\bfI\ we write~$\tiC(\Delta)$. Analogously for the groups and characters etc.. 
\end{remark}

\subsection{Coefficients for the smooth representations}
In this paper we only consider smooth representations of locally compact groups~$\H$ on~$\CoefField$-vector spaces, where~$\CoefField$ is an algebraically closed field whose characteristic, denoted by~$p_\CoefField$, is different from~$p$. \daniel{We write~$\mf{R}_\CoefField(\H)$ or~$\mf{R}(\H)$ for the category of those representations.}

The theory of semisimple characters in~\bfI,~\bfII~and~\cite{stevens:05}, see also \S\ref{subsecSemiChar}, is still valid for~$\CoefField$, because~$\CoefField$ contains a full set~$\mu_{p^\infty}(\CoefField)$ of~$p$-power roots of unity. We fix a group isomorphism 
$\phi$ from~$\mu_{p^\infty}(\CC)$ to~$\mu_{p^\infty}(\CoefField)$ and define
\[\ti{\C}^\CoefField(\Delta):=\{\phi\circ\theta |\ \theta\in\ti{\C}(\Delta)\},\]
and analogously~$\C^\CoefField(\Delta)$ for strata equivalent to self-dual semisimple strata and with respect to~$\psi_\F^\CoefField=\phi\circ\psi_\F$.
We identify~$\ti{\C}^\CoefField(\Delta)$ and~$\C^\CoefField(\Delta)$ with~$\ti{\C}(\Delta)$ and~$\C(\Delta)$, resp., and skip the superscript,  because from now on we only consider~$\CoefField$-valued (self-dual) semisimple characters. We are going to apply the results of~\bfI,~\bfII~and~\cite{stevens:05} to~$\CoefField$-valued (self-dual) semisimple characters without further remark.

\subsection{Buildings and Moy--Prasad filtrations}\label{subsecBuilding}
In this section we recall the description of Bruhat--Tits building of~$\tiG$ and~$\G$ and its Moy--Prasad filtrations in terms of lattice functions. 

To~$\G$ and~$\tiG$ are attached Bruhat--Tits buildings~$\mf{B}(\G)$,~$\mf{B}(\tiG)$ and~$\mf{B}_{red}(\tiG)$, see~\cite{bruhatTitsIII:84} and~\cite{bruhatTitsIV:87}. Important for the study of smooth representation, for example for the concept of depth, are the following filtrations, constructed by Moy--Prasad (\cite{moyPrasad:94},\cite{moyPrasad:96}): Let~$x$ be a point of~$\tiG$ and~$y$ be a point of~$\B(\G)$. They carry 
\begin{itemize}
 \item a filtration of the Lie algebra with~$\o_\F$-modules:
 \begin{itemize}
  \item $(\mf{g}_{y,t})_{t\in\bbR}$,~$\mf{g}_{y,t}\subseteq\Lie(\G)$ 
  \item $(\tilde{\mf{g}}_{x,t})_{t\in\bbR}$,~$\tilde{\mf{g}}_{x,t}\subseteq\Lie(\tiG)$ 
 \end{itemize}
 and 
 \item a filtration of subgroups~$(\G_{y,t})_{t\geq 0}$,~$(\tiG_{x,t})_{t\geq 0}$ of~$\G$ and~$\tiG$, respectively. 
\end{itemize}

Those can be entirely described using lattice functions. We refer to~\cite{broussousLemaire:02},~\cite{broussousStevens:09} and~\cite[\S3.1]{skodlerack:17-1} for lattice functions and lattice sequences. 

\begin{definition}
 A family~$\Gamma=(\Gamma(t))_{t\in\bbR}$ of full~$\o_\D$-lattices of~$\V$ is called an~$\o_\D$-lattice function if for all real numbers~$t<s$ we have
 \begin{itemize}
  \item~$\Gamma(t)$ is a full~$\o_\D$-lattice in~$V$,
  \item~$\Gamma(t)=\bigcap_{u<t}\Gamma(u)$
  \item~$\Gamma(t)\w_\D=\Gamma(t+\frac{1}{d})$,
 \end{itemize}
where~$d$ is the index of~$\D$ (In our case of a non-split quaternion algebra~$d=2$.).

\end{definition}

We further define~$\Gamma(t+):=\bigcup_{u>t}\Gamma(u)$, and we define the set of discontinuity points:
\[\disc(\Gamma):=\{s\in\bbR|\ \Gamma(s)\neq\Gamma(s+)\}.\]
We can translate~$\Gamma$ by a real number~$s$:~$(\Gamma-s)(t):=\Gamma(t+s)$, and the set of all real translates of~$\Gamma$ is called the translation
class of~$\Gamma$. We denote this class by~$[\Gamma]$. The set of all~$\o_\D$-lattice functions in~$\V$, resp. translation classes of those, is denoted by~$\Latt_{\o_\D}^1\V$, resp.~$\Latt_{\o_\D}\V$, see~\cite{broussousLemaire:02}. 

An~$\o_\D$-lattice function~$\Gamma$ with~$\disc(\Gamma)\subseteq \mathbb{Q}$
corresponds to a lattice sequence~$\Lambda_\Gamma$ in the following way: There exists a minimal positive integer~$e$, such that~$\disc(\Gamma)$ is contained in~$q+\frac{1}{e}\ZZ$
for some~$q\in\bbQ$. Then define:
\[\Lambda_\Gamma(z):=\Gamma\left(\frac{z}{e}\right),\ z\in\ZZ.\]
Conversely we can attach a lattice function to an~$\o_\D$-lattice sequence~$\Lambda$. 
\daniel{Recall that~$\lceil t\rceil$ denotes the smallest integer not smaller than~$t$.} We define:
\begin{equation}\label{eqGammaLambda}\Gamma_\Lambda(t):=\Lambda\left(\left\lceil t e(\Lambda|\F) \right\rceil\right),\ t\in\bbR,\end{equation}
                                                                                                                               
where~$e(\Lambda|\F)$ is the~$\F$-period of~$\Lambda$. 

We fix an~$\o_\D$-lattice function~$\Gamma$ and an~$\o_\D$-lattice sequence~$\Lambda$. 
A lattice sequence~$\Lambda'$ is called an~\emph{affine translation} of~$\Lambda$ if there are a positive integer~$a$ and an integer~$b$
such that
\[\Lambda'(z):=\Lambda\left(\left\lfloor\frac{z-b}{a}\right\rfloor\right),\]
and we denote~$\Lambda'$ by~$(a\Lambda+b)$. If~$a=1$ then we call ~$\Lambda'$ just a~\emph{translation} of~$\Lambda$ and we write $[\Lambda]$ for the translation class. Two lattice sequences $\Lambda$ and~$\Lambda'$ are said to be in the same \emph{affine class} if both have coinciding affine translations.

Then~$\Gamma$ and~$\Gamma_{\Lambda_\Gamma}$ are translates of each other and~$\Lambda$ and~$\Lambda_{\Gamma_\Lambda}$ are in the same affine class. 

The invariant of the translation class of~$\Gamma$ is the~\emph{square lattice function}:
\[t\in\bbR\mapsto\tilde{\mf{a}}_t(\Gamma):=\{a\in\A|\ a\Gamma(s)\subseteq\Gamma(s+t),\ \text{for all }s\in\bbR\},\ \timfa(\Gamma):=\timfa_0(\Gamma),\]
and analogously we have~$(\tilde{\mf{a}}_z(\Lambda))_{z\in\bbZ}~\text{and}~\timfa(\Lambda)$ for~$[\Lambda]$. We write~$\Latt^2_{\o_\F}\A$ for the set
\[\{(\tilde{\mf{a}}_t(\Gamma))_{t\in\bbR}|\ \Gamma\in\Latt^1_{\o_\D}\V \},\]
and there are canonical maps:
\begin{equation}\label{eqBL1}
\Latt^1_{\o_\D}\V\rightarrow\Latt_{\o_\D}\V\stackrel{\sim}{\longrightarrow}\Latt^2_{\o_\F}\A.\end{equation}

Note that~$\Latt^1_{\o_\D}\V$ carries an affine structure, see~\cite[\S~I.3]{broussousLemaire:02}. 
The description of~$\mf{B}(\tiG)$ and~$\mf{B}_{red}(\tiG)$ in terms of lattice functions is stated in the following theorem. 
\begin{theorem}[\cite{broussousLemaire:02}~I.1.4,I.2.4,\cite{bruhatTitsIII:84}~2.11,2.13]\label{thmBuilding}
\begin{enumerate}
 \item\label{thmBuildingi} There exists an affine~$\tiG$-equivariant map 
 \[\iota_{\tiG}:\ \mf{B}(\tiG)\rightarrow\Latt_{\o_\D}^1\V.\]
 Further~$\iota_{\tiG}$ is a bijection, and two~$\tiG$-equivariant affine maps~$\iota_1,\iota_2$ differ by a translation, i.e. there is an element~$s\in\bbR$ such that~$\iota_2\circ\iota_1^{-1}$ has the form 
 \[\Gamma\mapsto\Gamma-s.\]
 \item There is a unique~$\tiG$-equivariant affine map
 \[\iota_{\tiG,red}:\ \mf{B}_{red}(\tiG)\rightarrow \Latt_{\o_\D}\V.\]
 We obtain a commutative diagram
 \[
 \begin{array}{ccc}
  \mf{B}(\tiG) & \stackrel{\iota_{\tiG}}{\longrightarrow} & \Latt^{1}_{\o_\D}\V \\
  \downarrow & & \downarrow\\
  \mf{B}_{red}(\tiG) &  \stackrel{\iota_{\tiG,red}}{\longrightarrow} & \Latt_{\o_\D}\V\\
 \end{array}
 \]
 \end{enumerate}
 \end{theorem}
 
We now describe~$\mf{B}(\G)$. For more details refer to~\cite{broussousStevens:09} and~\cite{lemaire:09}.  Recall the dual of a lattice function~$\Gamma$: 
\[\Gamma^\#(t):=\Gamma((-t)+)^{\#}.\]
The lattice function~$\Gamma^\#$ depends on~$h$, because~$\#$ does, but for different~$\epsilon$-hermitian forms~$h_1,h_2$ on~$\V$, with respect to~$(\D,\ainv)$, with common isometry group~$\G$ the respective duals~$\Gamma^{\#_{h_1}}$ and~$\Gamma^{\#_{h_2}}$
just differ by a translation. The lattice function~$\Gamma$ is called~\emph{self-dual} with respect to~$h$ if~$\Gamma^\#=\Gamma$ or equivalently if~$\tilde{\mf{a}}_t(\Gamma)$ is~$\sigma_h$-invariant for every~$t\in\bbR$, called~\emph{self-dual} square lattice function. 
The map in~\eqref{eqBL1} restricts to a canonical bijection between the set of self-dual~$\o_\D$-lattice functions, which we denote by~$\Latt^1_h\V$, and the set of self-dual square lattice functions, denoted by~$\Latt^2_-\A$. 
The latter inherits an affine structure from~$\Latt_{o_\D}\V$. 

The translation class of the lattice function attached to a with respect to~$h$ self-dual lattice sequence~$\Lambda$ contains exactly one self-dual lattice function. We are going to denote this self-dual lattice function by~$\Gamma_{\Lambda,h}$ 
instead of~\eqref{eqGammaLambda}.

The building of~$\G$ is now described as follows in terms of lattice functions. 

\begin{theorem}[\cite{broussousStevens:09}~4.2]\label{thmIotaG}
 There is a unique affine~$\G$-equivariant map
 \[\iota_G:\ \mf{B}(\G)\rightarrow \Latt_{h}^1\V.\]
 Moreover this map is bijective. 
\end{theorem}

Note that the map~$\iota_\G$ depends on~$h$. Given~$(\V,h_1)$ 
and~$(\V,h_2)$ two~$\epsilon$-hermitian spaces w.r.t.~$(\D,\ainv)$
with isometry group~$\G$ then~$\sigma_{h_1}=\sigma_{h_2}$ and we obtain the 
diagram
\[
\begin{array}{ccc}
 \mf{B}(\G) & \overset{\iota_{\G,h_1}}{\longrightarrow} & \Latt^1_{h_1}\V\\
\iota_{\G,h_2}\big\downarrow & & \big\downarrow\rotatebox{90}{\text{$\sim$}}\\
 \Latt_{h_2}^1\V & \stackrel{\sim}{\longrightarrow} & \Latt^2_-\A \\
\end{array}
\]
and the diagram commutes by the uniqueness assertion of Theorem~\ref{thmIotaG}. Therefore
for every~$x\in\mf{B}(\G)$ the lattice functions~$\iota_{\G,h_1}(x)$ and~$\iota_{\G,h_2}(x)$  
are in the same translation class. 

Given~$h$ we embed the building~$\mf{B}(\G)$ into~$\mf{B}(\tiG)$ and~$\mf{B}_{red}(\tiG)$ via the 
following diagram. 

\[
\begin{array}{ccccc}
\mf{B}(\G) &  \longrightarrow  & \mf{B}(\tiG) & & \\
\big\downarrow\rotatebox{90}{\text{$\sim$}}& \rcirclearrowleft & \big\downarrow\rotatebox{90}{\text{$\sim$}} &  & \\
\Latt^1_{h}\V & \longrightarrow & \Latt^1_{\o_\D}\V & \longrightarrow & \Latt_{\o_\D}\V \\
\downarrow\rotatebox{90}{\text{$\sim$}} & \rcirclearrowleft & \downarrow\rotatebox{90}{\text{$\sim$}} & \rcirclearrowleft & \downarrow\rotatebox{90}{\text{$\sim$}} \\
\Latt^2_-\A & \longrightarrow & \Latt^2_{\o_\F}\A & \longrightarrow & \mf{B}_{red}(\tiG) \\
\end{array}
\]

The embedding of~$\mf{B}(\G)$ into~$\mf{B}_{red}(\tiG)$ does not depend on~$h$ contrary to the embedding into~$\mf{B}(\tiG)$. 

We now turn to the description of the Moy--Prasad filtrations. 
At first an~$\o_\D$-lattice functions~$\Gamma$ and an~$\o_\D$-lattice sequence~$\Lambda$ define filtrations of compact open subgroups: 
\[\tilde{\P}_t(\Gamma)=(1+\tilde{\mf{a}}_t(\Gamma))^\times,\ t\geq 0,\]
\[\tilde{\P}_n(\Lambda)=(1+\tilde{\mf{a}}_n(\Lambda))^\times,\ n\in\bbN_0,\]
and if~$\Gamma$ and~$\Lambda$ are self-dual: 
\[\P_t(\Gamma):=\G\cap\tilde{\P}_t(\Gamma),\ \P_n(\Lambda):=\G\cap\tilde{\P}_n(\Lambda),\]
and we need the filtration~$\mf{a}_t(\Gamma):=\tilde{\mf{a}}_t(\Gamma)\cap \A_-$ and analogously~$(\mf{a}_n(\Lambda))_{n\in\bbN}$. 
The relation of those filtrations to the Moy--Prasad filtrations is: 

\begin{theorem}[\cite{lemaire:09}]
If we identify the Lie algebra of~$\tiG$ with~$\A$ and the Lie algebra of~$\G$ with~$\A_-$, then we have for all non-negative real numbers~$t$ \daniel{and points}~$x\in\mf{B}(\tiG)$ and~$y\in\mf{B}(\G)$: 
\begin{enumerate}
 \item $\tiG_{x,t}=\tiP_t(\iota_{\tiG}(x))$, $\tilde{\mf{g}}_{x,t}=\tilde{\mf{a}}_t(\iota_{\tiG}(x))$ and
 \item  $\G_{y,t}=\P_t(\iota_{\G}(y))$, $\mf{g}_{y,t}=\mf{a}_t(\iota_{\G}(y))$. 
\end{enumerate}
\end{theorem}

As usual we skipt the subscript zero and write~$\P(),\ \tiP()$ for~$\P_0(),\ \tiP_0()$. 

Finally we need the description of the parahoric subgroups in terms of lattice functions/sequences. 
For~$\tiG$ those are the (full) stabilizers~$\tiP(\Gamma),\ \tiP(\Lambda)$. For the classical group~$\G$ the stabilizers are in general too large. 
The parahoric subgroup~$\P^0(\Lambda)$ of~$\G$ is constructed as follows: The quotient~$\P(\La)/\P_1(\Lambda)$ is the set of~$\k_\F$-rational points of a reductive group~$\bbP(\La)$ defined over~$\k_\F$. Let~$\bbP^0(\La)$ be the unital component of~$\bbP(\La)$. 
The pre-image of~$\bbP^0(\La)(\k_\F)$ in~$\P(\La)$ is the parahoric subgroup~$\P^0(\Lambda)$ of~$\G$ defined by~$\Lambda$.
Similar we have~$\P^0(\Gamma)$ using~$\P(\Gamma)/\P_{0+}(\Gamma)$.

We need a finer set of lattice functions/sequences for the study of strata. 
\begin{definition}[\cite{broussousLemaire:02},\cite{skodlerack:14}~7.1,\cite{skodlerack:17-1}~3.6]
 Let~$\b$ be an element of~$\End_\D\V$ and suppose~$\E=\F[\b]$ is a product of fields~$\E=\prod_{i\in\I}\E_i$ with associated splitting~$\V=\oplus_{i\in\I}\V^i$. An~$\o_\D$-lattice function~$\Gamma$ of~$\V$ is called~\emph{$\o_\E-\o_\D$-lattice function} if~$\Gamma$ is split by~$(\V^i)_{i\in\I}$ and~$\Gamma^i=\Gamma\cap\V^i$ is an~$\o_{\E_i}$-lattice function in~$\V^i$,~$i\in\I$. Similarly we define~$\o_\E-\o_\D$-lattice sequences. 
\end{definition}

We use~$\Latt^1_{\o_\E,\o_\D}\V$ to denote the set of~$\o_\E-\o_D$-lattice functions and we set
\[\Latt^1_{h,\o_\E,\o_\D}\V:=\Latt^1_{\o_\E,\o_\D}\V\cap\Latt^1_h\V.\]

We have two simplicial structures on~$\mf{B}(\G)$. The first one, the~\emph{strong} structure, is given by the branching of~$\mf{B}(\G)$, i.e. two points~$x,y\in\mf{B}(\G)$ are said to lie in the same facet if, for all apartments~$\mathcal{A}$ of~$\mf{B}(\G)$, we have
\[x\in\mathcal{A} \text{ iff. }y\in\mathcal{A}, \]
see~\cite{corvallisTits:79}.
The second simplicial structure, called~\emph{weak}, is given by intersection of facets of~$\mf{B}_{red}(\tiG)$ 
to~$\mf{B}(\G)$ via the canonical embedding of~$\mf{B}(\G)$ into~$\mf{B}_{red}(\tiG)$, see~\cite{abramenkoNebe:02}.
The strong facets are unions of weak facets and one obtains the strong structure if one removes the thin panels from the weak structure, see oriflame construction in~\cite{abramenkoNebe:02}. 

\subsection{Centralizer}\label{subsecCentralizer}
Let~$\Delta$ be a semisimple stratum and~$\tiG_\b$ be the centralizer of~$\b$ in~$\tiG$. 
\subsubsection{The self-dual case}
Assume further that~$\Delta$ is self-dual semisimple and write~$\G_\b$ for the centralizer of~$\b$ in~$\G$. 
The stratum provides a pair~($\b,\Lambda$) consisting of an element~$\b$ of the Lie algebra of~$\G$ (which generates a product~$\E$ of
field extensions of~$\F$) and an~$\o_\E-\o_\D$-lattice sequence~$\Lambda$. We need to attach to them a point 
of~$\mf{B}(\G_\b)$ ($\cong\prod_{i\in\I_{0,+}}\mf{B}(\G_{\b}^i)$,~$\G^i$ the image of the projection of~$\G$ to~$\Aut_\D(\V^i)$), and interpret this point as a tuple of 
lattice sequences~$\Lambda^i_\b,\ i\in\I$. The important requirement the tuple has to satisfy is 
the compatibility with the Lie algebra filtrations (abv. \textbf{CLF}, cf. \cite[\S6]{skodlerack:13}):
\begin{equation}\label{eqCLFG}
\mf{a}_z(\Lambda)\cap\End_{\E\otimes_\F\D}(\V)=(\oplus_{i\in\I_0}\mf{a}_z(\Lambda^i_\b))\oplus (\oplus_{i\in\I_+}\ti{\mf{a}}_z(\Lambda^i_\b)),\ z\in\bbZ.
\end{equation}
The CLF-property is meant with respect to the canonical embedding of Lie algebras
$\prod_{i\in\I_{0,+}}\End_{\E_i\otimes\D}\V^i\longrightarrow\End_\D\V$ (no~$\I_-$, and canonical with respect to~$h$.) 
We also need the CLF-property for general linear groups:
\begin{equation}\label{eqCLFGLn}
\tilde{\mf{a}}_z(\Lambda)\cap\End_{\E\otimes_\F\D}(\V)=\oplus_{i\in\I}\tilde{\mf{a}}_z(\Lambda^i_\b),\ z\in\bbZ.
\end{equation}
The construction of~$(\Lambda^i_\b)_{i\in\I}$ is done in several steps: 
\begin{remark}\label{remLbeta}We interpret the buildings~$\mf{B}(\G_\b)$ and~$\mf{B}(\G)$ in terms of lattice functions using \S\ref{subsecBuilding}. Let~$\Gamma_{\Lambda,h}$ be the self-dual lattice function attached to~$\Lambda$. 
\begin{enumerate}
\item
By~\cite[Theorem~7.2]{skodlerack:13} there exists a~$\G_\b$-equivariant, affine, injective CLF-map
\[j_\b:\mf{B}(\G_\b)\hookrightarrow\mf{B}(\G)\]
whose image in terms of lattice functions is~$\Latt^1_{h,\o_\E,\o_\D}\V$, in particular it contains~$\Gamma_{\Lambda,h}$.
\item\label{remLbetai} We define for~$i\in\I$ the skewfields: 
 \[\D^i_\b:=\left\{\begin{matrix}
		    \E_i\otimes\D, & \text{if~$\E_i$ has odd degree over~$\F$.}\\
                    \E_i, & \text{else}
                   \end{matrix}\right.
.\]
and there is a right-$\D^i_\b$-vector space~$\V_\b^i$ such that $\End_{\D^i_\b}(\V^i_\b)$ is~$\E_i$-algebra isomorphic to~$\End_{\E_i\otimes\D}(\V^i)$,
and further we can find for every~$i\in\I_0$ an~$\epsilon$-herrmitian-$\D^i_\b$-form~$h^i_\b$ on~$\V_\b^i$ such that its adjoint anti-involution~$\sigma_{h^i_\b}$
coincides with the pullback of the restriction of~$\sigma_h$. 
The construction of~$j_\b$ in~\loccit\ (which mainly uses~\cite[II.1.1.]{broussousLemaire:02}) provides a tuple~$(\Gamma^i_\b)_{i\in\I_ {0.+}}$ 
of~$\o_{\D^i_\b}$-lattice functions, such that~$j_\b((\Gamma^i_\b)_{i\in\I_ {0.+}})=\Gamma_{\Lambda,h}$. We fix a map~\cite[II.3.1]{broussousLemaire:02} and attach an~$\o_{\D^i_\b}$-lattice function~$\Gamma^i_\b$ 
to~$\Gamma\cap\V^i$ for~$i\in\I_-$. 
 \item Let~$e$ be the~$\F$-period of~$\Lambda$.
 We define the~$\o_{\D^i_\b}$-lattice sequence~$\Lambda_\b^i$ via~$\Lambda^i_\b(z):=\Gamma^i_\b(\frac{z}{e}).$
\item Both CLF-properties~\eqref{eqCLFG} and~\eqref{eqCLFGLn} are satisfied, see~\cite[Theorem~7.2]{skodlerack:14} and~\cite[Theorem~II.1.1]{broussousLemaire:02}.
\end{enumerate}
\end{remark}
We write~$\Lambda_\b$ for~$(\Lambda^i_\b)_{i\in\I}$ and we are going to write~$\mf{b}(\Lambda),\tilde{\mf{b}}(\Lambda),\mf{b}_z(\Lambda),\tilde{\mf{b}}_z(\Lambda)$ 
for the intersections of $\mf{a}(\Lambda),\tilde{\mf{a}}(\Lambda),\mf{a}_z(\Lambda),\tilde{\mf{a}}_z(\Lambda)$ with~$\B$, and we define~$\tilde{\P}(\Lambda_\b), \P(\Lambda_\b),\tilde{\P}_t(\Lambda_\b),\P_t(\Lambda_\b),$ $t\geq 0$, etc. as the intersection of~$\tilde{\P}(\Lambda), \P(\Lambda),\tilde{\P}_t(\Lambda),\P_t(\Lambda)$, etc.  with~$\tiG_\b$. We denote by~$\P^0(\Lambda_\b)$ the parahoric subgroup of~$\G_\b$ attached to~$\Lambda_\b$,
in particular $\P^0(\Lambda_\b)\cong\prod_{i\in\I_{0,+}}\P^0(\Lambda^i_\b)$.

\subsubsection{The general linear case}
Suppose~$\Delta$ is a simple stratum.
By~\cite[Lemma II.3.1]{broussousLemaire:02} there exists an affine,~$\tG_\beta$-eqivariant injective map between the Bruhat--Tits buildings of~$\tG_\beta$ and~$\tG$
\[j_\beta:\mf{B}(\tG_\beta)\rightarrow\mf{B}(\tG) \]
whose image corresponds to the set of~$\o_\E$-$\o_\D$-lattice functions in~$\V$ and which is compatible with the Lie algebra filtrations.
This is map is unique up to a translation on~$\mf{B}(\tG_\beta)$ by Theorem~\ref{thmBuilding}\ref{thmBuildingi}.
We choose and fix one, see~\cite[Remark 3.5]{skodlerack:17-1}. 
 
\subsection{Base extension}
Consider~$\bbG$ from~\S\ref{subsecSemiChar}.
Let~$\L|\F$ be a quadratic unramified extension of~$\F$. In fact we only will use the carefully chosen extension in~\bfII.2.1. There is a very explicit description of the base extension from~$\F$ to~$\L$ in \S2 of \bfII. Later we are going to reduce statements to the group~$\bbG(\L)$, also 
denoted by~$\G\otimes \L$, in particular this notation includes the condition~$\det=1$. We  will use a canonical injective map~$i_\L$ of~$\mf{B}(\G)$ into~$\mf{B}(\G\otimes \L)$ given by~$\Gamma_{i_\L(x)}=\Gamma_x$. 
and the same for the ambient general linear groups:~$i_\L: \mf{B}(\tiG)\rightarrow \mf{B}(\tiG\otimes \L)$. 
If we work over~$\L$ we give the objects in question the subscript~$\L$, for example we write~$\timfg_{\L,x}$ and~$\mfg_{\L,x}$ for the Moy--Prasad filtration of a point~$x$
in~$\mf{B}(\G\otimes \L)$. 
For the definition of semisimple characters of~$\tiG\otimes\L$ we choose the~$\Gal(\L|\F)=\langle\tau\rangle$-fixed extension~$\psi_\L$ of~$\psi_\F$ given by
$\psi_\L(x):=\psi_\F(\frac12\tr_{\L|\F}(x))$.
There is an action of~$\Gal(\L|\F)$ on~$\mf{B}(\tiG\otimes\L)$. It induces a~$\Gal(\L|\F)$-action on~$\mf{B}_{red}(\tiG\otimes\L)$ and~$\mf{B}(\G\otimes\L)$. 
The action on~$\mf{B}(\tiG\otimes\L)$ is defined by
\[(\tau.\Gamma_\L)(t):=(\Gamma_\L)(t-\frac12)\varpi_\D.\]
Further the~$\Gal(\L|\F)$-action on~$\A\otimes_\F \L=\End_\L\V$ is
defined by the~$\Gal(\L|\F)$-action on the second factor. The latter action coincides with the conjugation with~$\varpi_\D\in\End_\F\V$.

\subsection{Intertwining}
We recall the notions of intertwining. Suppose we are given a smooth representation~$\gamma$ on some compact open subgroup~$\K$ of some totally disconnected
locally compact group~$\H$. For an element~$g\in\H$ we write ~$\I_g(\gamma)$ for~$\Hom_{\K\cap \K^g}(\gamma,\gamma^g)$, 
and we denote by~$\I_\H(\gamma)$ the set of all~$g\in\H$ such that~$\I_g(\gamma)$ is non-zero, and we call~$\I_\H(\gamma)$  the set of~\emph{intertwining elements} 
of~$\gamma$ in~$\H$.

\subsection{Restriction}
We recall that, given locally compact totally disconnected groups~$\G_1$ and~$\G_2$ such that~$\G_2$ is a topological subgroup of~$\G_1$, we denote 
by~$\Res^{\G_1}_{\G_2}$ the functor from~$\mf{R}(\G_1)$ to~$\mf{R}(\G_2)$ given by restriction from~$\G_1$ to~$\G_2$. 

\daniel{
\subsection{Endomorphisms and involutions}\label{subsecEndoInv}
We have the following algebras of endomorphisms:
\[\A=\End_\D\V,\ \A_\F=\End_\F\V,\ \A_\L=\End_\L\V,\]
and they can be interpreted as fixed point sets using involutions. 
We fix a~$\tau$-skew-symmetric element~$l$ of~$\o^\times_\L$ and a uniformizer~$\varpi_\D$ of~$\D$ which normalizes~$\L$ such that the square of~$\varpi_\D$ is~$\varpi_\F$, see~\bfII2.1. 
Both elements define~$\F$-endomorphisms of~$\V$ via
right-multiplication and we still denote these endomorphisms by~$l$ and~$\varpi_\D$. Let~$\tau_l$ and~$\tau_{\varpi_\D}$ be the involutions on~$\A_\F$ given by conjugation with~$l$ and~$\varpi_\D$, respectively. Then both generate a Kleinian four-group~$\langle\tau_l,\tau_{\varpi_\D}\rangle$  whose fixed point set is~$\A$. The averaging function
\[\emptyset_{\langle\tau_l,\tau_{\varpi_\D}\rangle}:\ \A_\F\twoheadrightarrow \A\]
defined via
\[\emptyset_{\langle\tau_l,\tau_{\varpi_\D}\rangle}(x):=\frac{1}{4}(x+\tau_l(x)+\tau_{\varpi_\D}(x)+\tau_{l}(\tau_{\varpi_\D}(x)))\]
is a very useful projector
for reducing proofs from~$\D$ to~$\F$.
}

\subsection{Tame corestriction}
The aim of this paragraph is to show the construction of self-dual tame corestrictions over~$\D$. This should already have been done in~\bfII, and one could leave it as an exercise for the reader, but since it is important in the proof of Theorem~\ref{thmStDuke5.1ForQuatCase} and
to indicate the little difference in the construction between the skew and the self-dual case, we have decided to devote to it a paragraph. It is an almost trivial generalization of~\cite[Proof of 3.31]{stevens:05} and~\cite[4.13]{secherreStevensIV:08}, all based on~\cite[(1.3.4)]{bushnellKutzko:93}.

For the~$\GL$-case tame corestrictions in the semisimple case are already defined in~\bfI4.13. Given a semisimple stratum~$\Delta=[\La,n,r,\b]$ then a map~$s:\A\rightarrow\B$ is called tame corestriction for~$\beta$ if 
under the canonical isomorphism~$\A_\F\simeq\A\otimes_\F\End_\A(\V)$ the map~$s_\F=s\otimes \id_{\End_\A(\V)}$ is a tame corestriction in the sense of~\cite[6.17]{skodlerackStevens:18}.
\begin{definition}
 We call a tame corestriction~$s$~\emph{self-dual} (with respect to~$h$)
 if~$s$ is~$\sigma_h$-equivariant, i.e.~$\sigma_h|_{\B}\circ s\circ\sigma_h=s.$
\end{definition}

A corestriction for~$\b$ always exists by~\cite[4.2.1]{broussous:99},
and we fix one, say~$s$. 
We define~$h_\F=\trd_{\D|\F}\circ h$ whose anti-involution~$\sigma_{h_\F}$ extends~$\sigma_h$ from~$\A$ to~$\A_\F$. Suppose at first~$\Delta$ is simple. Two tame corestrictions~$s_{1,\F}$ and 
$s_{2,\F}$ differ by a multiplication with an element of~$\o_\E^\times$, see~\cite[(1.3.4)]{bushnellKutzko:93}. Then, as~$\o_\E^\times$ is contained in~$\A$, every multiplication of~$s$ by an element~$u$ of~$\o_\E^\times$ provides a tame corestriction. If~$s_\F$ is self-dual, so is~$s$ too, because~$\A$ and~$\End_\A(\V)$ are~$\sigma_{h_\F}$-invariant. We conclude the existence of a self-dual~$s$ from~\loccit\ in taking a~$\sigma_h$-invariant additive character~$\psi_\E$. 
For the semisimple case we choose a corestriction~$s_i:\A^{ii}\rightarrow\B^i$ for every~$i\in\I_0\cup\I_+$ (self-dual for~$i\in\I_0$) and define
\[s_{i}(a):=\sigma_h(s_{-i}(\sigma_h(a))),\ a\in\A^{ii},\ i\in\I_-.\]
Finally, the definition:
\[s(a):=\sum_{i\in\I}s_i(a^{ii}),\ a\in\A,\]
provides a self-dual corestriction.

\subsection{Brauer characters and Glauberman correspondence}
Most of the statements in the construction of cuspidal representations for locally compact totally disconnected classical groups are first  written for complex representations and then transferred to the modular case using the theory of Brauer characters for representations of finite groups. 
For example various extension of semisimple characters to prop-$p$ proups play an important role in the construction of cuspidal representations. Those extension are defined for the complex case as representations which satisfy a certain criterion using induction and restriction. The Brauer map 
for a given pro-$p$-group gives a bijective correspondence between finite dimensional smooth complex and~$l$-modular representations and the Brauer maps for a pro-$p$-group and a compact open subgroup preserve induction and restriction.
Thus if we apply the Brauer map to such an extension of a semisimple character, we obtain an~$l$-modular representation which satisfies a similar induction and restriction criterion. 

In contrast to the rest, in
this section if we say~\emph{character} we mean 
the trace of a representation. 
Let~$\K$ be a finite group. Here in this section we forget about the category structure of~$\mf{R}_\CoefField(\K)$ and consider it just as the set of isomorphism classes of~$\CoefField$-representations of~$\K$. We further assume for this paragraph that~$p_\CoefField$ is positive and does not divide the order of~$\K$.   

The following construction is provided in~\cite[\S2]{navarro:98}. In choosing a maximal ideal containing~$p_\CoefField$ in the ring of algebraic integers they fix a group isomorphism between the groups of roots of unity of order prime to~$p_\CoefField$
\begin{equation} ()^*:\mu_{p_\CoefField\not\mid}(\bbC)\simarrow\mu_{p_\CoefField\not\mid}(\CoefField),
\label{eqRootOfUnity}\end{equation}
and extend this map to the additive closure (in fact it is extended to the  localization of the ring of algebraic integers with respect to the mentioned maximal ideal.)
Using~\eqref{eqRootOfUnity}
one attaches to a representation~$\eta\in\mf{R}_{\CoefField}(\K)$ a complex character~$\chi_\eta$, the Brauer character of~$\eta$,  and therefore a representation~$\Br_\K(\eta)\in\mf{R}_\bbC(\K)$, and this definition is compatible with direct sums and restriction~$(\K\geq\K')$:
\begin{equation}\label{eqRestriction}
\begin{array}{ccc}\mf{R}_\CoefField(\K) &\stackrel{\Br_\K}{\longrightarrow} & \mf{R}_\bbC(\K)\\
\Res^\K_{\K'}\downarrow& \circlearrowleft & \downarrow\Res^\K_{\K'}\\
\mf{R}_\CoefField(\K') &\stackrel{\Br_{\K'}}{\longrightarrow} & \mf{R}_\bbC(\K').\\
  \end{array}
\end{equation}

We call~$\Br_\K$ the~\emph{Brauer map} for~$\K$. Note that the map~$(\chi_\eta)^*$ given by~$(\chi_\eta)^*(g):=(\chi_\eta(g))^*,\ g\in\K,$ is the character of~$\eta$ with values in~$\CoefField$ given by the trace. 
\begin{theorem}[{\cite[p18, Theorems 2.6 and 2.12]{navarro:98}}]\label{thmBr}
The Brauer map for~$\K$ is a bijection, preserves dimensions and maps the set of classes of irreducible~$\CoefField$-representations onto the set of classes of irreducible~$\bbC$-representations of~$\K$.
\end{theorem}  

By the following theorem~$\Br$ commutes with induction.

\begin{proposition}[Brauer--Nesbitt~{\cite[Theorem (8.2)]{navarro:98}}]\label{propBrInd}
 The following diagram is commutative.
 \begin{equation}\label{eqInduction}\begin{array}{ccc}
\mf{R}_\CoefField(\K) &\stackrel{\Br_\K}{\longrightarrow} & \mf{R}_\bbC(\K)\\
\ind^\K_{\K'}\uparrow& \circlearrowleft & \uparrow\ind^\K_{\K'}\\
\mf{R}_\CoefField(\K') &\stackrel{\Br_{\K'}}{\longrightarrow} & \mf{R}_\bbC(\K').\\
\end{array}
\end{equation}\end{proposition}


One can now transfer the Glauberman correspondence~\cite{glauberman:68} to the modular case via the Brauer map. 

Let~$\Gamma$ be a finite solvable operator group acting on~$\K$ such that~$\K$ and~$\Gamma$ have relatively prime orders. Let~$\K^\Gamma$ be the fixed point set of~$\K$. 
We consider now sets of isomorphism classes of irreducible representation 
denoted by~$\Irr_{?}(!)$. 
The group~$\Gamma$ acts on~$\Irr_{?}(\K)$ and the fixed point set will be denoted by~$\Irr_{?}(\K)^\Gamma$. Then Glauberman constructs a bijection 
\[\gl^\Gamma_\bbC:~\Irr_\bbC(\K)^\Gamma\simarrow\Irr_\bbC(\K^\Gamma),\]
and we define the~\emph{Glauberman transfer}~$\gl^\Gamma_\CoefField$ for~$\CoefField$-representations via~$\gl^\Gamma_\CoefField:=\Br^{-1}_{\K^\Gamma}\circ\gl^\Gamma_\bbC\circ\Br_\K|_{\Irr_{\CoefField}(\K)^\Gamma}$. 

If~$\Gamma$ is the Galois group of the field extension~$\L|\F$ then we write~$\gl_\CoefField^{\L|\F}$ for the Glauberman transfer.

\begin{remark}
 In this work we only apply the Glauberman correspondence for~$\Gamma$ a cyclic group of order two. So here~$\gl_?(\eta)$ is the unique element of~$\Irr_?(\K^\Gamma)$ with odd multiplicity in~$\Res_{\K^\Gamma}^\K(\eta)$.
 \end{remark}

We set for~$\CoefField$ with characteristic zero~$\Br_{\K}$ to be the identity map in terms of characters. 
 
\section{Exhaustion for semisimple Characters}
In this section we prove the following theorem.
\begin{theorem}[see \cite{stevens:05}~5.1 and~\cite{dat:09}~\S8.9\ for the non-quaternionic case]\label{thmStDuke5.1ForQuatCase}
 Let~$\pi$ be an irreducible representation of~$\G$. Then there is a self-dual semisimple stratum~$\Delta$ with~$r=0$ and an element~$\theta$ of~$\C(\Delta)$ such that~$\theta$ is contained in~$\pi$. 
\end{theorem}

The proof of this theorem requires several steps similar to~\loccit.
We fix an irreducible smooth representation~$\pi$ of~$\G$. The proof of the theorem is done by induction. 
\begin{enumerate}
 \item In the base case we have the existence of a trivial semisimple character of the same depth as~$\pi$ contained in~$\pi$.
 \item The induction for~$\theta\in\C(\Delta)$ contained in~$\pi$ is on the fraction~$\frac{r}{e(\Lambda|F)}$. 
 \item For the induction step we need to be able to change lattice sequences: Roughly speaking, given a self-dual semisimple character~$\theta\in\C(\Delta)$ with positive~$r$ and contained in~$\pi$ there is a self-dual stratum~$\Delta'$ such that~$\beta=\beta'$ 
 and~$\frac{r'}{e(\Lambda'|F)}<\frac{r}{e(\Lambda|F)}$ and an element~$\theta'\in\C(\Delta')$ which is contained in~$\pi$. 
 \item These steps are not enough, because one has to ensure that the difference between~$\frac{r}{e(\Lambda|F)}$ and~$\frac{r'}{e(\Lambda'|F)}$ is bounded from below by a positive constant independent of~$\Lambda$ and~$\Lambda'$. 
\end{enumerate}

Recall that the depth of~$\pi$ is the infimum of all  non-negative real numbers~$t$ with a point~$x$ in~$\mf{B}(\G)$ such that the trivial representation of~$\G_{x,t+}$ is contained in~$\pi$.   
Let us recall that the barycentric coordinates of a point~$x$ in~$\mf{B}(\G)$ are the barycentric coordinates of the point with respect to the vertexes of any chamber containing~$x$. 
\begin{lemma}[see~\cite{moyPrasad:94}~5.3,~7.4]\label{lemdepth}
 The depth of~$\pi$ is attained at a point with rational barycentric coordinates. 
\end{lemma}

This follows from~\loccit\ because the depth is attained at an optimal point, see~\loccit~7.4. 
\begin{remark}
 One could think that a continuity argument with the function
 \[x\in \mf{B}(\G)\mapsto \depth(\pi,x):=\inf\{s\geq 0|\ V_\pi^{\G_{x,s+}}\neq 0\} \]
 could lead to Lemma~\ref{lemdepth}, but it is unclear if this function is continuous. It is upper-continuous, but maybe not lower continuous. Here the idea in~\loccit\ of taking optimal points comes into play 
 which form a finite set for a given chamber~$\C$.  
\end{remark}

We prove the upper-continuity of~$\depth(\pi,*)$ in the above remark. It is not needed for what follows in this article. 
\begin{proof}
 Take~$x\in \mf{B}(\G)$. It corresponds to a self-dual lattice function~$\Gamma_x$ with set~$\disc(x)$ of discontinuity points of~$\Gamma$. We take a CAT(0)-metric~$d(*,*)$ on~$\mf{B}(\G)$ given in~\cite{corvallisTits:79}. The point~$x$ lies in the interior of a facet~$F_x$.
 Then for all positive real~$\delta_1$ there exists a positive real~$\delta_2$ such that 
 \begin{enumerate}
  \item The ball around~$x$ with radius~$\delta_2$ does not intersect any facet of lower dimension than the dimension of~$F_x$.
  \item For all~$x'\in \mf{B}(\G)$ with~$d(x,x')<\delta_2$ and all~$t\in\disc(x)$ there exists a~$t'\in\disc(x')$ such that~$|t-t'|<\delta_1$ and~$\Gamma_{x}(t)=\Gamma_{x'}(t')$.
  \item For all~$x'\in \mf{B}(\G)$ with~$d(x,x')<\delta_2$ and all~$t'\in\disc(x')$ there exists a~$t\in\disc(x)$ such that~$|t-t'|<\delta_1$ and~$\Gamma_{x}(t)\supseteq\Gamma_{x'}(t')$. \label{uppercontiii}
 \end{enumerate}
 Statement~\ref{uppercontiii} implies~$\Gamma_{x}(t)\supseteq\Gamma_{x'}(t+\delta_1)$ for all~$t\in\bbR$. 
 Then, for every~$x'\in \mf{B}(\G)$ with~$d(x,x')<\delta_2$ and every~$s\geq 0$, the group~$\G_{x,s+}$ contains~$\G_{x',(s+2\delta_1)+}.$
 In particular those~$x'$ satisfy
 $\depth(\pi,x')\leq \depth(\pi,x)+2\delta_1$ 
 which finishes the proof.
\end{proof}

For the proofs of the next lemmas we need some duality for the Moy--Prasad filtrations. 
Given a subset~$S$ of~$\A$ the dual~$S^*$ of~$S$ with respect to~$\psi_\F$ is defined as the subset of~$A$ consisting of all elements~$a$ of~$\A$ which satisfy~$\psi_A(sa)=1$ for 
all~$s\in S$. ($\psi_\A:=\psi_\F\circ\trd$).
The main property of this duality is:

\begin{lemma}\label{lemStarD}
 Let~$x$ be a point of~$\mf{B}(\tiG)$. Then we have~$\timfg^*_{x,t}=\timfg_{x,-t+}$ for all~$t\in\bbR$.
\end{lemma}

\begin{proof}
 We choose a splitting basis~$(v_k)_k$ for~$\Gamma_x$, i.e. 
 \[\Gamma_x(t)=\oplus_k v_k\pD^{\lceil d(t-a_k)\rceil}.\]
 \daniel{Depending on the choice of the basis we have an~$\F$-algebra isomorphism $\A$ $\simeq\M_m(\D)$ and therefore a left~$\D$-vector space action on~$\A$ given by the canonical one on~$\M_m(\D)$. For~$1\leq j,k\leq m$ let~$E_{ij}$ be the element of~$\A$ with kernel~$\oplus_{k\neq j} v_\k\D$ which sends~$v_j$ to~$v_i$.}
 They form a~$\D$-left basis of~$\A$ which splits~$\timfg_x$, more precisely: 
 \[\timfg_{x,t}=\oplus_{ij}\pD^{\lceil d(t+a_j-a_i)\rceil}E_{ij}.\]
 We now show the assertion of the lemma. The inclusion~$\supseteq$ is obvious. For the other inclusion we first remark that for every real number~$t$ the lattice~$\timfg^*_{x,t}$ is split by~$(E_{ij})_{ij}$ too:
 \[\timfg^*_{x,t}=\oplus_{ij}\pD^{e_{ij}(t)}E_{ij},\ e_{ij}(t)\in \bbZ\]
 Then~$e_{ji}(t)+\lceil d(t+a_j-a_i)\rceil$ is positive  and therefore~$e_{ji}(t)> d(-t+a_i-a_j)$ which finishes the proof since~$e_{ji}(t)$ is an integer. 
\end{proof}

For an element~$\beta\in \A$ we define the map~$\tipsi_\beta:\A\ra\CoefField$ via~$\tipsi_\beta(1+a):=\psi_\A(\beta a)$. Some restrictions of~$\tipsi_\beta$ are characters, i.e. multiplicative, as for example in the following case: 

\begin{definition}
 Let~$\Delta=[\Lambda,n,n-1,\beta]$ be a  stratum which is not equivalent to a null-stratum. Then we define~$d_\Delta:=\frac{n}{e(\Lambda|F)}$ to be the depth of~$\Delta$. Let~$x\in\mf{B}(\tiG)$ be a point corresponding to~$\Lambda$.
 The coset of~$\Delta$ in terms of the building is defined as~$\beta+\timfg_{x,-d_\Delta+}$, and if~$\Delta$ is self-dual then we call~$\beta+\mfg_{x,-d_\Delta+}$ 
 its~\textit{self-dual} coset. To~$\Delta$ is attached the character~$\tipsi_\Delta:\ \tiG_{x,d_\Delta}\ra\CoefField$ defined via restriction of~$\tipsi_\beta$. 
  Note that~$\tipsi_\Delta$ is trivial on~$\tiG_{x,d_\Delta+}$ by Lemma~\ref{lemStarD}.
 If~$\Delta$ is self-dual we write~$\psi_\Delta$ for the restriction of~$\tipsi_\Delta$ to~$\G_{x,d_\Delta}$.
 We say that~$\pi$ contains~$\Delta$ or the associate coset, if it contains~$\psi_\Delta$. 
\end{definition}

\begin{proposition}[cf.~\cite{stevens:02}~2.11 and~2.13 for~$\G\otimes \L$]\label{propSemisimpleStratuminpi}
 Suppose~$\pi$ has positive depth. Then~$\pi$ contains a self-dual semisimple stratum~$\Delta$ with~$n=r+1$ and of the same depth. 
\end{proposition}

For the proposition we need a convexity lemma for semisimple strata. Let us recall: The minimal polynomial~$\mu_\Delta$ for a stratum~$\Delta:=[\Lambda,n,n-1,\beta]$ is the minimal polynomial in~$\k_\F[X]$ of the residue class~$\bar{\mf{\eta}}(\Delta)$ of~$\mf{\eta}(\Delta):=\varpi_\F^\frac{n}{\gcd(n,e(\Lambda|F))}\beta^\frac{e(\Lambda|F)}{\gcd(n,e(\Lambda|F))}$ modulo~$\timfa_{1}(\Lambda)$, see~\bfI.\S4.2.
Further we recall the~\emph{characteristic polynomial}~$\chi_\Delta$ of~$\Delta$ which is the mod~$\mf{p}_\F$-reduction of the reduced characteristic polynomial of $\mf{\eta}(\Delta)$.
The polynomials~$\mu_\Delta$ and~$\chi_\Delta$ depend on the choice of~$\varpi_\F$. 
For the next lemma we need the~\emph{barycentre}~$\frac12\Lambda+\frac12\Lambda'$ of two lattice sequences~$\Lambda,\Lambda'$ with the same~$\F$-period~$e$:
\[\left(\frac12\Lambda+\frac12\Lambda'\right)(m):=\left(\frac12\Gamma_\Lambda+\frac{1}{2}\Gamma_{\Lambda'}\right)\left(\frac{m}{2e}\right),\ m\in\bbZ.\]
Here we have used the affine structure on~$\Latt^1_{\o_\D}\V$, see the description in~\cite[Proof of~7.1]{broussousStevens:09}. 

\begin{lemma}\label{lemConvexitySemisimpeStrata}
 Suppose that~$\Delta:=[\Lambda,n,n-1,\beta]$ and~$\Delta':=[\Lambda',n,n-1,\beta]$ are strata over~$\D$ which are equivalent to semisimple strata and share the characteristic polynomial and the~$\D$-period. Then~$\Delta'':=[\frac12\Lambda+\frac12\Lambda',2n,2n-1,\beta]$ is equivalent to a semisimple stratum. 
\end{lemma}


\begin{proof}
It suffices to prove that~$\Res_\F(\Delta'')$ is equivalent to a semisimple stratum, by~\cite[Theorem~4.54]{skodlerack:17-1}, and therefore for the proof we suppose that~$\Delta,\Delta'$ are strata over~$\F$ instead of~$\D$.
 At first~$\Delta$ and~$\Delta'$ have the same minimal polynomial because it is the radical of the characteristic polynomial by the equivalence to semisimple strata. 
 By convexity, see~\cite[5.5]{stevens:02}, we obtain that the minimal polynomial of~$\Delta''$ divides~$\mu_\Delta$. But since~$\Delta$ is equivalent 
 to a semisimple stratum we obtain by~\bfI.4.8\ that~$\mu_\Delta$ and therefore~$\mu_{\Delta''}$ is square-free. Thus, in case that~$X$ does not
 divide~$\mu_{\Delta''}$, we are done by~\bfI4.8.  
 
 Suppose now that~$X$ divides~$\mu_\Delta$. 
 The element~$\bar{\mf{\eta}}(\Delta)$ generates a semisimple algebra over~$\k_\F$ which is isomorphic to the algebra generated by~$\bar{\mf{\eta}}(\Delta')$
 via~$\mf{\eta}(\Delta)$ mod~$\timfa_1(\Lambda)$ is send to~$\mf{\eta}(\Delta')$ mod~$\timfa_1(\Lambda')$, see~\cite[4.46]{skodlerack:17-1}. 
 Let~$e$ be an idempotent which splits~$\Lambda$ and  commutes  with~$\b$ and such that the minimal polynomial of the 
 stratum~$e\Delta:=[e\Lambda,n,n-1,e\beta]$ is~$X$ and~$X\nmid\mu_{(1-e)\Delta}$. See~\cite[6.11]{skodlerackStevens:18} for the construction. 
 The element~$e$ is a polynomial in~$\mf{\eta}(\Delta)$ (Note:~$\mf{\eta}(\Delta)=\mf{\eta}(\Delta')$) with coefficients in~$\o_\F$.  
 Thus~$e$ also splits~$\Lambda'$. Further~$e\Delta$ and~$e\Delta'$ intertwine which implies that~$e\Delta'$ cannot be fundamental by~\cite[6.9]{skodlerackStevens:18}, i.e. the square-free minimal 
 polynomial of~$e\Delta'$  needs to be just~$X$. The stratum~$(1-e)\Delta'$ is fundamental and has the same minimal polynomial as~$(1-e)\Delta$, because the second is 
 fundamental and both strata intertwine. Thus~$(1-e)\Delta''$ is equivalent to a semisimple stratum by the first part of the proof. 
 This reduces to the case~$e=1$, i.e.~$\Delta$ and~$\Delta'$ are equivalent to null-strata. But then~$\Delta''$ 
 is equivalent to a null-stratum too, by convexity~\cite[5.5]{stevens:02}. This finishes the proof.
\end{proof}

For the proof of Proposition~\ref{propSemisimpleStratuminpi} we are going to use Theorem~\ref{thmStevens4p4ForNonGsplitAndAlsoForNonFundamental} (cf.~\cite[Theorem~4.4]{stevens:02}) whose proof uses the erratum of~\cite[Proposition~4.2]{stevens:02} in~\S\ref{appErratum4p2}, see Proposition~\ref{propStevensStrataFix}.

\begin{proof}[Proof of Proposition~\ref{propSemisimpleStratuminpi}]
The depth of~$\pi$ is rational because,~ by~\cite{moyPrasad:94},~$\pi$ attains its depth at an optimal point of~$\mf{B}(\G)$, say~$x$, in particular at a point with rational barycentric coordinates. Then there is an element~$b$ of~$\mfg_{x,-d_\pi}$ such that~$b+\mfg_{x,-d_\pi+}$ is contained in~$\pi$. 
We are going to show that there is a self-dual semisimple stratum~$\Delta$ with~$r+1=n$ whose coset (in~$\timfg$) contains the coset~$b+\timfg_{x,-d_\pi+}$. 
By Theorem~\ref{thmStevens4p4ForNonGsplitAndAlsoForNonFundamental} there is a point~$x'_\L\in \mf{B}(\G\otimes \L)$ with rational barycentric coordinates which satisfies the following property $(*)$: The coset~$b+\timfg_{x'_\L,-d_\pi+}$ is a coset of a semisimple stratum over~$ \L$,~$b\in\timfg_{x'_\L,-d_\pi}$ and~$\timfg_{x'_\L,-d_\pi+}$ 
contains~$\timfg_{i_\L(x),-d_\pi+}$.
The Galois group~$\langle\tau\rangle$ of~$\L|\F$ acts on~$\mf{B}(\tiG\otimes \L)$ with fixed point set~$\mf{B}(\tiG)$, and~$\tau(x'_\L)$ also satisfies property~$(*)$.
We define~$x''_\L$ as the barycentre of the segment between~$x'_\L$ and~$\tau(x'_\L)$. Then~$x''_\L$ also satisfies~$(*)$. Indeed: the containments are trivial by convexity~\cite[5.5]{stevens:05}, and  the corresponding coset is a coset of a semisimple stratum, by Lemma~\ref{lemConvexitySemisimpeStrata}. 
The point~$x''_\L$ is fixed by~$\tau$ and is therefore of the form~$i_\L(x'')$ for some~$x''\in \mf{B}(\G)$. 
Let~$\Delta''=[\Lambda'',n'',n''-1,b]$ be a self-dual stratum for the coset~$b+\timfg_{x'',-d_\pi+}$.  
Note that~$\Delta''$ is not equivalent to a null-stratum, i.e.~$b\not\in\timfg_{x'',-d_\pi+}$, by the definition of~$d_\pi$, as the coset~$b+\mfg_{x'',-d_\pi+}$ is contained in~$\pi$. In other words: the depth of~$\Delta''$ is~$d_\pi$. As~$\Delta''\otimes \L$ is equivalent to a semisimple stratum and as~$b\in \Lie(\G)$
we obtain that~$\Delta''$ is equivalent to a self-dual semisimple stratum by~\bfI4.54 and~\bfII.4.7.
\end{proof}

The same idea of base extension using Theorem~\ref{thmStevens4p4ForNonGsplitAndAlsoForNonFundamental} shows: 
\daniel{
\begin{corollary}[cf.~\cite{stevens:02}~4.4]\label{corFundSemisimple}
Let~$\Delta$ be a self-dual  stratum with~$n=r+1$. 
\begin{enumerate}
 \item If~$\Delta$ is fundamental then there exists a self-dual semisimple stratum~$\Delta'$ with~$n'=r'+1$ such that
$\frac{n}{e(\Lambda|F)}=\frac{n'}{e(\Lambda'|F)}$, and~$\beta+\timfa_{-r}\subseteq\beta'+\timfa'_{-r'}$.   
\item\label{corFundSemisimpleii} If~$\Delta$ is not fundamental, then there exists a self-dual null-stratum~$[\La',r',r',0]$ such that
\[\beta+\timfa_{-r}(\La)\subseteq\timfa'_{-r'}(\La')\]
and~$\frac{r'}{e(\La'|\F)}$ is smaller than~$\frac{n}{e(\La|\F)}$
\end{enumerate}\end{corollary}
}

\daniel{Apparently we also need the~$\tiG$-version, whose proof is similar to the proof of the previous corollary, using the~$\tiG$-part of Theorem~\ref{thmStevens4p4ForNonGsplitAndAlsoForNonFundamental} instead. 
\begin{corollary}[For (ii) cf.~\cite{secherreStevensIV:08}~3.11]\label{corFundSemisimpletiG}
Let~$\Delta$ be a stratum with~$n=r+1$.
\begin{enumerate}
 \item If~$\Delta$ is fundamental then there exists a semisimple stratum~$\Delta'$ with~$n'=r'+1$ such that
$\frac{n}{e(\Lambda|F)}=\frac{n'}{e(\Lambda'|F)}$, and~$\beta+\timfa_{-r}\subseteq\beta'+\timfa'_{-r'}$.   
\item\label{corFundSemisimpletiGii} If~$\Delta$ is not fundamental, then there exists a null-stratum~$[\La',r',r',0]$ such that
\[\beta+\timfa_{-r}(\La)\subseteq\timfa'_{-r'}(\La')\]
and~$\frac{r'}{e(\La'|\F)}$ is smaller than~$\frac{n}{e(\La|\F)}$ \end{enumerate}\end{corollary}
}
From now on we assume in this section that~$\pi$ has positive depth. 
By Proposition~\ref{propSemisimpleStratuminpi} there exists a self-dual semisimple stratum~$[\Lambda,n,n-1,\beta]$ contained in~$\pi$. 
\daniel{Now let~$\mathcal{M}$ and~$\mathcal{M}_\F$ be the Levi sub-algebra of~$\A=\End_\D\V$ and~$\A_\F=\End_\F\V$, respectively, which stabilizes the associated decomposition 
\begin{equation}\label{eqassDecomp}
 \bigoplus_{i\in\I}\V^i=\V
\end{equation}
of~$\b$, and we denote by~$\mathcal{M}_\L\subseteq\End_\L\V$ the Levi sub-algebra corresponding to~$\b\otimes 1$. Then~$\mathcal{M}_\L\subseteq\mathcal{M}\otimes_\F\L$.}
We formulate the key proposition for the induction step for Theorem~\ref{thmStDuke5.1ForQuatCase}:

\begin{proposition}[see~\cite{stevens:05}~5.4~over~$\L$]\label{propthetaInductionStep}
 Given a self-dual semisimple stratum~$\Delta$ with positive~$r$, an element~$\theta\in\C(\Delta(1-))$ and an element~$c\in\mf{a}_{-r}\cap\mathcal{M}$, we suppose that~$\theta\psi_c$ 
 is contained in~$\pi$. We fix a self-dual tame corestriction~$s_\beta$ with respect  to~$\beta$. Let~$\Lambda'$ be a self-dual~$\o_E$-$\o_D$-lattice sequence,~$r'$ a positive integer 
 and~$b'$ an element of~$\mf{b}'_{-r'}\cap \mf{b}_{-r}$  such that~$s_\beta(c)+\timfb_{1-r}$ is contained in~$b'+\timfb'_{1-r'}$. 
 Suppose further that~$\frac{r'}{e(\Lambda'|F)}\leq\frac{r}{e(\Lambda|F)}$.
 Then~$\Delta'$ with~$\beta':=\beta$ is a self-dual semisimple stratum and there are~$\theta'\in\C(\Delta'(1-))$ and~$c'\in\mf{a}'_{-r'}$ such that
 $s_\beta(c')$ is equal to~$b'$ and~$\theta'\psi_{c'}$ is contained in~$\pi$. The element~$c'$ can be chosen to vanish if~$b'=0$. 
\end{proposition}

Essentially~\textbf{Proposition~\ref{propthetaInductionStep}} says that if~$\theta\psi_c$ is contained in~$\pi$ then one can work in~$\B$ to look for a ``better'' character~$\theta'\psi_{c'}$.
Note that assuming~$b'$ to be also contained in~$\mf{b}_{-r}$ is not a restriction, because we could just take~$b'=s(c)$.~\textbf{We explain the strategy of its proof} (taken from~\cite[5.4]{stevens:05}): 
At first one constructs open compact subgroups~$\ti{\K}^t_1(\Lambda)$ and~$\ti{\K}^t_2(\Lambda)$ ($t\in\bbN$) of~$\ti{\P}^r(\Lambda)$ via 
\begin{equation}\label{eqK1}\ti{\K}_1^t(\Lambda):=1+\timfa_{\lfloor\frac{t}{2}\rfloor+1}\cap(\prod_{i\neq j}A^{ij}\oplus\prod_i\timfa^i_{t}),\end{equation}
\begin{equation}\label{eqK2}\ti{\K}_2^t(\Lambda):=1+\timfa_{\lfloor\frac{t+1}{2}\rfloor}\cap(\prod_{i\neq j}A^{ij}\oplus\prod_i\timfa^i_{t}),\end{equation}
and further~$\tiH_i^t(\beta,\Lambda)$ and~$\tiJ_i^t(\beta,\Lambda)$ as intersections of~$\tiH(\beta,\Lambda)$ and~$\tiJ(\beta,\Lambda)$ with~$\ti{\K}_i^t(\Lambda)$. 
And we get the groups~$\K^t_i,\H_i^t$ and~$\J_i^t$ if we intersect further down to~$\G$. 
We extend~$\theta$ to a semisimple character of~$\H^{\lfloor\frac{r}{2}\rfloor+1}(\beta,\Lambda)$ (which we still call~$\theta$)  and we consider the character~$\xi:=\theta\psi_c$ on~$\H^r_1(\beta,\Lambda)$. 
Mutatis mutandis as in~\cite[5.7]{stevens:05} one shows that~$\xi$ is contained in~$\pi$ (using~\bfI4.51, instead of~\cite[3.17]{stevens:05}, and~\cite[3.20]{stevens:05} on~$\G\otimes_\F\L$ followed by restriction to~$\G$.) This representation~$\xi$ is very helpful for detecting if a certain  representation is contained in~$\pi$:

\begin{lemma}[cf.~\cite{bushnellKutzko:93}~(8.1.7),~\cite{stevens:05}~5.8]\label{lemSubrepCriteria}
 We granted~$r>0$. Let~$\rho$ be an irreducible representation on an open subgroup~$\U$ of~$\K^r_2(\Lambda)$. 
 Then $\rho$ is a subrepresentation of~$\pi$, if its restriction to~$\U\cap \H^r_1(\beta,\Lambda)$ contains~$\xi|_{\U\cap \H^r_1(\beta,\Lambda)}$. 
\end{lemma}

By~\cite[5.12]{stevens:05} we can cut the line in~$\mf{B}(\G)$ between~$\Lambda$ and~$\Lambda'$ into segments with cutting points~$\Lambda=:\Lambda_1,\Lambda_2,\ldots,\Lambda_s:=\Lambda'$ such that for each index~$1\leq k<s$ we have~$\tiP_{r_{k+1}}(\Lambda_{k+1})\subseteq\tiK^{r_k}_2(\Lambda_\k)$. 
For the definition of~$r_t$, see~\cite[5.1]{stevens:05}.
Indeed: One applies~\loccit\ to the line in~$\mf{B}(\Aut_\F(\V))$ and intersects the inclusions down to~$\tiG$.  
They still satisfy~$\frac{r_t}{e(\Lambda_t|F)}\geq \frac{r_{t+1}}{e(\Lambda_{t+1}|F)}$. 
By~\loccit\ this reduces Proposition~\ref{propthetaInductionStep} to the case~$s=2$. 
Thus we have to prove Case~$s=2$ and Lemma~\ref{lemSubrepCriteria} to obtain  Proposition~\ref{propthetaInductionStep}.

For the proof of Lemma~\ref{lemSubrepCriteria} we need:

\begin{lemma}[cf.~\cite{stevens:05}~5.9,~\cite{kurinczukStevens:19}~3.4]\label{lemHeisenberg}
Granted~$r>0$, there is a unique irreducible representation~$\mu$ of~$\J_1^r$ containing~$\xi$ (called the Heisenberg representation of~$\xi$ to~$\J_1^r$),
because the bi-linear form
 \[k_\xi: \J^r_1/\H^r_1 \times  \J^r_1/\H^r_1\ra\CoefField^\times,\ k_\xi(\bar{x},\bar{y}):=\xi([x,y])=\theta([x,y])\]
 is non-degenerate.
\end{lemma}

\daniel{
We will use base extension to~$\L$ for the proof. So we use the following notation over~$\L$:
\begin{notation}\label{notBaseExtKgroup}
If~$\b$ and~$\Lambda$ are fixed and~$t\in\NN_0$ then we write:
\begin{itemize}
 \item $\tilde{\H}_\L^t=\tilde{\H}^t(\beta\otimes 1,\Lambda_\L),\ \tilde{\J}_\L^t=\tilde{\J}^t(\beta\otimes 1,\Lambda_\L)$ and~$\H_\L^t,\ \J_\L^t$ for their intersections with $\G\otimes_\F\L$ (They all are subgroups of~$\tilde{\G}\otimes_\F\L=\Aut_\L(\V)$). 
 \item We also define, for~$s\in\{1,2\}$, the groups~$\tilde{\K}^t_{\L,s}(\Lambda)$ in replacing in~\eqref{eqK1} and~\eqref{eqK2} $\timfa_*(\Lambda)$ by~$\timfa_*(\Lambda_\L)$ and~$\A^{ij}$ by~$(\A\otimes_\F\L)^{ij}$,~$i,j\in\I$. 
 Note that this is not the group given in~\cite[5.2]{stevens:05}, because~$\L[\beta\otimes 1]$ could have more factors than~$\F[\beta]$ and we consider~$\mc{M}\otimes_\F\L$ instead of~$\mc{M}_\L$. 
 \item We then define  the groups
 \[\J_{\L,s}^t=\tilde{\K}_{\L,s}^t\cap\J_\L^t,\ \H_{\L,s}^t=\tilde{\K}_{\L,s}^t\cap\H_\L^t.\]
\end{itemize}
\end{notation}
}

\begin{proof}
Let~$\xi_\L$ and~$\theta_\L$ be the unique~$\Gal(\L|\F)$-fixed extensions to~$\H^r_{\L,1}$ and $\H^{\lfloor\frac{r}{2}\rfloor+1}_\L$  of~$\xi$ and~$\theta$, respectively.
We have the analogous form~$\k_{\xi_\L}$ on~$\J^r_{\L,1}/\H^r_{\L,1}$ for $\xi_\L$, and this form is non-degenerate by the proof in~\cite[5.9]{stevens:05}. Let~$\bar{x}$ be in the kernel of~$k_\xi$ and let~$y$ be an element of~$\J^r_{\L,1}$.
 Then 
 \[\k_{\xi_\L}(\bar{x},\bar{y})=\theta_\L([x,y])=\theta_\L([x,\tau(y)])=\k_{\xi_\L}(\bar{x},\overline{\tau(y)}).\] 
 In particular
 \[\k_{\xi_\L}(\bar{x},\bar{y})^2=\k_{\xi_\L}(\bar{x},\bar{y}\tau(\bar{y})).\]
 We take here the obvious Galois-action on~$\J^r_{\L,1}/\H^r_{\L,1}$. Its fixed point set is~$\J^r_1/\H^r_1$ because the first~$\Gal(\L|\F)$-cohomology of~$\H^r_{\L,1}$ is trivial, 
 in particular~$\bar{y}\tau(\bar{y})$ is an element of~$\J^r_1/\H^r_1$. Thus~$k_{\xi_L}(\bar{x},\bar{y})^2$ vanishes and therefore~$k_{\xi_\L}(\bar{x},\bar{y})=1$, because
 it is a~$p$-th root of unity. Thus~$\k_\xi$ is non-degenerate.
\end{proof}

\begin{proof}[Proof of Lemma~\ref{lemSubrepCriteria}]
 For the proof we skip the argument~$\Lambda$ in the notation. 
 The proof needs two parts: We show
 \begin{enumerate}
  \item There exists up to isomorphism only one irreducible representation~$\omega$ of~$\K_2^r$ which contains~$\xi$. \label{pflemSubreCriteria-Part1}
  In fact we will further obtain that~$\ind_{\H^r_1}^{\K^r_2}\xi$ is a multiple of~$\omega$. 
  \item The restriction of~$\pi$ to~$U$ contains~$\rho$. \label{pflemSubreCriteria-Part2}
 \end{enumerate}
 Part~\ref{pflemSubreCriteria-Part2} follows as in the final argument in the proof of~\cite[(8.1.7)]{bushnellKutzko:93}.
 We only have to prove Part \ref{pflemSubreCriteria-Part1}. 
 We take the representation~$\mu$ of Lemma~\ref{lemHeisenberg} and prove that the~$\omega:=\ind_{\J^r_1}^{\K^r_2}\mu$ is irreducible. Let~$g$ be an element of~$\K^r_2$ which intertwines~$\mu$. Then the~$g$-intertwining space~$\I_g(\mu)$ 
 satisfies the formula:
 \[(\J^r_1:\H^r_1)\dim_\CoefField(\I_g(\mu))=\dim_\CoefField(\I_g(\ind_{\H^r_1}^{\J^r_1}\xi))=\left|\H^r_1\backslash \J^r_1g\J^r_1/\H_1^r\right|\]
 by~\cite[(4.1.5)]{bushnellKutzko:93}. The latter cardinality is equal to
 \[\frac{(\J^r_1:\H^r_1)(\J^r_1:\J^r_1\cap g^{-1}\J^r_1 g)}{(\H^r_1:\H^r_1\cap g^{-1}\H^r_1 g)},\]
 and therefore odd. Thus~$g$ intertwines the Glaubeman transfer~$\mu_\L$ of~$\mu$ by~\cite[2.4]{stevens:01-2}, in particular~$g$ is an element of~$\J^r_{\L,1}$ 
 by the proof of~\cite[5.9]{stevens:05}. (It shows~$\I_{\K^r_{\L,2}}(\mu_\L)\subseteq \J_{\L,1}^r$, see the second part of the proof in~\loccit\ and Remark~\ref{remonSt055.9} below.) Thus~$g\in \J^r_1$. 
 Therefore~$\omega$ is irreducible and 
 \[\ind_{\H^r_1}^{\K^r_2}\xi\cong\ind_{\J^r_1}^{\K^r_2}\ind_{\H^r_1}^{\J^r_1}\xi\cong \ind_{\J^r_1}^{\K^r_2}\mu^{\oplus (\J_1^r:\H_1^r)^{\frac12}}\cong \omega^{\oplus (\J_1^r:\H_1^r)^{\frac12}}\]
 which finishes the proof.  
\end{proof}

\daniel{
\begin{remark}\label{remonSt055.9}
 The proof in~\cite[5.9]{stevens:05}
 uses an Iwahori decomposition adapted to~$\mc{M}_\L$. But still the proof works with an Iwahori decomposition adapted to~$\mc{M}\otimes_\F\L$ which may properly contain~$\mc{M}_\L$. At the end of the proof in \loccit\ the author refers to~\cite[(8.1.8)]{bushnellKutzko:93} for the~\emph{simple} case. In our situation, where we split a stratum over~$\L$ with respect to~$\mc{M}\otimes_\F\L$, the proof in~\cite[(8.1.8)]{bushnellKutzko:93} is applied to the \emph{quasi-simple} case (a notion introduced by Sech\'erre in~\cite{secherreI:04}), i.e. to a stratum~$\Delta\otimes\L$ where~$\Delta$ is a simple stratum over~$\D$. The used part of the proof in~\cite[(8.1.8)]{bushnellKutzko:93} is the induction. 
\end{remark}
}

To finally prove Proposition~\ref{propthetaInductionStep} we need the Cayley map (depending on~$\Lambda$)
\[\Cay:\ \mfa_1\ra \P_1(\Lambda),\ \Cay(a):=\frac{1+\frac{a}{2}}{1-\frac{a}{2}}.\]
It is a bijection. 

\begin{proof}[Proof of Proposition~\ref{propthetaInductionStep}]
 We only need to consider the case~$s=2$ by the above explanation, see the paragraph following Lemaa~\ref{lemSubrepCriteria}. We are going to use base extension to use parts of the proof of~\cite[5.4 (with assumption (H))]{stevens:05}. 
 We need to find~$\theta'\in\C(\Delta'(1-))$ and~$c'\in\mf{a}_{-r'}(\Lambda')$ such that~$\theta'\psi_{c'}\subseteq\pi$ and~$s(c')=b'$.
 
 \textbf{Step 1}: At first,~$\Delta'=[\Lambda',n',r',\b]$ is semisimple, because 
 \[\frac{r'}{e(\Lambda'|\F)}\leq\frac{r}{e(\Lambda|\F)}<\frac{-k_0(\b,\Lambda)}{e(\Lambda|\F)}=\frac{-k_0(\b,\Lambda')}{e(\Lambda'|\F)}\]
 
 \textbf{Step 2} (find~$\theta'$): Let~$\theta_\L\in\C((\Delta\otimes\L)(1-))$ be the Glauberman lift of~$\theta\in\C(\Delta(1-))$. We write~$\Lambda_\L$ for the lattice sequence~$\Lambda$ seen as an~$\o_\L$-lattice sequence. Note that we have 
 \[\C((\Delta\otimes\L)(1-))=\C(e(\Lambda'|\F)\Lambda_\L,e(\Lambda'|\F)r-1,\b\otimes 1)\]
 and  
 \[e(\Lambda'|\F)r-1\geq e(\Lambda|\F)r'-1,\]
 and we take an extension~$\theta_\L^{1}\in\C(e(\Lambda'|\F)\Lambda_\L,e(\Lambda|\F)r'-1,\b\otimes 1)$ of~$\theta_\L$. 
 Let~$\theta_\L'$ be the transfer of~$\theta_\L^{1}$ to
 \[\C(e(\Lambda|\F)\Lambda'_\L,e(\Lambda|\F)r'-1,\b\otimes 1)=\C(\Lambda'_\L,r'-1,\b\otimes 1),\]
 and we define~$\theta'$ as the restriction of~$\theta_\L'$ 
 to~$\H^{r'}(\b,\Lambda')=\H(\Delta'(1-))=\H^{r'}(\b\otimes 1,\Lambda'_\L)\cap\G$. 
 
\textbf{Step 3} (finding~$c'$):  
 Here we restrict  to~$\F$ and we write~$\La_\F$ if we consider~$\La$ as a lattice sequence over~$\F$. At first the map~$s_\F=s\otimes_\F \id_{\End_\A(\V)}$ from~$\A_\F=\A\otimes_\F\End_\A(\V)$ to~$\B_\F$ is a tame corestriction relative to~$\F[\b]/\F$ by definition~\cite[4.13]{skodlerack:17-1}. 
 By~\cite[5.4 after (H)]{stevens:05} there exists an element
 \[c'_\F\in\mf{a}_{-r}(\La_\F)\cap\mf{a}_{-r'}(\La'_\F)\cap\mathcal{M}_\F\]
 such that~$s_\F(c'_\F)=b'$.
 We take the average
 \[c'=\emptyset_{\langle\tau_l,\tau_{\varpi_\D}\rangle}(c'_\F)\in\mf{a}_{-r}(\La)\cap\mf{a}_{-r'}(\La')\cap\mathcal{M},\]
 see~\S\ref{subsecEndoInv}.
 
 \textbf{Step 4} (to show~$\th'\psi_{c'}\subseteq\pi$): We still follow the proof of~\cite[5.4 after (H)]{stevens:05}. 
 Define~$\delta:=c'-c$. By the proof of \loccit\ using~\cite[8.1.13,1.4.10]{bushnellKutzko:93} there exists an element
 $x_\F\in\mf{m}(\Delta_\F)\cap\mc{M}_\F$ such that
 \[\delta-a_{\b}(x_\F)\in\mf{a}_{-r}(\La_\F)\cap\mf{a}_{1-r'}(\La'_\F)\cap\mathcal{M}_\F.\]
 Take~$x:=\emptyset_{\langle\tau_l,\tau_{\varpi_\D}\rangle}(x_\F)$. Then~$x$ is an element of~$\mf{m}(\Delta)$ and therefore an element of $\mf{m}(\Delta\otimes\L)$. In particular~$\Cay(x)$ is an element of~$\I_{\H(\Delta\otimes\L)}(\theta_\L)$ and normalizes~$\theta_\L|_{\H(\Delta\otimes\L)}$. 
 Therefore by \cite[3.21]{stevens:05} we have on~$\H^r(\b\otimes 1, \La_\L)\cap\H^{r'}(\b\otimes 1,\La'_\L)$:
 \[\theta_\L^{\Cay(x)}=\theta_\L\psi_{\Cay(x)^{-1}\b\Cay(x)-\b}=\theta_\L\psi_{\Cay(x)^{-1}\delta}=\theta_\L\psi_{\delta},\]
which implies that
 $(\theta_\L\psi_c)^{\Cay(x)}$ and~$\th'_\L\psi_{c'}$ coincide 
there. The element~$\Cay(x)$ stabilizes~$\H^{r}(\b,\La)$ (\bfI5.39, because~$\Cay(x)$ stabilizes the stratum~$\Delta$) and also~$\K^r_2(\La)$ (because~$x\in\mathcal{M}\cap\mf{a}_1(\La)$). Fix an Iwahori decomposition with respect to~\eqref{eqassDecomp}. The characters~$\xi$ and~$\th'\psi_{c'}$ respect this Iwahori decomposition and are trivial on the lower and upper unipotent parts. Thus, as~$\Cay(x)\in\prod_{i\in\I}\A^{ii}$, $\xi$ and~$(\th'\psi_{c'})^{\Cay(x)^{-1}}$ coincide on the intersection of their domains. We still have
 \[\H^{r'}(\b,\La')\subseteq\P_{r'}(\La')\subseteq\K^r_2(\La)\]
 by assumption for the case~$s=2$.
 We apply Lemma~\ref{lemSubrepCriteria} to obtain
 \[\th'\psi_{c'}\subseteq\pi.\] 
\end{proof}

The theory of optimal points gives the following lemma.
\begin{lemma}[cf.\cite{stevens:02}~4.3,\cite{moyPrasad:94}~6.1]\label{lemOptimal}
\begin{enumerate}
\item\label{lemOptimalassi} Let~$\Lambda$ be a lattice sequence of $\D$-period~$e$, and~$m$ a positive integer such that~$\timfa_{-m}(\La)\neq\timfa_{-m+1}(\La)$. Then there is a lattice chain~$\Lambda'$
of~$\D$-period~$e'$ and an integer~$m'$ such that~$\frac{m'}{e'}\leq\frac{m}{e}$ and~$\timfa_{-m}(\La)\subseteq\timfa'_{-m'}(\La')$.
\item Let~$\Lambda$ be a self-dual lattice sequence of~$\D$-period~$e$ and~$m$ a positive integer such that~$\timfa_{-m}(\La)\neq\timfa_{-m+1}(\La)$. Then there is a self-dual 
lattice sequence~$\Lambda'$ of~$\D$-period~$e'$ smaller than~$2\dim_\D\V$ and an integer~$m'$ such that~$\frac{m'}{e'}\leq\frac{m}{e}$ and~$\timfa_{-m}(\La)\subseteq\timfa'_{-m'}(\La')$.
\end{enumerate}
\end{lemma}

The main point of the lemma is that~$e'$ is bounded. 
The skewfield~$\D$ does not play a role, i.e. the proof of~\ref{lemOptimal}\ref{lemOptimalassi} is the same for~$\F$ and~$\D$. We give here a very simple proof of the above lemma using a different idea than roots.  

\begin{proof}
 The second assertion follows directly from the first  by applying~$( )^\#$ and taking the barycentre. Without loss of generality we can assume that~$\frac{m}{e}$ is smaller than~$1$. 
 We reformulate the statement. 
 
 {\it We consider a point~$x$ in~$\mf{B}(\tiG)$ and~$t\in ]0,\frac{1}{d}[$. The point~$[x]$ of~$\mf{B}_{red}(\tiG)$ lies in the closure of a chamber~$\C$. 
 Then there is a midpoint~$[y]$ of a facet of~$\C$ such that~$\timfg_{x,-t}\subseteq\timfg_{y,-t}$.}
 
 For simplicity we assume~$d=1$, i.e. we prove the statement over~$\F$. (or one just rescales to get for~$\timfg_x$ the period~$1$.)
 For a lattice~$\M$ which occurs in the image of a lattice function~$\Gamma$ corresponding to~$x$ we set~$s_\M$ to be the maximum of all real~$s$ 
 such that~$\Gamma(s)=\M$. 
 We define the following sequence of real numbers
 \[s_0:=0,\ s_j:=s_{\Gamma(s_{j-1}-t)},\ j\geq 0.\]
 At first we observe that the sequence gets periodic mod~$\bbZ$, say the period is given by~$s_{j+1},\cdots,s_{j+e'}\equiv s_j$ mod~$\bbZ$. Let~$[y]\in \mf{B}_{red}(\tiG)$ be the barycentre of 
 the facet whose vertexes correspond to the homothety classes of the lattices~$\Gamma(s_{j+l}),\ l=1,\cdots, e'$. Note that these homothety classes differ pairwise.
 We write~$u$ for~$s_j-s_{j+e'}$, in particular~$t\geq \frac{u}{e'}$, and let~$\Gamma'$ be a lattice function corresponding to~$y$.  Let~$\Gamma''$ be the lattice function obtained from~$\Gamma$ in
 deleting  all lattices from~$\Gamma$ which are not in the image of~$\Gamma'$, i.e. if~$\Gamma(s)$ does not occur in the image of~$\Gamma'$ then we replace 
 ~$\Gamma(s)$ by~$\Gamma((s+v)+)$ where~$v$ is the smallest non-negative real number such that~$\Gamma((s+v)+)$ is in the image of~$\Gamma'$. 
 Then, for~$i\geq 0$,~$\Gamma''(]s_{j+i+1},s_{j+i}])$ contains exactly~$u$ lattices because there are no repetitions in the period. Indeed:, if~$t_1,t_2\in ]s_{j+i+1},s_{j+i}]$ satisfy~$\Gamma''(t_1)\supsetneqq\Gamma''(t_2)$ then~$\Gamma(s_{\Gamma''(t_1)})\supsetneqq\Gamma(s_{\Gamma''(t_2)})$ and~$s_{j+i+2}<s_{\Gamma(s_{\Gamma''(t_1)}-t)}<s_{\Gamma(s_{\Gamma''(t_2)}-t)}\leq s_{j+i+1}$. So we get injective maps:
  \[\Gamma''(]s_{j+1},s_{j}])\into\Gamma''(]s_{j+2},s_{j+1}])\ldots\into\Gamma''(]s_{j+e'+1},s_{j+e'}]),\]
  in particular they are all bijections. 
 So~$\timfg_{x,-t}\subseteq\timfg_{y,-\frac{u}{e'}}\subseteq\timfg_{y,-t}.$ 
\end{proof}

We obtain the following immediate corollary. 

\begin{corollary}\label{corNonFundSmallPeriod}
 We can find~$\La'$ in Corollary~\ref{corFundSemisimple}\ref{corFundSemisimpleii} and Corollary~\ref{corFundSemisimpletiG}\ref{corFundSemisimpletiGii} with~$\D$-period not greater than~$2\dim_\D\V$ and~$\dim_\D\V$, respectively. 
\end{corollary}

Now we are able to finish the proof of Theorem~\ref{thmStDuke5.1ForQuatCase}.

\begin{proof}[Proof of Theorem~\ref{thmStDuke5.1ForQuatCase}]
 The proof is similar to the first part of the argument after the proof of~\cite[5.5]{stevens:05}. Let~$z$ be the minimal element of~$\frac{1}{(4N)!}\mathbb{Z}$ ($N:=\dim_\F V$) 
 such that there is a self-dual semisimple character~$\theta\in\C(\Delta)$ contained in~$\pi$ with$\frac{r}{e(\Lambda|F)}\leq z$. We claim that~$z$ vanishes. 
 Assume for deriving a contradiction that~$z$ is positive. We extend~$\theta$ to~$\C(\Delta(1-))$ and call it again~$\theta$ and there is a~$c\in\mfa_{-r}$ such 
 that~$\theta\psi_c$ is contained in~$\pi$. The element~$c$ can be chosen in~$\prod_i \A^{ii}$ by~\loccit~5.2. Let~$s_\beta$ be a self-dual tame corestriction with respect to~$\beta$. Then the multi-stratum~$[\Lambda_\b,r,r-1,s_{\beta}(c)]$ has to be fundamental, i.e. at least one of the strata~$[\Lambda^i_\b,r,r-1,s_{\beta_i}(c_i)]$ 
 has to be fundamental, by the argument in the proof of~\loccit~5.5 using Proposition~\ref{propthetaInductionStep} and Corollary~\ref{corNonFundSmallPeriod} (instead of using~\cite[4.3]{stevens:02}). 
 Note further the latter stratum being fundamental also implies that~$\frac{r}{e(\Lambda|F)}$ is an element of~$\frac{1}{(4N)!}\mathbb{Z}$ by~\cite[2.11]{stevens:02} 
 (using~\cite[3.11]{secherreStevensIV:08}), i.e.~$\frac{r}{e(\Lambda|F)}=z$ by the choice of~$z$. 
 We apply~
 Corollary~\ref{corFundSemisimpletiG} and Corollary~\ref{corFundSemisimple} 
 to choose for every~$i\in \I_{0,+}$ a semisimple stratum~$[\Xi^i,r_i,r_i-1,\alpha_i]$ or~$[\Xi^i,r_i,r_i,\alpha_i=0]$ such that
 \begin{enumerate}
  \item the stratum is self-dual if~$i\in \I_0$,
  \item $s_{\beta_i}(c_i)+\timfa_{1-r}(\Lambda^i_\b)\subseteq  \alpha_i+\timfa_{1-r_i}(\Xi^i)$, for all~$i\in \I_0\cup \I_+$, and 
  \item $\frac{r}{e(\Lambda^i_\b|\E_i)}\geq\frac{r_i}{e(\Xi^i|\E_i)}$, for all~$i\in \I_{0,+}$, with equality if~$[\Lambda_\b^i,r,r-1,s_\beta(c_i)]$ is fundamental. 
 \end{enumerate}
 We take a self-dual~$\o_\E$-$\o_\D$-lattice sequence~$\Lambda'$ such that~$\Lambda'^i_\b$ is in the affine class  of~$\Xi^i$ for every~$i\in\I_{0,+}$, see~\bfII.5.3, and~$e(\La|\F)|e(\La'|\F)$.
 We put~$r':=\frac{e(\Lambda'|\F)r}{e(\Lambda|\F)}$, and we consider the multi-stratum~$[\Lambda'_\b,r',r'-1,s_{\beta}(c)]$. We have
 \[s_\beta(c)+\timfb_{1-r}(\Lambda)\subseteq s_\beta(c)+\timfb_{1-r'}(\Lambda'),~s_\beta(c)\in\timfb_{-r}(\Lambda)\cap\timfb_{-r'}(\Lambda').\]
 Thus, by Proposition~\ref{propthetaInductionStep}, there is a self-dual semisimple stratum~$\Delta'$ with $\beta'=\beta$ and a character~$\theta'\in\C(\Delta'(1-))$  and 
 an element~$c'\in\mfa'_{-r'}$ such that~$\theta'\psi_{c'}$ is contained in~$\pi$ and~$s_{\beta}(c')=s_\beta(c)$. 
 Now~$\Delta'$ is semisimple and~$[\Lambda'_\b,r',r'-1,s_{\beta}(c)]$ is equivalent to a semisimple multi-stratum. Then 
 $[\Lambda',n',r'-1,\beta'+c']$ is equivalent to a semisimple stratum by~\bfI.4.15. Further the self-duality of the stratum implies that it is equivalent 
 to a self-dual semisimple stratum, by~\bfII.4.7, say~$\Delta''=[\La',n',r''=r'-1,\b'']$. Then~$\C(\Delta'')=\C(\Delta'(1-))\psi_{c'}$. 
 Thus there is an element~$\theta''$ of~$\C(\Delta'')$ contained in~$\pi$ and
 \[\frac{r''}{e(\Lambda''|F)}=\frac{r'-1}{e(\Lambda'|F)}<\frac{r'}{e(\Lambda'|F)}=\frac{r}{e(\Lambda|F)}=z.\]
Note that on the other hand we could have started with~$\theta''$ and therefore~$\frac{r''}{e(\Lambda''|F)}=z$. A contradiction. 
\end{proof}
\section{Heisenberg representations}\label{secHeisenberg}
The study of Heisenberg representations and their extensions are the technical heart of Bushnell--Kutzko theory, for both: 
the construction of cuspidal representations and the exhaustion. 
We will review the results for~$\G_\L:=\G\otimes\L$ and extend them to~$\G$.
In this section we fix a self-dual semisimple stratum~$\Delta=[\La,n,0,\b]$. 
Let~$\La'$ be a self-dual~$\o_\E$-$\o_\D$-lattice sequence which satisfies~$\timfb(\La')\subseteq \timfb(\La)$. 
Let us recall that we have the following sequence of groups: 
\[\H^i_\La:=\H^i(\b,\La),\ \J^i_{\La}:=\J^i(\b,\La),\ i\in\bbN,\ \J_{\La}:=\J(\b,\La), \J^0_{\La}:=\J^1(\b,\La)\P^0(\La_\b)\]
and
\[\J^1_{\La',\La}:=\J^1_\La\P_1(\La'_\b),\ \J^0_{\La',\La}:=\J^1_\La\P^0(\La'_\b),\ \J_{\La',\La}:=\J^1_\La\P(\La'_\b).\]
We have similar subgroups~$\J^i_{\La'_\L,\La_\L}$ and~$\H^i_{\La_\L}$ of~$\G_\L$. 
We fix a character~$\theta\in\C(\Delta)$, and let~$\theta'$ be the transfer of~$\theta$ from~$\La$ to~$\La'$. We denote the~$\Gal(\L|\F)$-Glauberman lifts of~$\theta$ 
and~$\theta'$ by~$\theta_\L$ 
and~$\theta'_\L$. 

At first we recall the Heisenberg representations for~$\G_\L$. 
\begin{proposition}[\cite{stevens:05}~3.29,3.31,\cite{kurinczukStevens:19}~5.1]\label{propStev3.29and3.31}
 \begin{enumerate} 
 \item There is up to equivalence a unique irreducible representation~$(\eta_{\La_\L},\J^1_{\La_\L})$ which contains~$\theta_\L$. 
 \item Let~$g$ be an element of~$\G_\L$. The~$\CoefField$-dimension of~$\I_g(\eta_{\La_\L},\eta_{\La'_\L})$ is at most one, and it is one if and only 
 if~$g\in\J^1_{\La'_\L}(\G_\L)_\b \J^1_{\La_\L}$. 
 \end{enumerate}
\end{proposition}

We want to prove its analogue for~$\G$.
At first we need a lemma which allows us to apply Bushnell--Fr\"ohlichs' work to construct Heisenberg representations: 

\begin{lemma}[cf.~\cite{stevens:05}~3.28 for~$\G_\L$.]\label{lemNonDeg}
The form
\[\k_\theta:\J^1_\La/\H^1_\La\times \J_\La^1/\H^1_\La\rightarrow \CoefField\]
defined via~$\k_\theta(\bar{x},\bar{y}):=\theta([x,y])$ is non-degenerate. 
The pair~$(\J^1_\La/\H^1_\La,\k_\theta)$ is a subspace of~$(\J^1_{\La_\L}/\H^1_{\La_\L},\k_{\theta_\L})$ ($\k_{\theta_\L}$ similarly defined).
\end{lemma} 

The proof is similar to the proof of Lemma~\ref{lemHeisenberg}.

\begin{proposition}\label{propEtaLa}
  \begin{enumerate} 
 \item There is up to equivalence a unique irreducible representation~$\eta_{\La}$ of~$\J^1_\La$ which contains~$\theta$. Further~$\eta_\La$ has degree~$(\J^1_\La:\H^1_\La)^{\frac12}$. 
 \item The representation~$\eta_\La$ is the~$\Gal(\L|\F)$-Glauberman transfer of~$\eta_{\La_\L}$ to~$\J^1_\La$. 
 \item Let~$g$ be an element of~$\G$. The~$\CoefField$-dimension of~$\I_g(\eta_{\La},\eta_{\La'})$ is at most one, and it is one if and only if~$g\in\J^1_{\La'}\G_\b\J^1_\La$. 
 \end{enumerate}
\end{proposition}

\begin{proof}
 We define~$\eta_\La$ as~$\gl^{\L|\F}_\CoefField(\eta_{\La_\L})$, see Proposition~\ref{propStev3.29and3.31}. 
 The restriction of~$\eta_{\La_\L}$ to~$\H^1_\La$ is a multiple of~$\theta$. Thus the same is true for~$\eta_\La$. And thus by~\cite[8.1]{bushnellFroehlich:83} 
 and Lemma~\ref{lemNonDeg} up to equivalence~$\eta_\La$ is the unique irreducible representation of~$\J^1_\La$ which contains~$\theta$, and further it has the desired 
 degree. An element of~$\G_\b$ intertwines~$\eta_{\La_\L}$ with~$\eta_{\La'_\L}$ with intertwining space of dimension~$1$, so it intertwines~$\eta_\La$ with~$\eta_{\La'}$, by~\cite[2.4]{stevens:01-2}. 
 On the other hand we have
 \[\I_\G(\eta_\La,\eta_{\La'})\subseteq\I_\G(\theta,\theta')=\J_{\La'}^1\G_\b\J^1_\La,\]
 which finishes the proof of the intertwining formula. 
 The most complicated part is the proof of dimension one of the non-zero intertwining spaces. 
 For this we refer to the proof of~\cite[5.1]{kurinczukStevens:19}. Note that after taking~$\Gal(\L|\F)$-fixed points in the rectangular diagram of~\cite[5.2]{kurinczukStevens:19}
 the rows and columns remain still exact by the additive Hilbert 90. The rest of the proof is mutatis mutandis. 
  
\end{proof}

For the exhaustion the following extensions of~$\eta_\La$ are the key technical tools. We will emphasize the importance when their application arises. 
Note that we say that a representation~$(\tilde{\gamma},\tilde{\K})$ is an \emph{extension} of a representation~$(\gamma,\K)$ if 
$\K$ is a subgroup of~$\tilde{\K}$ and the restriction of~$\tilde{\gamma}$ to~$\K$ is equivalent to~$\gamma$. 

\begin{proposition}[\cite{stevens:08}~3.7,~\cite{kurinczukStevens:19}~5.6,~5.7]\label{propStev3.7}
 Suppose~$\ti{\mf{a}}(\La')\subseteq\ti{\mf{a}}(\La)$. 
 There is up to equivalence a unique irreducible representation~$(\eta_{\La'_\L,\La_\L},\J^1_{\La'_\L,\La_\L})$ which extends $(\eta_{\La_\L},\J^1_{\La_\L})$ 
 such that~$\eta_{\La'_\L,\La_\L}$ and~$\eta_{\La'_\L}$ induce equivalent irreducible representations on~$\P_1(\La'_\L)$.    
 Moreover the set of intertwining elements of~$\eta_{\La'_\L,\La_\L}$ in~$\G_\L$ is~$\J^1_{\La'_\L,\La_\L}(\G_\L)_\beta \J^1_{\La'_\L,\La_\L}$. 
 The intertwining spaces $\I_g(\eta_{\La'_\L,\La_\L})$ have all~$\CoefField$-dimension at most one. 
\end{proposition}

\begin{proposition}\label{propEtaLaLa}
Suppose~$\ti{\mf{a}}(\La')\subseteq\ti{\mf{a}}(\La)$. 
 There is up to equivalence a unique irreducible representation~$(\eta_{\La',\La},\J^1_{\La',\La})$ which extends~$(\eta_{\La},\J^1_{\La})$ 
 such that~$\eta_{\La',\La}$ and~$\eta_{\La'}$ induce equivalent irreducible representations on $\P_1(\La')$.     
 Moreover~$\eta_{\La',\La}$ is the~$\Gal(\L|\F)$-Glauberman transfer of~$\eta_{\La'_\L,\La_\L}$ to~$\J^1_{\La',\La}$ and the 
 set of intertwining elements of~$\eta_{\La',\La}$ in~$\G$ is~$\J^1_{\La',\La}\G_\beta\J^1_{\La',\La}$.
 The intertwining spaces~$\I_g(\eta_{\La',\La})$ have all~$\CoefField$-dimension at most one. 
\end{proposition}

\begin{proof}
 We set~$\eta_{\La',\La}$ to be~$\gl_\CoefField(\eta_{\La'_\L,\La_\L})$. Now~$\eta_{\La',\La}$ is the only irreducible representation of~$\J^1_{\La',\La}$ with an odd multiplicity in~$\eta_{\La'_\L,\La_\L}$ and therefore the irreducible
 constituents of~$\eta_{\La'_\L,\La_\L}|_{\J^1_{\La}}$ with odd multiplicity are contained in~$\eta_{\La',\La}$, i.e.~$\eta_\La$ is 
 contained in~$\eta_{\La',\La}$, knowing that the restriction of~$\eta_{\La'_\L,\La_\L}$ to~$\J^1_{\La_\L}$ is equivalent to~$\eta_{\La_\L}$. 
 The trace condition~\cite[(6)]{glauberman:68} implies for the characteristic zero case that the Glauberman transfers~$\gl_\bbC^{\L|\F}(\eta_{\La'_\L,\La_\L})$ and~$\gl_\bbC^{\L|\F}(\eta_{\La_\L})$ have the same degree.
 We have \[\Br_{\J^1_{\La'_\L,\La_\L}}(\eta_{\La'_\L,\La_\L,\CoefField})=\eta_{\La'_\L,\La_\L,\bbC}\text{  and }\Br_{\J^1_{\La'_\L,\La_\L}}(\eta_{\La'_\L,\CoefField})=\eta_{\La'_\L,\bbC},~\text{see~\S2}.\] Therefore, by Theorem~\ref{thmBr}, we have the equality of degrees of the Glauberman transfers in the modular case. Thus~$\eta_{\La',\La}$ is an extension of~$\eta_{\La}$.
 As in the proof of Proposition~\ref{propEtaLa} we obtain the formula for~$\I_\G(\eta_{\La',\La})$ using Proposition~\ref{propStev3.7} instead of Proposition~\ref{propStev3.29and3.31}.
 It remains to show the following three statements:
 \begin{enumerate}
  \item The representations~$\pi:=\ind_{\J^1_{\La',\La}}^{\P_1(\La')}\eta_{\La',\La}$ and~$\ind_{\J^1_{\La'}}^{\P_1(\La')}\eta_{\La'}$ are
  \begin{enumerate}
  \item irreducible, and
  \item equivalent.   
 \end{enumerate}
 \item The multiplicity of~$\eta_{\La}$ in~$\pi$ is one. 
 \item The intertwining spaces of~$\eta_{\La',\La}$ have at most $\CoefField$-dimension one. 
 \end{enumerate}
 The irreducibility follows from~$\I_{\P_1(\La')}(\eta_{\La',\La})=\J^1_{\La',\La}$ and~$\I_{\P_1(\La')}(\eta_{\La'})=\J^1_{\La'}$. 
 The statement about the intertwining spaces follows from Proposition~\ref{propEtaLa}.
 For the equivalence note at first that~$\eta_{\La'_\L,\La_\L}$ has 
 multiplicity one in\[\ind_{\J^1_{\La'_\L,\La_\L}}^{\P_1(\La'_\L)}\eta_{\La'_\L,\La_\L}\]  (The latter is irreducible, so apply second Frobenius reciprocity and Schur!),
 in particular~$\eta_{\La'_\L,\La_\L}$ has odd multiplicity in $\ind_{\J^1_{\La'_\L}}^{\P_1(\La'_\L)}\eta_{\La'_\L}$. 
 Thus
 \[\Res^{\P_1(\La')}_{\J^1_{\La',\La}}(\gl_\CoefField^{\L|\F}(\ind_{\J^1_{\La'_\L}}^{\P_1(\La'_\L)}\eta_{\La'_\L}))\supseteq \gl_\CoefField^{\L|\F}(\eta_{\La'_\L,\La_\L})=\eta_{\La',\La}.\]
 By irreducibility~$\ind_{\J^1_{\La',\La}}^{\P_1(\La')}\eta_{\La',\La}$ is equivalent to~$\gl_\CoefField^{\L|\F}(\ind_{\J^1_{\La'_\L}}^{\P_1(\La'_\L)}\eta_{\La'_\L})$. 
 Similarly, we obtain that $\gl_\CoefField^{\L|\F}(\ind_{\J^1_{\La'_\L}}^{\P_1(\La'_\L)}\eta_{\La'_\L})$ is also equivalent to $\ind_{\J^1_{\La'}}^{\P_1(\La')}\eta_{\La'}$. It remains to show the multiplicity assertion: Note that the set
 \[\Hom_{\J^1_\La\cap ^g\J^1_{\La',\La}}(\eta_\La,\eta_{\La',\La}^g)\]
 is trivial if~$g\not\in\I_{\G}(\eta_{\La})$. Thus by Frobenius reciprocity and Mackey theory we have
 \[\Hom_{\J^1_\La}(\eta_\La,\ind_{\J^1_{\La',\La}}^{\P_1(\La')}\eta_{\La',\La})=\Hom_{\J^1_\La}(\eta_\La,\eta_\La)=\CoefField.\]
 This finishes the proof.
\end{proof}

\begin{remark}\label{remGlaubetaLaprimeLa}
 The proof also shows that~$\gl_\CoefField^{\L|\F}$ maps the class of~$\eta_{\La'_\L,\La_\L}$ to the class of~$\eta_{\La',\La}$.
\end{remark}

We need to show that the definition of~$\eta_{\La',\La}$ only depends on~$\timfb(\La')$ instead of~$\La'$. 

\begin{proposition}\label{propetaIndLa}
 Let~$\La''$ be a self-dual~$\o_\E$-$\o_\D$-lattice sequence such that $\timfb(\La'')=\timfb(\La')$ and~$\ti{\mf{a}}(\La'')\subseteq\ti{\mf{a}}(\La)$
 and suppose~$\ti{\mf{a}}(\La')\subseteq\ti{\mf{a}}(\La)$. 
 Then~$\J^1_{\La',\La}=\J^1_{\La'',\La}$ and~$\eta_{\La'',\La}$ is equivalent to~$\eta_{\La',\La}$. 
\end{proposition}

\begin{proof}
 We consider a path of self-dual~$\o_\D$-$\o_\E$ lattice sequences $\La'=\La_0,$ $\La_1,\ldots,\La_l=\La''$ on a segment from~$\La'$ to~$\La''$ in~$\mf{B}(\G)$, such that
 \[\ti{\mf{a}}(\La_i)\cap\ti{\mf{a}}(\La_{i+1})\in\{\ti{\mf{a}}(\La_i),\ti{\mf{a}}(\La_{i+1})\},\]
 for all~$i=0,1,\ldots,l-1$, in particular we have~$\timfb(\La_i)=\timfb(\La')$ for all~$i=0,\ldots,l$. 
 Thus by transitivity it is enough to consider the case~$\ti{\mf{a}}(\La')\supseteq\ti{\mf{a}}(\La'')$. 
 The representations~$\ind_{\J^1_{\La',\La}}^{\P_1(\La')}\eta_{\La',\La}$ and~$\ind_{\J^1_{\La'}}^{\P_1(\La')}\eta_{\La'}$ 
 are equivalent, and thus~$\ind_{\J^1_{\La',\La}}^{\P_1(\La'')}\eta_{\La',\La}$ is equivalent to~$\ind_{\J^1_{\La'}}^{\P_1(\La'')}\eta_{\La'}$.
 Now~$\J^1_{\La',\La}=\J^1_{\La'',\La}$,~$\J^1_{\La'}=\J^1_{\La'',\La'}$,~$\eta_{\La'',\La'}=\eta_{\La'}$ and
 \[\ind_{\J^1_{\La'}}^{\P_1(\La'')}\eta_{\La'}\cong\ind_{\J^1_{\La''}}^{\P_1(\La'')}\eta_{\La''}\]
 by definition of~$\eta_{\La'',\La'}$. Thus~$\eta_{\La',\La}$ and~$\eta_{\La'',\La}$ are equivalent by Proposition~\ref{propEtaLaLa}. 
\end{proof}

By last proposition we can now define~$\eta_{\La',\La}$ without assuming~$\ti{\mf{a}}(\La')\subseteq\ti{\mf{a}}(\La)$. 
\begin{definition}\label{defEta}
 Granted~$\timfb(\La')\subseteq\timfb(\La)$, we choose~$(\eta_{\La',\La},\J^1_{\La',\La})$ in the isomorphism class of~$(\eta_{\La'',\La},\J^1_{\La'',\La})$,
 where~$\La''$ is a self-dual~$\o_\E$-$\o_\D$-lattice sequence such that~$\timfb(\La')=\timfb(\La'')$ and~$\ti{\mf{a}}(\La'')\subseteq\ti{\mf{a}}(\La)$. 
\end{definition}


\begin{corollary}[cf.~\cite{stevens:08}~3.8]\label{corResetaLaLaPrime}
 Let~$\La''$ be a further self-dual~$\o_\E$-$\o_\D$-lattice sequence such that~$\timfb(\La'')\subseteq\timfb(\La')$. 
 Then the restriction of~$\eta_{\La'',\La}$ to~$\J^1_{\La',\La}$ is equivalent to~$\eta_{\La',\La}$.
\end{corollary}

\begin{proof}
It suffices to consider the complex case, by~\eqref{eqRestriction} and Theorem~\ref{thmBr}. So assume~$\CoefField=\bbC$.
 The result now follows from~\cite[Proposition~3.8]{stevens:08} and the Glauberman correspondence, indeed
 \[\eta_{\La''_\L,\La_\L}|_{\J^1_{\La'_\L,\La_\L}}\cong\eta_{\La'_\L,\La_\L},\]
 by~\cite[3.8]{stevens:08} and~$\gl_\bbC^{\L|\F}(\eta_{\La'_\L,\La_\L})=\eta_{\La',\La}$, and thus the latter representation occurs with odd multiplicity 
 in~$\eta_{\La''_\L,\La_\L}|_{\J^1_{\La',\La}}$, thus by Remark~\ref{remGlaubetaLaprimeLa}~$\eta_{\La'',\La}$ contains~$\eta_{\La',\La}$ and hence, as they share the degree with~$\eta_\La$, we get the result. 
\end{proof}

\subsection{Heisenberg representations for general linear groups}
The construction of Heisenberg representations for~$\tilde{\G}$ is similar to the construction for~$\G$. 
So we just state the result we need later. We fix a semisimple stratum~$\Delta=[\Lambda,n,0,\beta]$ and we need to consider the groups 
\[\tilde{\H}^1_\Lambda:=\tilde{\H}^1(\beta,\Lambda), \ \tilde{\J}^1_\Lambda:=\tilde{\J}^1(\beta,\Lambda).\]

\begin{proposition}\cite{secherreI:04}\label{propHeisGL}
 Let~$\tilde{\theta}$ be an element of~$\tilde{\C}(\Delta)$. 
 Then there is up to isomorphism a unique representation~$\tilde{\eta}$ on~$\tilde{\J}^1_\Lambda$ containing~$\tilde{\theta}$. 
\end{proposition}

\section{The isotypic components}
In general proofs on smooth complex representations of locally totally disconnected groups  cannot be easily  transferred to the modular case. One trick used for the case of cuspidal representations of~$\G\otimes\L$  in~\cite{kurinczukStevens:19} is the following lemma:

\begin{lemma}\label{lemIsotypic}
 Let~$\H$ be locally compact and totally disconnected topological group and~$\K,\K^1$ be compact open subgroups of~$\H$ such that~$\K^1$ is a normal pro-$p$ subgroup of~$\K$. Let further~$\pi$ be a smooth $\CoefField$-representation of~$\H$, and~$\eta$ be a smooth~$\CoefField$-representation of~$\K^1$ such that~$\eta$ is normalized by~$\K$.  Then we have
 \[\pi|_\K=\pi^\eta\oplus\pi(\eta)\]
 (a direct sum of~$\K$-subrepresentations), where~$\pi^\eta$ is the~$\eta$-isotypic component of~$\pi$ and~$\pi(\eta)$ the largest subrepresentation of~$\pi|_{\K^1}$ which does not contain a copy of~$\eta$.
\end{lemma}

This lemma, which is trivial using the fact that~$\pi|_{\K^1}$ is semisimple, is still very useful. 

\section{$\b$-extension}\label{secBetaExt}
In this section we generalize~$\b$-extensions to~$\G$, see \cite[\S4]{stevens:08} for the case of~$\G\otimes \L$ (see also~\cite{secherreII:05},~\cite[5.2.1]{bushnellKutzko:93} for~$\tiG$). Its construction for classical groups is a complicated process.
 We fix a self-dual semisimple stratum~$[\La,n,0,\b]$ and a self-dual semisimple character~$\th\in\C(\La,0,\b)$. 
 We fix self-dual~$\o_\E$-$\o_\D$ lattice sequences~$\La_m,\La,\La',\La_\M$ and~$\La''$ such that~$\timfb(\La_m)$ (resp.~$\timfb(\La_\M)$) is 
 minimal (resp. maximal) and 
 \[\timfb(\La_m)\subseteq\timfb(\La)\subseteq\timfb(\La')\subseteq\timfb(\La_\M),\]
 and such that~$\timfb(\La)\subseteq\timfb(\La'')$. So we have
\[\timfb(\La)\subseteq\timfb(\La')\cap\timfb(\La'').\]
We are going to use the representations~$\eta_\La$,~$\eta_{\La'}$,~$\eta_{\La,\La'}$, etc.~constructed in~\S\ref{secHeisenberg}.
Recall that~$\eta_{\Lambda'}$ is the Heisenberg representation of the transfer of~$\th$ to~$C(\La',0,\b)$.

We call~$\La_\M$ and~$\La_m$ a~\emph{maximal}, respectively a~\emph{minimal},~$\o_\E$-$\o_\D$-lattice sequence according to the fact that~$\timfb(\La_\M)$ is maximal and~$\timfb(\La_m)$
is minimal with respect to inclusion.

\subsection{General idea}\label{subsecIntrobExt} 
We present here Stevens' strategy from~\cite[\S4]{stevens:08}:
We put
\[\ext(\La,\La'):=\{(\kappa'_\cong,\J_{\La,\La'})\in\mf{R}_\CoefField(\J_{\La,\La'}) |\ \kappa'|_{\J^1_{\La,\La'}}\cong \eta_{\La,\La'}\},~\ext(\La'):=\ext(\La',\La'),\]
where the subscript~$\cong$ indicates the isomorphism class of the representation in question.
Depending on~$\La_\M$ we only choose certain extensions of~$\eta_{\La'}$ to~$\J_{\La'}$, i.e. a subset~$\bext_{\La_\M}(\La')$ of~$\ext(\La')$, and call them~$\b$-extensions with respect to~$\La_\M$.

\begin{enumerate}
 \item For~$\La_\M$ the set~$\bext(\La_\M)$ is defined to consist of those elements of~$\ext(\La_\M)$ which are mapped into~$\ext(\La_m,\La_\M)$ under restriction to~$\J_{\La_m,\La_\M}$.
\item For~$\La'$ we construct a bijection~$\Psi$ from~$\ext(\La',\La_\M)$ to~$\ext(\La')$, see below, and the set~$\bext_{\La_\M}(\La')$ is then defined to be the image of the composition
\[\bext(\La_\M)\stackrel{\Res_{\J_{\La',\La_\M}}^{\J_{\La_\M}}}{\longrightarrow}\ext(\La',\La_\M)\stackrel{\Psi}{\simarrow}\ext(\La').\]
\end{enumerate}
We define a map
\[\Psi_{\La,\La',\La''}:\ \ext(\La,\La')\rightarrow\ext(\La,\La'')\]
as follows (to get~$\Psi$ in (ii) substitute~$(\La,\La',\La'')$ by~$(\La',\La_\M,\La')$ ):  
\begin{itemize}
 \item Consider a path of self-dual~$\o_\E$-$\o_\D$ lattice sequences
 \begin{equation}\label{eqPath}\La'=\La_0,\La_1,\cdots,\La_l=\La''\end{equation}
 such that  
 \begin{equation}\label{eqIncl}
  \ti{\mf{a}}(\La_i)\cap\ti{\mf{a}}(\La_{i+1})\in\{\ti{\mf{a}}(\La_i),\ti{\mf{a}}(\La_{i+1})\},\ \timfb(\La)\subseteq\timfb(\La_i)\cap\timfb(\La_{i+1}).
 \end{equation}
 for all indexes~$i\in\{0,\cdots,l-1\}$.
 \item Define the maps~$\Psi_{\La,\La_i,\La_{i+1}}$ in requiring isomorphic inductions to~$\P(\La_\b)\P_1(\La)$, 
 \item and then put
 \begin{equation}\label{eqDefPsi}
 \Psi_{\La,\La',\La''}:=\Psi_{\La,\La_{l-1},\La_l}\circ\Psi_{\La,\La_{l-2},\La_{l-1}}\circ\cdots\circ\Psi_{\La,\La_{0},\La_{1}}.
 \end{equation}
\end{itemize}
We now give the details to those steps, beginning with the maximal case. 

\subsection{The existence of~$\b$-extensions for the maximal compact case}
We are interested in  extensions of~$\eta_\La$ to~$\J_\La$, but not all of them (cf.~\cite[Remark~4.2]{stevens:08}). 
At first we define~$\b$-extensions for~$\La_\M$. 
We define~$\bext(\La_\M)$ as in~\S\ref{subsecIntrobExt}(i), i.e. as the set of all isomorphism classes of irreducible representations~$\kappa$ of~$\J_{\La_\M}$ such that the restriction of~$\kappa$ to~$\J^1_{\La_m,\La_\M}$ is isomorphic to~$\eta_{\La_m,\La_\M}$.

The following proposition shows that~$\bext(\La_\M)$ is non-empty. Note that~$\J^1_{\La_m,\La}$ is a pro-$p$-Sylow subgroup of~$\J_{\La}$ and every pro-$p$-Sylow subgroup 
of~$\J_\La$ is of such a form, i.e. for an appropriate~$\La_m$, and they are all conjugate in~$\J_\La$. 

\begin{proposition}\label{propExisteceOfbextLaM} Granted~$\timfb(\La_m)\subseteq\timfb(\La)\subseteq\timfb(\La')$:
\begin{enumerate}
  \item \label{propExisteceOfbextLaM-i} There exists an extension~$(\kappa,\J_\La)$ of~$(\eta_{\La_m,\La},\J^1_{\La_m,\La})$.
  \item\label{propExisteceOfbextLaM-ii} Let~$(\kappa',\J_{\La'})$ be an extension of~$(\eta_{\La_m,\La'},\J^1_{\La_m,\La'})$.
  Then the restriction of~$\kappa'$ to~$\J^1_{\La,\La'}$ is equivalent to~$(\eta_{\La,\La'},\J^1_{\La,\La'})$. 
  \item \label{propExisteceOfbextLaM-iii} Let~$(\kappa,\J_{\La_\M})$ be an extension of~$(\eta_{\La_\M},\J^1_{\La_\M})$. Then are equivalent:
  \begin{enumerate}
  \item\label{propExisteceOfbextLaM-iii(a)} $\kappa\in\bext(\La_\M)$. 
  \item\label{propExisteceOfbextLaM-iii(b)} $\kappa$ is an extension of~$\eta_{\La_\M}$ such that for every pro-$p$-Sylow subgroup~$S$ of~$\J_{\La_\M}$ the restriction of~$\kappa$ to~$S$ is intertwined by the whole of~$\G_\b$.
  \end{enumerate}
 \end{enumerate}
\end{proposition}

For the proof we need a lemma:
\begin{lemma}[\cite{bushnellKutzko:93}~(5.3.2)(proof), cf.~\cite{kurinczukStevens:19}~2.7]\label{lemTensProdOffEnd}
 Let~$\K$ be a totally disconnected and locally compact group and let~$(\rho_i,\W_i),\ i=1,2,$ be two smooth representations of~$\K$. Suppose~$\K_2$ is a normal open subgroup of~$\K$ contained in the kernel of~$\rho_2$. Suppose that the sets~$\End_\K(\W_1)$ 
 and~$\End_{\K_2}(\W_1)$ coincide. Then: 
 \[\End_\K(\W_1\otimes_\CoefField\W_2)\cong\End_\K(\W_1)\otimes_\CoefField \End_\K(W_2).\]
 In particular if~$\K$ is compact and~$\K_2$ is pro-finite with~$p_\CoefField$ not dividing the pro-order of~$\K_2$  we get:
 \begin{enumerate}
  \item\label{lemTensProdOffEndParti} $\W_1\otimes_\CoefField\W_2$ is irreducible if and only if~$\W_1$ and~$\W_2$ are irreducible. 
  \item\label{lemTensProdOffEndPartii} Suppose~$\rho_1$ is irreducible and let~$\rho$ be an irreducible representation of~$\K$ such that~$\rho|_{\K_2}$ is isomorphic to a direct sum of copies of~$\rho_1|_{\K_2}$. Then there is an irreducible representation~$\rho'_2$ on~$\K$ containing~$\K_2$ in its kernel such that 
  $\rho$ is equivalent to~$\rho_1\otimes\rho'_2$. 
 \end{enumerate}
\end{lemma}

\begin{proof}
 Take a~$\CoefField$-basis~$f_i$ of~$\End_\CoefField(\W_2)$. We have the~$\K$-action on~$\End_{\CoefField}(\W_1\otimes\W_2)$ via conjugation:~$k.\Phi:=k\circ\Phi\circ k^{-1}$,
 where we consider on~$\W_1\otimes\W_2$ the diagonal action of~$\K$. Then every element~$\Phi=\sum_ig_i\otimes f_i$ of~$\End_\K(\W_1\otimes_\CoefField\W_2)$  is fixed by~$\K_2$
 and therefore~$g_i$ has to be~$\K_2$-equivariant and therefore~$\K$-equivariant by assumption. So~$\Phi$ is an element of 
 $\End_\K(\W_1\otimes_\CoefField\W_2)\cap (\End_\K(\W_1)\otimes_\CoefField \End_\CoefField (W_2))$. Now a similar argument for~$\Phi$ using a~$\CoefField$-basis of~$\End_\K(W_1)$
 shows that~$\Phi$ is an element of~$\End_\K(\W_1)\otimes_\CoefField \End_\K(W_2)$.  
%
Now~$\ref{lemTensProdOffEndPartii}$ follows from~\ref{lemTensProdOffEndParti} because, by~\ref{lemTensProdOffEndParti},~$\W_1\otimes\ind_{K_2}^\K 1$ has a Jordan--H\"{o}lder composition series where all the factors are of the form~$\W_1\otimes_\CoefField\W_2$, $\W_2$ depending on the factor.
 So it remains to show~\ref{lemTensProdOffEndParti}. So, suppose~$\W_1$ and~$\W_2$ are irreducible~$\CoefField$-representations of~$\K$. Then, by
 \[\End_{\K_2}(\W_1)=\End_\K(\W_1)=\CoefField\]
and the pro-finiteness of~$\K_2$ with~$p_\CoefField$ not dividing the pro-order,~$\W_1|_{\K_2}$ is irreducible too. Let~$\ti{\W}$ be a non-zero subrepresentation of~$\W_1\otimes_\CoefField\W_2$ and 
\[\sum_{i=1}^u w_i^{(1)}\otimes_\CoefField w_i^{(2)}\]
a non-zero sum of elementary tensors contained in~$\ti{\W}$. Let~$u$ be minimal. For every pair~$(i_1,i_2)$ of  indexes and any finite tuple~$(k_j)_j$ of~$\K_2$ we have 
\[\sum_jk_jw_{i_1}^{(1)}=0\text{ if and only if }\sum_jk_jw_{i_2}^{(1)}=0\]
by the minimality of~$u$. Thus there is a~$\K_2$-isomorphism of~$\W_1$ which maps~$w^{(1)}_{i_1}$ to~$w^{(1)}_{i_2}$ and therefore~$w^{(1)}_{i_1}$ and~$w^{(1)}_{i_2}$ are linearly dependent, as~$\End_{\K_2}(\W_1)=\CoefField$, and the minimality of~$u$ implies~$i_1=i_2$. Thus~$u=1$,~$\W_1\otimes w_1^{(2)}$ is contained in~$\tilde{\W}$ (because~$\W_1$ is irreducible over~$\K_2$), and thus~$\W_1\otimes\W_2\subseteq\tilde{\W}$. 
\end{proof}

\begin{proof}[Proof of Proposition~\ref{propExisteceOfbextLaM}]
 The existence assertion~\ref{propExisteceOfbextLaM-i} is proven mutatis mutandis to~\cite[Theorem~4.1]{stevens:08}. 
 Assertion~\ref{propExisteceOfbextLaM-ii} follows from Corollary~\ref{corResetaLaLaPrime}. 
 For~\ref{propExisteceOfbextLaM-iii}:  Let~$\La_m$ be a self-dual~$\o_\E$-$\o_\D$-lattice sequence such that~$\timfb(\La_m)$ is minimal 
 and~$\ti{\mf{a}}(\La_m)\subseteq\ti{\mf{a}}(\La_\M)$. The representation~$\eta_{\La_m,\La_\M}$ is intertwined by the whole of~$\G_\b$, by Proposition~\ref{propEtaLaLa}, 
 and further the pro-$p$-Sylow subgroups of~$\J_{\La_\M}$ are all conjugate in~$\P(\La_{\M,\b})$. Thus~\ref{propExisteceOfbextLaM-iii(a)} implies~\ref{propExisteceOfbextLaM-iii(b)}.
 Suppose~\ref{propExisteceOfbextLaM-iii(b)} then, by Lemma~\ref{lemTensProdOffEnd}\ref{lemTensProdOffEndPartii}~$\kappa|_{\J^1_{\La_m,\La_\M}}$ is equivalent to~$\eta_{\La_m,\La_\M}\otimes\varphi$ for some inflation~$\varphi$ of a  characters of~$\J^1_{\La_m,\La_\M}/\J^1_{\La_\M}$ (The latter group is isomorphic to~$\P_1(\La_{m,\b})/\P_1(\La_{\M,\b})$).
 Thus~$\varphi$ is intertwined by the whole of~$\G_\b$ and by the analogue of~\cite[3.10]{stevens:08} we obtain that~$\varphi$ is trivial. 
\end{proof}

\begin{corollary}
 The sets~$\ext(\La,\La')$ and~$\b$-$\ext(\La_\M)$ are not empty, and~$\b$-$\ext(\La_\M)$ does not depend on the choice of~$\La_m$.
\end{corollary}

\subsection{The map~$\Psi_{\La,\La',\La''}$ in the inclusion case}
At first we assume~$\ti{\mf{a}}(\La)\subseteq\ti{\mf{a}}(\La')\cap\ti{\mf{a}}(\La'')\in\{\ti{\mf{a}}(\La'),\ti{\mf{a}}(\La'')\}$. Take~$\kappa'_\cong\in\ext(\La,\La')$. 

\begin{lemma}[cf.~\cite{stevens:08}~4.3]\label{lemMutMutStev4.3}
There is a unique~$(\kappa''_\cong,\J_{\La,\La''})\in\ext(\La,\La'')$ such that
\begin{equation}\label{eqk'k''}
\ind_{\J_{\La,\La'}}^{\P_{\La,\La}}\kappa'\cong\ind_{\J_{\La,\La''}}^{\P_{\La,\La}}\kappa''.
\end{equation}
(Where we define~$\P_{\La,\La'}:=\P(\La_\b)\P_1(\La')$, e.g.~$\P_{\La,\La}:=\P(\La_\b)\P_1(\La)$.)
\end{lemma}


\begin{proof}
The proof follows the idea of~\cite[Lemma~6.2]{kurinczukStevens:19} (\cite[5.2.5]{bushnellKutzko:93} and~\cite[4.3]{stevens:08}). By transitivity we only need to proof the assertion for the cases~$(\La,\La',\La'')=(\La,\La',\La)$ 
 and~$(\La,\La',\La'')=(\La,\La,\La'')$. 
 We just consider the first case. We denote by~$\pi$ the left hand side of~\eqref{eqk'k''}. We have to find~$\kappa_{\cong}\in\ext(\La,\La)$ such that 
~\eqref{eqk'k''} holds (for~$\kappa''=\kappa$), and it will be unique, because~$\eta_\La$ has multiplicity one in~$\pi$, because the restriction~$\pi|_{\P_1(\La)}$, which is~$\ind_{\J^1_{\La,\La'}}^{\P_1(\La)}\eta_{\La,\La'}$ is irreducible and isomorphic to~$\ind_{\J^1_{\La}}^{\P_1(\La)}\eta_{\La}$, by Proposition~\ref{propEtaLaLa}. We choose~$\kappa=\pi^{\eta_\La}$, see Lemma~\ref{lemIsotypic}. We obtain~\eqref{eqk'k''} for~$\kappa''=\kappa$ because~$\ind_{\J_\La}^{\P_{\La,\La}}\kappa$ is irreducible by an argument similar to the one given at the beginning of the proof. 
\end{proof}

We define~$\Psi_{\La,\La',\La''}(\kappa'_\cong):=\kappa''_\cong$ using~$\kappa''_\cong$ from Lemma~\ref{lemMutMutStev4.3}.
In fact~$\J_{\La,\La'}$ does only depend on~$\timfb(\La)$ instead of~$\La$, and even more:

\begin{lemma}\label{lemIndepOfLa}
Suppose~$\ti{\La}\in\Latt_{\o_\E,\o_\D}^1(\V)$. Then~$\Psi_{\La,\La',\La''}\circ\Res_{\J_{\La,\La'}}^{\J_{\ti{\La},\La'}}$ and~$\Res_{\J_{\La,\La''}}^{\J_{\ti{\La},\La''}}\circ\Psi_{\ti{\La},\La',\La''}$ coincide if~$\timfb(\La)\subseteq\timfb(\ti{\La})$ and~$\ti{\mf{a}}(\ti{\La})\subseteq\ti{\mf{a}}(\La')\cap\ti{\mf{a}}(\La'')$. 
\end{lemma}

\begin{proof}
 To show this assertion it is enough to consider the case~$\ti{\mf{a}}(\La)\subseteq\ti{\mf{a}}(\tilde{\La})$. 
 (For the general case take~$\bar{\La}\in\Latt_{\o_\E,\o_\D}\V$ with~$\ti{\mf{a}}(\bar{\La})\subseteq\timfa(\ti{\La})$ and~$\timfb(\La)=\timfb(\bar{\La})$, and 
 use a segment from~$\La$  
 to~$\bar{\La}$.) 
 We start with~$(\kappa',\kappa'')$ satisfying~\eqref{eqk'k''} for~$\ti{\La}$ instead of~$\La$. Then we restrict to~$\P_{\La,\ti{\La}}$ and 
 induce to~$\P_{\La,\La}$ to obtain~\eqref{eqk'k''} for~$\La$. This proves the lemma. 
\end{proof}

By Lemma~\ref{lemIndepOfLa} we can define~$\Psi_{\La,\La',\La''}$ if~$\ti{\mf{a}}(\Lambda)$ may not be contained in $\ti{\mf{a}}(\La')\cap\ti{\mf{a}}(\La'')$.
Suppose~$\timfb(\La)\subseteq\ti{\mf{a}}(\La')\cap\ti{\mf{a}}(\La'')$ and choose~$\ti{\La}$ such that~$\ti{\mf{a}}(\ti{\La})\subseteq\ti{\mf{a}}(\La')\cap\ti{\mf{a}}(\La'')$ and~$\timfb(\La)=\timfb(\ti{\La})$. Then define~$\Psi_{\La,\La',\La''}$
to be~$\Psi_{\ti{\La},\La',\La''}$.

\subsection{The map~$\Psi_{\La,\La',\La''}$ in the general case}\label{secGeneralPsi}

We do not require~$\ti{\mf{a}}(\La')\cap\ti{\mf{a}}(\La'')\in\{\ti{\mf{a}}(\La'),\ti{\mf{a}}(\La'')\}$ here. 
We choose a path~\eqref{eqPath} of self-dual~$\o_\E$-$\o_\D$-lattice sequences and define~$\Psi_{\La,\La',\La''}$ as in~\eqref{eqDefPsi}. 
Now one has to prove that this definition is independent of the choice of the path. For that it is enough to consider a triangle of self-dual~$\o_\E$-$\o_\D$
lattice sequences~$\La_1,\La_2,\La_3$ such that~$\ti{\mf{a}}(\La_1)\subseteq\ti{\mf{a}}(\La_2)\subseteq \ti{\mf{a}}(\La_3)$ with~$\timfb(\La)\subseteq\timfb(\La_1)$ and 
show the commutativity
\[\Psi_{\La,\La_2,\La_3}\circ\Psi_{\La,\La_1,\La_2}=\Psi_{\La,\La_1,\La_3}.\]
We take an~$\o_\E$-$\o_\D$-lattice sequence~$\ti{\La}$ such that~$\ti{\mf{a}}(\ti{\La})\subseteq\ti{\mf{a}}(\La_1)$ and~$\timfb(\La)=\timfb(\ti{\La})$. 
We choose~$\kappa_{i,\cong}\in\ext(\La,\La_i)$,~$i=1,2,3$, such that $\Psi_{\La,\La_1,\La_2}(\kappa_{1,\cong})=\kappa_{2,\cong}$ 
and~$\Psi_{\La,\La_1,\La_3}(\kappa_{1,\cong})=\kappa_{3,\cong}$. Then $\Psi_{\La,\La_2,\La_3}(\kappa_{2,\cong})=\kappa_{3,\cong}$ follows from~\eqref{eqk'k''} and transitivity. 
This finishes the definition of~$\Psi_{\La,\La',\La''}$. We have the following result on intertwining:

\begin{proposition}[cf.~\cite{stevens:08}~Lemma~4.3]\label{propIntPsi}
 Suppose~$\Psi_{\La,\La',\La''}(\kappa'_\cong)=\kappa''_\cong$. Then $\I_{\G_\b}(\kappa')=\I_{\G_\b}(\kappa'')$.
\end{proposition}

For this proposition we need the following intersection property:
\begin{lemma}[cf.\cite{stevens:08}~2.6]\label{lemIntProp}
 Let~$g$ be an element of~$\G_\b$. Then we have 
 \begin{equation}\label{eqIntProperty}
 (\P_1(\La)g\P_1(\La))\cap \G_\b=\P_1(\La_\b)g\P_1(\La_\b).
 \end{equation}
\end{lemma}

\begin{proof}
At first: The proof of 
\begin{equation}\label{eqIntPropTilde}(\ti{\P}_1(\La)g\ti{\P}_1(\La))\cap\ti{\G}_\b=\ti{\P}_1(\La_\b)g\ti{\P}_1(\La_\b)\end{equation}
is mutatis mutandis to the proof of~\bfII.4.8. 
Now one takes~$\sigma$-fixed points on both sides of~\eqref{eqIntPropTilde} to obtain~\eqref{eqIntProperty} by~\cite[2.12(i)]{kurinczukStevens:19}.
\end{proof}

\begin{proof}[Proof of Proposition~\ref{propIntPsi}]
 Using the construction of~$\Psi_{\La,\La',\La''}$ we can assume without loss of generality that~$\ti{\mf{a}}(\La)\subseteq\ti{\mf{a}}(\La')\subseteq\ti{\mf{a}}(\La'')$. 
 Now the proof is as for the second part of~\cite[4.3]{stevens:08}  (see~\cite[2.9]{secherreII:05}), using Lemma~\ref{lemIntProp} instead of~\cite[2.6]{stevens:08}. 
\end{proof}

\subsection{Defining~$\b$-extensions in the general case}
\label{secGeneralbetaExt}
Suppose further that~$\timfb(\La'')$ is contained in~$\timfb(\La_\M)$ in this paragraph. 


\begin{definition}[cf.\cite{stevens:08}~4.5,~4.7]
Granted~$\timfb(\La)\subseteq\timfb(\La_\M)$, we call the following set the~\emph{set of (equivalence classes) of~$\b$-extensions of~$\eta_\La$ to~$\J_\La$ relative to~$\La_\M$}:
\[\bext_{\La_\M}(\La):=\{\Psi_{\La,\La_\M,\La}(\Res^{\J_{\La_\M}}_{\J_{\La,\La_\M}}\kappa_\cong)|\ \kappa_\cong\in\bext(\La_\M)\}.\]
We call the elements of
\[\bext^0_{\La_\M}(\La):=\Res^{\J_{\La}}_{\J^0_\La}(\bext_{\La_\M}(\La))\]
(equivalence classes) of~\emph{$\b$-extensions of~$\eta_\La$ to~$\J^0_\La$ relative to~$\La_\M$}.
\end{definition}

\begin{theorem}[\cite{stevens:08}~4.10]\label{thmbetaext}
 Granted~$\timfb(\La')\cup\timfb(\La'')\subseteq\timfb(\La_\M)$, there is a unique map $\Psi_{\La',\La''}^0$ from~$\bext^0_{\La_\M}(\La')$ to~$\bext^0_{\La_\M}(\La'')$  depending on~$\La_\M$ such that 
 \begin{equation}\label{eqDefPsi0}
\Psi_{\La',\La''}^0\circ\Res^{\J_{\La'}}_{\J^0_{\La'}}\circ\Psi_{\La',\La_\M,\La'}\circ\Res^{\J_{\La_\M}}_{\J_{\La',\La_\M}}=
 \Res^{\J_{\La''}}_{\J^0_{\La''}}\circ\Psi_{\La'',\La_\M,\La''}\circ\Res^{\J_{\La_\M}}_{\J_{\La'',\La_\M}}
\end{equation}
 on~$\bext(\La_\M)$. 
 The map~$\Psi_{\La',\La''}^0$ is bijective. 
\end{theorem}

\begin{proof}
At first:  A map~$\Psi^0_{\La',\La''}$ satisfying~\eqref{eqDefPsi0} is uniquely determined and surjective by the definition of~$\bext^0_{\La_\M}(\La')$ 
and~$\bext^0_{\La_\M}(\La'')$. Further we have $\Psi^0_{\La',\La'}=\id_{\bext_{\La_\M}^0(\La')}$ and if~$\Psi^0_{\La'.\La''}$ 
and~$\Psi^0_{\La''.\La'''}$ satisfy~\eqref{eqDefPsi0} then $\Psi^0_{\La'',\La'''}\circ\Psi^0_{\La'.\La''}$ too. 
Further if~$\Psi_{\La',\La''}^0$ exists and is bijective then we can take~$(\Psi_{\La',\La''}^0)^{-1}$ as~$\Psi_{\La'',\La'}^0$. 
Thus we only have to consider the case~$\ti{\mf{a}}(\La')\subseteq\ti{\mf{a}}(\La'')$.
We define~$\Psi^0_{\La',\La''}$ in several steps:
\begin{itemize}
\item Let~$\kappa'^0$ be a~$\b$-extension of~$\eta_{\La'}$ to~$\J^0_{\La'}$ and let $\kappa'$ be an arbitrary $\b$-ex\-tension of~$\eta_{\La'}$ to~$\J_{\La'}$ 
such that the restriction of~$\kappa'$ to~$\J^0_{\La'}$ is equivalent to~$\kappa'^0$. 
\item We choose a  class~$\kappa_{\M,\cong}\in\bext(\La_\M)$  
such that\[\Psi_{\La',\La_\M,\La'}(\Res^{\J_{\La_\M}}_{\J_{\La',\La_\M}}\kappa_{\M,\cong})=\kappa'_\cong\] and we put
$\kappa''_\cong:=\Psi_{\La'',\La_\M,\La''}(\Res^{\J_{\La_\M}}_{\J_{\La'',\La_\M}}\kappa_{\M,\cong})$.
\item Then we define:
\[\Psi^0_{\La',\La''}(\kappa'^0_\cong):=\kappa''|_{\J^0_{\La''}}=:\kappa''^0_\cong.\]
\end{itemize} 
We claim that~$\Psi^0_{\La',\La''}$ is well-defined, i.e. independent of the choices made. 
Denote
\[\kappa_\cong:=\Psi_{\La',\La',\La''}(\kappa'_\cong).\]
Then we obtain by definition and Lemma~\ref{lemIndepOfLa}
(see the paragraph after~\ref{lemIndepOfLa}~to get the analogue for the general~$\Psi$, see~\S\ref{secGeneralPsi})
\begin{eqnarray*}
 \kappa_\cong &=& (\Psi_{\La',\La_\M,\La''}\circ\Psi_{\La',\La',\La_\M})(\kappa'_\cong)\\
 &=& \Psi_{\La',\La_\M,\La''}(\Res^{\J_{\La_\M}}_{\J_{\La',\La_\M}}\kappa_{\M,\cong})\\
 &=& (\Res^{\J_{\La''}}_{\J_{\La',\La''}}\circ\Psi_{\La'',\La_\M,\La''})(\Res^{\J_{\La_\M}}_{\J_{\La'',\La_\M}}\kappa_{\M,\cong})\\
 &=& \Res^{\J_{\La''}}_{\J_{\La',\La''}}\kappa''_\cong.
\end{eqnarray*}
Thus~\eqref{eqk'k''} (with~$\La=\La'$) is satisfied. We restrict~\eqref{eqk'k''} to~$\P^0_{\La',\La'}$ (This is~$\P^0(\La'_\b)\P_1(\La')$) to obtain:
\begin{equation}\label{eqCompatibilityDefPsi0}
 \ind_{\J^0_{\La',\La''}}^{\P^0_{\La',\La'}}\kappa''|_{\J^0_{\La',\La''}}\cong \ind_{\J_{\La'}^0}^{\P^0_{\La',\La'}}\kappa'^0
\end{equation}
Both sides are irreducible, because their restrictions to~$\P_{\La'}^1$ are. 
This implies that~$\kappa'^0$ uniquely determines the isomorphism class of the restriction of~$\kappa''^0_\cong$ to~$\J^0_{\La',\La''}$, 
because, as in the proof of Lemma~\ref{lemMutMutStev4.3},~$\eta_{\La',\La''}$ has multiplicity one in the left hand side of~\eqref{eqCompatibilityDefPsi0}. 
Now mutatis mutandis as in the proof of~\cite[4.10]{stevens:08} one shows that there is only one element of~$\bext_{\La_\M}^0(\La'')$ 
extending~$\Res^{\J^0_{\La''}}_{\J^0_{\La',\La''}}\kappa''^0_\cong$. Thus~$\kappa''^0_\cong$ is uniquely determined by~$\kappa'^0_\cong$.
This shows that~$\Psi^0_{\La',\La''}$ is well-defined. On the other hand the restriction of~$\kappa''$ to~$\J^0_{\La',\La''}$ uniquely determines~$\kappa'^0_\cong$,
by equation~\eqref{eqCompatibilityDefPsi0}. Hence the injectivity of~$\Psi^0_{\La',\La''}$. 
\end{proof}

\begin{proposition}\label{propBetakappaetamixed}
Let~$\kappa'_\cong\in\bext_{\La_\M}(\La')$. Then, granted~$\timfb(\La)\subseteq\timfb(\La')$, we have that~$\eta_{\La,\La'}$ is isomorphic to the restriction of~$\kappa'$ to~$\J^1_{\La,\La'}$. 
\end{proposition}

\begin{proof}
 We consider the segment from~$\La_\M$ to~$\La'$ in the building~$\mf{B}(\G_\beta))$ of~$\G_\beta$ and pairwise different points
 \[\La_\M=\La_0,\La_1,\La_2,\ldots,\La_{l-1},\La_u=\La'\]
 on the segment such that for all indexes~$s\in\{1,\ldots,u\}$ the condition
 \[\timfa(\La_{s-1})\cap\timfa(\La_s)\in\{\timfa(\La_{s-1}),\timfa(\La_s)\}\]
 is satisfied, i.e. there is a facet in~$\mf{B}(\G)$ with respect to the weak structure, containing one of the points~$\La_{s-1}\ ,\La_s$ and such that its closure contains both points. We further get
 \[\timfb(\La_\M)\supseteq\timfb(\La_1)=\ldots\timfb(\La_{u-1})=\timfb(\La')\] 
 because~$\La_1,\La_2,\ldots,\La_{u-1}$ are elements of the same facet in the building of~$\G_\beta$. 
 By the construction of~$\beta$-extensions there is a~$\beta$-extension~$\kappa_\M$ for~$\La_\M$ such that \[\Phi_{\La',\La_\M,\La'}((\kappa_\M)_\cong)=\kappa'_\cong.\]
 The map~$\Phi_{\La',\La_\M,\La'}$ decomposes as
 \[\Phi_{\La',\La_\M,\La'}=\Phi_{\La',\La_{u-1},\La_u}\circ\Phi_{\La',\La_{u-2},\La_{u-1}}\circ\ldots\circ\Phi_{\La',\La_0,\La_1}\]
 and we prove by induction that, for every~$s$ from~$0$ to~$u$, the restriction to~$\J^1_{\La,\La_s}$ of a representation~$\kappa_s$ in the isomorphism class~$\Phi_{\La',\La_\M,\La_s}((\kappa_\M)_\cong)$ is isomorphic to~$\eta_{\La,\La_s}$.
 The base case~$(s=0)$ follows from Proposition~\ref{propExisteceOfbextLaM}\ref{propExisteceOfbextLaM-ii}. And for positive~$s$ we apply Lemma~\ref{lemIndepOfLa} and Lemma~\ref{lemMutMutStev4.3} to obtain
 the isomorphism 
 \[\ind_{J_{\La,\La_{s-1}}}^{\P_{\La''_s,\La''_s}}\kappa_{s-1}\cong\ind_{J_{\La,\La_{s}}}^{\P_{\La''_s,\La''_s}}\kappa_s\]
 where~$\La''_s$ is a lattice sequence which satisfies
 \begin{itemize}
  \item $\timfa(\La''_s)\subseteq\timfa(\La_{s-1})\cap\timfa(\La_s)$ and 
  \item $\timfb(\La''_s)=\timfb(\La)$.
 \end{itemize}
We can find such a lattice sequence on the segment between~$
 \frac12\Gamma_{\La_{s-1}}+\frac12\Gamma_{\La_s}$ and~$\Gamma_{\La}$. 
 We restrict the compatibility to~$\P^1_{\La''_s,\La''_s}$ to obtain by induction hypothesis
 the isomorphism
 \[\ind_{J^1_{\La,\La_{s-1}}}^{\P^1_{\La''_s,\La''_s}}\eta_{\La,\La_{s-1}}\cong\ind_{J^1_{\La,\La_{s}}}^{\P_{\La''_s,\La''_s}^1}\Res^{\J_{\La_s}}_{\J^1_{\La,\La_s}}(\kappa_s).\]
 The latter restriction of~$\kappa_s$ is therefore isomorphic to~$\eta_{\La,\La_s}$ by Proposition~\ref{propEtaLaLa}. 
 \end{proof}

\subsection{$\beta$-extensions for general linear groups}
We have a similar theory of $\beta$-ex\-tensions for~$\tilde{\G}$. 
We only recall the definition for~$\beta$-extensions
for the case of maximal compact subgroups. 
Let~$\Delta=[\Lambda,n,0,\beta]$ be a semisimple stratum such that~$\tilde{\P}(\Lambda_\beta)$ is a maximal compact subgroup of~$\tilde{\G}$. We fix an element~$\tilde{\theta}$ of~$\tilde{\C}(\Delta)$ 
 on~$\tilde{\H}^1_\Lambda$. 
Let~$\tilde{\eta}$ be 
the Heisenberg representation of~$\tilde{\J}^1_\Lambda$ 
containing~$\tilde{\theta}$, see Proposition~\ref{propHeisGL}.  We denote by~$\beta-\widetilde{\ext}(\Lambda)$ the set of 
all representations~$\tilde{\kappa}$ of 
 $\tilde{\J}_\Lambda:=\tilde{\J}(\beta,\Lambda)$ such that
\begin{itemize}
 \item The restriction of~$\tilde{\kappa}$  to~$\tilde{\J}^1_\Lambda$ is isomorphic to~$\tilde{\eta}$ and 
 \item the restriction of~$\tilde{\kappa}$ to a pro-$p$-Sylow subgroup of~$\tilde{\J}_\Lambda$ is 
 intertwined by every element of the centralizer~$\tilde{\G}_\beta$. 
\end{itemize}

See~\cite{stevens:08} for the case~$\D=\F$ and~\cite{secherreI:04} for the simple case.

\section{Cuspidal types}\label{secCuspType}

 In this section we construct cuspidal types for~$\G$, similar to~\cite{stevens:08} and~\cite{kurinczukStevens:19} for~$\G\otimes\L$ (cf.~\cite{bushnellKutzko:93},~\cite{secherreII:05}), and we follow their proofs. 
 Let~$\theta\in\C(\Delta)$ with~$r=0$ be a self-dual semisimple character with Heisenberg representation~$(\eta,\J^1)$.
From now on we skip the lattice sequence from the subscript if there is no cause of confusion, e.g. we write~$\J^1,\J,\eta$ for~$\J^1_\La,\J_{\La},\eta_\La$. 

\begin{definition}[cf.~\cite{MiSt} Definition~3.3]
 Let~$\rho$ be an irreducible~$\CoefField$-representations of $\P(\Lambda_\b)$ whose restriction to~$\P^0(\Lambda_\b)$ is an inflation of a direct sum  of cuspidal irreducible representations of~$\bbP(\Lambda_\b)^0(\k_\F)$. We call such a~$\rho$ a~\emph{cuspidal inflation} w.r.t.~$(\Lambda,\beta)$. 
 Let~$\kappa$ be a~$\b$-extension of~$\eta$ (with respect to a self-dual~$\o_\E$-$\o_\D$-lattice sequence~$\La_\M$ with~$\timfb(\La_\M)$ maximal such that~$\timfb(\La_\M)$ contains~$\timfb(\La)$). We  call the representation~$\lambda:=\kappa\otimes\rho$ a \emph{cuspidal type} of~$\G$ if \vspace{-0.2cm}
 \begin{itemize}
  \item the parahoric~$\P^0(\Lambda_\b)$ is maximal and 
  \item the centre of~$\G_\beta$ is compact.  
 \end{itemize}
\end{definition}

\begin{remark}\label{remCuspTypeMustBeSuppBySkewStratum}
 If~$\l$ is a cuspidal type then the underlying stratum~$\Delta$ has to be skew, i.e. the action of~$\sigma_h$ 
 on the index set~$\I$ is trivial, because of the compactness of the centre of~$\G_\beta$. 
\end{remark}
%
%
%
%

The main motivation for the definition of a~$\beta$-extension is the following theorem, which is assertion~\ref{thmMain-ii} in Theorem~\ref{thmMain}: 
\begin{theorem}[cf.{~\cite[Theorem~6.18]{stevens:08},~\cite[Theorem~12.1]{kurinczukStevens:19}}]\label{thmCuspType}
 Let~$\lambda$ be a cuspidal type. Then~$\ind_\J^\G\lambda$ is a cuspidal irreducible representation of~$\G$. 
\end{theorem}

The proof in case that~$\CoefField$ has
positive characteristic needs an extra lemma. We write $\mult_\eta(\pi)$ for the multiplicity of~$\eta$ in a smooth representation~$\pi$.

\begin{lemma}[cf.~\cite{kurinczukStevens:19} Corollary~8.5(ii)]\label{lemetaIsotypicinInduced}
 Let~$(\lambda=\kappa\otimes\rho,\J)$ be a cuspidal type. Put~$\pi=\ind_{\J}^\G\rho$. Then
 \begin{enumerate}
  \item\label{lemetaIsotypicinInduced-i} $\Hom_{\J^1}(\kappa,\pi)\simeq\rho$ over~$\J$ and 
  \item\label{lemetaIsotypicinInduced-ii} $\pi^\eta=\lambda$ ($\lambda$ is canonically contained in~$\pi$).
 \end{enumerate}
\end{lemma}

\begin{proof}
 \begin{enumerate}
  \item The proof of~\cite[Corollary~8.5(ii)]{kurinczukStevens:19}
  is valid without changes in our situation. 
  \item The representation~$(\lambda,\J)$ is contained in~$\pi^\eta$ and 
  \[\mult_\eta(\pi^\eta)=\dim_\CoefField\Hom_{\J^1}(\eta,\pi)=\dim_\CoefField(\rho)=\mult_\eta(\lambda)\]
  by~\ref{lemetaIsotypicinInduced-i}. This finishes the proof. 
 \end{enumerate}
\end{proof}

\begin{proof}[Proof of Theorem \ref{thmCuspType}]
 Let~$\lambda=\kappa\otimes\rho$ be a cuspidal type. Put~$\pi:=\ind_\J^\G\lambda$. We want to apply the irreducibility criterion~\cite[4.2]{vigneras:01} to show that~$\pi$ is irreducible and hence cuspidal irreducible. We have to show two parts:
 \begin{itemize}
  \item[Part 1:] $\I_\G(\lambda)=\J$ (which already implies the irreducibility of~$\pi$ if~$\CoefField$ has characteristic zero) and 
  \item[Part 2:] The representation~$\lambda$ is a direct summand of every smooth irreducible representation of~$\G$ whose restriction to~$\J$ has~$\lambda$ as a subrepresentation.    
 \end{itemize}

Part 1: The proof of this part is similar to the proof of~\cite[Proposition~6.18]{stevens:08}, but we are going to give a simplification avoiding the use of~\cite[Corollary~6.16]{stevens:08}. 

 Take an irreducible component~$(\rho_0,W_{\rho_0})$ of~$\rho|_{\J^0}$ and denote
~$\kappa_0=\kappa|_{\J^0}$. Then~$\kappa_0\otimes\rho_0$ is irreducible
 by Lemma~\ref{lemTensProdOffEnd}\ref{lemTensProdOffEndParti}. 
 The restriction of~$\lambda$ to~$\J^0$ is equivalent to a direct sum of~$\J/\J^0$-conjugates
 of~$\kappa_0\otimes\rho_0$, because~$\J^0$ is a normal subgroup of~$\J$. Note that~$\J/\J^0$ is isomorphic to
 $\P(\Lambda_\b)/\P^0(\Lambda_\b)$, so that the conjugating elements can be taken in~$\P(\Lambda_\b)$. 
 An element~$g\in\G$ which intertwines~$\lambda$ intertwines~$\kappa_0\otimes\rho_0$ up to~$\P(\Lambda_\b)$-conjugation, and it also intertwines~$\eta$. So it is an element of~$\J^1\G_\b\J^1$. We can therefore without loss of generality assume~$g$ as an element of~$\G_\b$ which intertwines~$\kappa_0\otimes\rho_0$. Hence as~$\I_g(\eta)$ is one-dimensional and the restriction of~$\rho_0$ to~$\J^1$ is trivial we obtain that a~$g$-intertwiner of~$\kappa_0\otimes\rho_0$, i.e. an non-zero element 
 of~$\I_g(\kappa_0\otimes\rho_0)$, has to be a tensor product of endomorphisms~$S\in\I_g(\eta)$ and~$T\in\End_\CoefField(\W_{\rho_0})$, see~\cite[Lemma 2.7]{kurinczukStevens:19}. 
 Let~$Q$ be a pro-p-Sylow subgroup of~$\J^0$. Then~$g$ is an element of $\I(\kappa|_Q)$ by the definition of~$\beta$-extension. In 
 particular~$S\in\I_g(\kappa|_Q)$, because~$\I_g(\eta)$ is~$1$-dimensional. Thus~$T\in\I_g(\rho_0|_Q)$. In particular~$g$ intertwines the 
 restriction of~$\rho$ to a pro-p-Sylow subgroup. Thus, by Morris theory, i.e. here~\cite[Lemma~7.4]{kurinczukStevens:19} and~\cite[Proposition~1.1(ii)]{stevens:08},~$g$ is an element of~$\P(\Lambda_\b)$. This finishes the proof of part one.
 
 Part 2: The proof is given in the proof of~\cite[Theorem~12.1]{kurinczukStevens:19}: An irreducible representation~$\pi'$ containing~$\lambda$ is a quotient of~$\pi$ by Frobenius reciprocity and we therefore have
 \[\lambda=\pi^\eta\twoheadrightarrow\pi'^\eta\supseteq\lambda\]
 by Lemma~\ref{lemetaIsotypicinInduced}\ref{lemetaIsotypicinInduced-ii}. Thus~$\pi'^\eta=\lambda$ and Lemma~\ref{lemIsotypic} finishes the proof.
\end{proof}

%
%
%
%

%

\section{Partitions subordinate to a stratum}\label{secIwahori}
In the proof of the exhaustion in~\cite{stevens:08} the author has to pass to decompositions of~$\V$ which are so-called exactly subordinate 
to a skew-semisimple stratum~$\Delta$ 
(with~$r=0$), 
see~\cite[Definition 6.5]{stevens:08}.
In our situation of quaternionic forms we need to generalize this approach, 
because the centralizer of~$\E_i$ in~$\End_{\D}\V^i$ is not given by the 
same vector space~$\V^i$, if~$\b_i\neq 0$, i.e.~$\End_{\E_i\otimes\D}\V^i$ is isomorphic to~$\End_{\D_\b^i}\V^i_\b$, see \S\ref{subsecCentralizer} (We have~$2\dim_{\E_i}(\V_\b^i)=\dim_{\E_i}V^i$). 
We generalize the notion of decompositions of~$\V$ which are exactly sub-ordinate to a semisimple stratum
by certain families of idempotents. (This is indicated by the arguments given in~\cite[\S5]{stevens:08}.)
We fix a semisimple stratum~$\Delta$ with~$r=0$. 
\begin{definition}
\begin{enumerate}
 \item We call a finite tuple of non-zero idempotents~$(e^{(j)})_{j\in S}$ of~$\B=\End_{\E\otimes_\F\D}(\V)$ an~\emph{$\E\otimes_\F\D$-partition} of~$\V$ 
 if $e^{(j)}e^{(k)}=0$ for all~$j\neq k,$ $j,k\in S$, and~$\sum_je^{(j)}=1$. An~$\E\otimes_\F\D$-partition~$(e^{(j)})_{j\in S}$ is called a \emph{subordinate} to 
 ~$\Delta$ if~$(\W^{(j)})_{j\in S}$ with~$\W^{(j)}:=e^{(j)}\V$ is a splitting of~$\Delta$, or equivalently if~$(\W^{(j)})_{j\in S}$ is a splitting of~$\Lambda$,
 i.e.~$e^{(j)}\in\ti{\mf{a}}(\Lambda)$,~$j\in S$. 
 \item We call an~$\E\otimes_\F\D$-partition~$(e^{(j)})_{j\in S}$ of~$\V$~\emph{properly subordinate} to~$\Delta$ if it is subordinate to~$\Delta$ and the residue class~$e^{(j)}+\tilde{\mf{b}}_1(\Lambda)$ 
 in~$\tilde{\mf{b}}(\Lambda)/\tilde{\mf{b}}_1(\Lambda)$ is a central idempotent. 
 \end{enumerate}
\end{definition}

Analogously we have the notion of ``being self-dual-subordinate to a stratum'':
\begin{definition}
Suppose~$\Delta$ is a skew semisimple stratum. Let~$(e^{(j)})_{j\in S}$ be an~$\E\otimes_\F\D$-partition of~$\V$ subordinate to~$\Delta$. 
 The partition~$(e^{(j)})_{j\in S}$ is called~\emph{self-dual-subordinate} to~$\Delta$ if the set of the idempotents~$e^{(j)}$ is~$\sigma_h$-invariant 
 with at most one fixed point. As in~\cite{stevens:08} we are then going to use a set~$S:$
 \begin{equation}\label{eqS}\{\pm 1,\ldots,\pm m\}\subseteq S\subseteq\{0,\pm 1,\ldots,\pm m\}\end{equation} as the index set, such that~$\sigma_h(e^{(j)})=e^{(-j)}$, for all~$j\in S$. Note that we have~$S=\{0\}$ if~$m=0$.  
 An~$\E\otimes_\F\D$-partition self-dual-subordinate to~$\Delta$ is called~\emph{properly self-dual-subordinate} to~$\Delta$, if the partition is properly subordinate to~$\Delta$. 
 Suppose~$(e^{(j)})_{j\in S}$ is  properly self-dual-subordinate to~$\Delta$. We call it~\emph{exactly subordinate to}~$\Delta$ if it cannot be refined by another~$\E\otimes_\F\D$-partition of~$\V$  properly self-dual-subordinate 
 to~$\Delta$. 
\end{definition}
\begin{remark}\label{remarkExSubordinate}
Let~$(e^{(j)})_{j\in S}$ be a partition of~$\V$ exactly subordinate to a skew-semisimple stratum~$\Delta=[\Lambda,n,0,\beta]$. Then, for every non-zero index~$j\in S$,
there is exactly one index~$i_j\in\I$ such that~$e^{(j)}\1^{i_j}\neq 0$. 
\end{remark}

\begin{proof}
 The set
 \[(\{e^{(j)}\1^{(i)}|\ j\in S\setminus\{0\},\ i\in\I\}\setminus \{0\})\cup\{e^{(j)}|\ j\in S\cap\{0\}\}\]
 is a refinement of the partition~$(e^{(j)})_{j\in S}$. The latter is  exactly subordinate to~$\Delta$, so both partitions coincide. This finishes the proof. 
\end{proof}

These notions of partitions subordinate to a stratum enable Iwahori decompositions as in~\cite{stevens:08}.
Let~$(e^{(j)})_{j\in S}$ be a~$\E\otimes_\F\D$-partition of~$\V$. Let~$\tilde{\M}$ be the Levi subgroup of~$\ti{\G}$ defined as:
\[\tilde{\M}:=\tiG\cap(\prod_{j\in S}\End_{\D}(\W^{(j)})).\]
Let~$\ti{\P}$ be a parabolic subgroup of~$\tiG$ with Levi~$\ti{\M}$, and write~$\ti{\U}_+$ and~$\ti{\U}_-$ for the radical of~$\ti{\P}$ and the opposite 
parabolic~$\ti{\P}^{op}$, respectively. We write~$\M,\P,\U_+$ and~$\U_-$ for the corresponding intersections with~$\G$. 

\begin{lemma}[cf. \cite{stevens:08}~5.2,~5.10]\label{lemStevensIwahori}
 Suppose~$(e^{(j)})_{j\in S}$ is subordinate to~$\Delta$.
 \begin{enumerate}
\item\label{lemStevensIwahoriPart1}  Then~$\ti{\H}^1(\b,\La)$ and~$\ti{\J}^1(\b,\La)$ have Iwahori decompositions with respect to the product~$\ti{\U}_-\ti{\M}\ti{\U}_+$.
 Moreover, the groups~$\ti{\H}(\b,\La)$ and~$\ti{\J}(\b,\La)$ have a Iwahori decomposition with
 respect to~$\ti{\U}_-\ti{\M}\ti{\U}_+$ if~$(e^{(j)})_{j\in S}$ is properly subordinate to~$\Delta$.
 \item Suppose~$\Delta$ is skew-semisimple and that~$(e^{(j)})_{j\in S}$ is self-dual-subordinate to~$\Delta$. 
 Then the groups~$\H^1(\b,\La)$ and~$\J^1(\b,\La)$ have Iwahori decompositions with respect to $\U_-\M\U_+$.
 Additionally,~$\H(\b,\La)$ and~$\J(\b,\La)$ have a Iwahori decomposition with respect to~$\U_-\M\U_+$ if~$(e^{(j)})_{j\in S}$ is properly self-dual-subordinate to~$\Delta$.
\end{enumerate}
\end{lemma}

\begin{proof}
 We just show the first assertion of~\ref{lemStevensIwahoriPart1}, because the other statements follow similarly. The idempotents satisfy
 $e^{(j)}\in\timfb(\La)\subseteq\timfb(\La_\L).$
 We apply \loccit\ to obtain for~$\ti{\M}_\L=\ti{\M}\otimes\L,\ti{\U}_{+,\L}=\ti{\U}_{+}\otimes\L$ and~$\ti{\U}_{-,\L}=\ti{\U}_{-}\otimes\L$:
 \[\ti{\H}^1(\b,\La)\subseteq
(\ti{\H}^1(\b\otimes 1,\La_\L)\cap\ti{\U}_{-,\L})(\ti{\H}^1(\b\otimes 1,\La_\L)\cap\ti{\M}_\L)(\ti{\H}^1(\b\otimes 1,\La_\L)\cap\ti{\U}_{+,\L}).\]
The~$\tau$-invariance of the three factors and the uniqueness of the Iwahori decomposition (w.r.t.~$\ti{\U}_{-,\L}\ti{\M}_\L\ti{\U}_{+,\L}$) gives the result.  
\end{proof}

Suppose that~$\Delta$ is skew-semisimple and 
that~$(e^{(j)})_{j\in S}$ is properly  self-dual-subordinate to~$\Delta$. 
As in~\S\ref{secConjCuspTypes}, as~$\Delta$ is fixed, we skip the parameters~$\La$ and~$\b$ for the sets~$\H^1,\J^1,\J,$ etc.. 
Let~$(\eta,\J^1)$ be the Heisenberg representation of 
a self-dual semisimple character~$\theta$ and let~$\kappa$ be a~$\beta$-extension of~$\eta$ with respect to some maximal~$\o_\E$-$\o_\D$-lattice sequence. 
Analogously to~\cite{stevens:08} we can introduce representations~$(\theta_\P)$,~$(\eta_\P,\J^1_\P)$ and~$(\kappa_\P,\J_\P)$. 
The corresponding groups are defined via: 
\[\J^1_\P=(\H^1\cap\U_-)(\J^1\cap\P)=(\H^1\cap\U_-)(\J^1\cap\M)(\J^1\cap\U_+)\]
and
\[\J_\P=(\H^1\cap\U_-)(\J\cap\P)=(\H^1\cap\U_-)(\J\cap\M)(\J^1\cap\U_+).\]
At first one extends~$\theta$ to a character~$\theta_\P$ trivially to~$\H^1_\P:=(\H^1\cap\U_-)(\H^1\cap\M)(\J^1\cap\U_+)$, i.e. via
\[\theta_\P(xy):=\theta(x),\ x\in\H^1,\ y\in (\J^1\cap\U_+).\]
We define~$(\eta_\P,\J^1_\P)$ as the natural representation (given by~$\eta$) on the set of~$(\J^1\cap\U_+)$-fixed vectors of~$\eta$. Similarly, 
we define~$(\kappa_\P,\J_\P)$ on the set of~$\J\cap\U_+$-fixed points, using~$\kappa$. 

Then we have the following properties: 
\begin{proposition}[cf. \cite{stevens:08}~Lemma~5.12,~Proposition~5.13 for~$\G\otimes\L$]\label{propStevens_5_12_5_13_etaP}
 $\kappa_\P$ is an extension of~$\eta_\P$ and~$\eta_\P$ is the Heisenberg representation of~$\theta_\P$ on~$\J^1_\P$. 
 Further we have~$\ind_{\J^1_\P}^{\J^1}\eta_\P\cong\eta$ and~$\ind_{\J_\P}^{\J}\kappa_\P\cong\kappa$. Further the representation~$\eta_\P$ occurs with multiplicity one in~$\eta$.
\end{proposition}

By~\cite[Lemma 5.12]{stevens:05} the natural representation~$(\eta_{\P_\L},\J^1_{\P_\L})$ of~$\eta_\L$ on the set of~$\J^1_{\P_\L}\cap\U_{\L,+}$-fixed points of~$\eta_\L$ induces to~$\eta_\L$ (for the complex case by~\loccit\ and for~$\CoefField$ by the Brauer map, as~$\Br(\eta^\CoefField_{\P_\L})=\eta^\bbC_{\P_\L})$. We will use this in the proof. 
\begin{proof}
 Note at first that~$\J^1/\H^1$ is abelian, so subgroups in between~$\H^1$ and~$\J^1$ are normal in~$\J^1$.  
 On the group~$(\J^1\cap\M)\H^1$, which we denote by~$\J^1_{\M},$ is a Heisenberg representation~$\eta_{\M}$ of~$\theta$. Every irreducible representation of~$\J^1_\P$ containing~$\th$ is of the form~$\eta_\M\otimes\phi$ for some inflation~$\phi$ of a character of~$
\J^1_\P/\J^1_\M$ (this is isomorphic to~$\J^1\cap\U_+/\H^1\cap\U_+$):
 \[\eta_\M\otimes\phi(xy):=\eta_\M(x)\phi(y),\ x\in\J^1_\M,\ y\in \J^1\cap\U_+.\]
 So, up to equivalence, the only irreducible representation of~$\J^1_\P$ containing~$\theta$ and~$\1_{\J^1\cap\U_+}$ (the trivial representation of~$\J^1\cap\U_+$) is~$\eta_\M\otimes\1_{\J^1\cap\U_+}$. Thus the Glauberman transfer~$(\gl_\CoefField^{\L|\F}(\eta_{\P_\L}),\J^1_\P)$ is equivalent to~$\eta_\M\otimes\1_{\J^1\cap\U^+}$. Now, if~$g\in\J^1$ intertwines the latter representation, so it normalizes~$\gl_\CoefField^{\L|\F}(\eta_{\P_\L})$ and therefore, by the injectivity of the Glauberman transfer,~$g$ normalizes~$\eta_{\P_\L}$ as well. Thus~$g\in\J^1_{\P_\L}\cap \G=\J^1_\P$, by Mackey, as~$\eta_{\P_\L}$ induces irreducibly to~$\J^1_\L$, and we obtain
 \[\ind_{\J^1_\P}^{\J^1}(\eta_\M\otimes\1_{\J^1\cap\U_+})\cong\eta,\]
as the left hand side is irreducible and contains~$\th$. 
The restriction of~$\eta$ to~$\J^1_\P$ is a direct sum of extensions of~$\eta_\M$, and further,~$\eta_{\M}\otimes\1_{\J^1\cap\U_+}$ has multiplicity one in~$\eta$, by Frobenius reciprocity. Thus, we conclude that~$\eta_\P$ is equivalent 
to~$\eta_\M\otimes\1_{\J^1\cap\U_+}$,
and is therefore the Heisenberg representation of~$\theta_\P$. This finishes the assertions corresponding to~$\eta_\P$. 

We now prove the assertion for~$\kappa_\P$. This representation is irreducible, because its restriction to~$\J^1_\P$ is equivalent to~$\eta_\P$, as~$\J^1\cap\U_+=\J\cap\U_+$ (The decomposition is properly subordinate to~$\Delta$!). Therefore, we have
\[\Res^\J_{\J^1}(\ind_{\J_\P}^\J\kappa_\P)\cong\ind_{\J^1_\P}^{\J^1}\eta_\P\cong\eta,\]
so~$\kappa_\P$ induces irreducibly to~$\kappa$.
\end{proof}
The restrictions of~$\theta$,\ $\eta_\P$ and~$\kappa_\P$ to~$\M$ are tensor-products, for example if~$0\in S$ then we have
\begin{itemize}
 \item $\theta|_{\H^1\cap\M}\cong (\theta_0,\H_{\La^0}^1)\otimes \bigotimes_{j>0}(\tilde{\theta}_j,\tH_{\La^j}^1) $ 
 \item $\eta_\P|_{\J^1\cap\M}\cong (\eta_0,\tilde{\H}_{\La^0}^1)\otimes\bigotimes_{j>0}(\tilde{\eta}_j,\tilde{\J}_{\La^j}^1)$ 
 \item $\kappa_\P|_{\J\cap\M}\cong (\kappa_0,\J_{\La^0})\otimes\bigotimes_{j>0}(\tilde{\kappa}_j,\tilde{\J}_{\La^j})$ 
\end{itemize}
and similarly if~$0\not\in S$.
Note that~$\eta_\P$ is a Heisenberg representation of~$\theta_\P$.

We have a proposition similar to Proposition~\ref{propStevens_5_12_5_13_etaP} for the mixed case: 
\begin{proposition}\label{propetaPmixed}
 Under the conditions of this paragraph suppose further that~$\Lambda'$ is a self-dual~$\o_\E$-$\o_\D$-lattice sequence such that~$\timfb(\La')$ is contained in~$\timfb(\La)$ 
 and such that~$(e^{(j)})_{j\in S}$ is properly subordinate to~$[\La',-,0,\beta]$. 
 Then there exist up to isomorphism exactly one representation
 $\eta_{\La',\La,\P}$ on
 \[\J^1_{\La',\La,\P}:=(\H^1_\La\cap\U_-)(\J^1_{\La',\La}\cap\P)\]
 such that it is an extension of~$\eta_\P$ and satisfies
 \[\ind_{\J^1_{\La',\La,\P}}^{\J^1_{\La',\La}}(\eta_{\La',\La,\P})\cong\eta_{\La',\La}.\]\end{proposition}
\begin{proof}
 The representation~$\eta_\P$ occurs with multiplicity one in~$\eta$ which is the restriction of~$\eta_{\La',\La}$ to~$\J^1$. This implies the uniqueness. 
 The existence is proven in three steps:
 \begin{enumerate}
  \item \label{eLLStep1} We prove that~$\J^1_{\La',\La,\P}$ is a group.
  \item \label{eLLStep2} We define~$\eta_{\La',\La,\P}$ as a certain extension of~$\eta_\P$.
  \item \label{eLLStep3} We prove that~$\eta_{\La',\La,\P}$ 
  induces to~$\eta_{\La',\La}$.  
 \end{enumerate}
 
Step~\ref{eLLStep1}: The group~$\J^1_{\La',\La}$ has an Iwahori decomposition with respect to~$\U_-\M\U_+$ because~$\J$ has one and we have
\[\J\cap\U_{\pm}=\J^1\cap\U_{\pm}.\]
The latter identity is a consequence of the partition~$(e^{(j)})_{j\in S}$ being properly subordinate to~$\Delta$. 
Further we get
\[\H^1(\J^1_{\La',\La}\cap\P)=\J_{\La',\La,\P}^1\]
because the group~$\H^1$ has an Iwahori decomposition with respect to~$\U_-\M\U_+$ and~$\J\cap\M$ normalizes~$\H^1\cap\U_+$. Now~$\H^1$  is normalized by~$\J^1_{\La',\La}\cap\P$, because it is normalized by~$\J$. 

Step~\ref{eLLStep2}: Let~$\W$ be the representation space of~$\kappa$. The restriction of~$\kappa$ to~$\J^1_{\La',\La}$ is isomorphic to~$\eta_{\La',\La}$ by Proposition~\ref{propBetakappaetamixed}, and we can therefore assume without loss of generality that the restriction of~$\kappa$ to~$\J^1_{\La',\La}$ is equal to~$\eta_{\La',\La}$. 
Now let~$\eta_{\La',\La,\P}$ be the natural representation of
$\J_{\La',\La,\P}^1$ on the set of~$\J\cap\U_+$-fixed points of~$\eta_{\La',\La,\P}$, i.e. we consider the representation
\[(\Res^{\J^1_{\La',\La}}_{\J^1_{\La',\La,\P}}(\eta_{\La',\La}),\W^{\J^1\cap\U_+}).\]
This is a restriction of~$\kappa_\P$ and an extension of~$\eta_\P$.

Step~\ref{eLLStep3}: We restrict the second isomorphism given in Proposition~\ref{propStevens_5_12_5_13_etaP} to~$\J_{\La',\La}^1$ to obtain
\[\eta_{\La',\La}\cong\ind_{\J_{\La',\La,\P}^1}^{\J^1_{\La',\La}}(\Res^{\J_\P}_{\J^1_{\La',\La,\P}}\kappa_\P)\]
which finishes the proof, because the latter restriction is~$\eta_{\La',\La,\P}$. 
\end{proof}

\begin{corollary}\label{corkPbetaext}
Suppose the partition~$(e^{(j)})_{j\in S}$ is exactly subordinate to~$\Delta$. Then \begin{enumerate}
\item~$\tilde{\kappa}_j$ is a~$2\beta_j$-extension, for all positive~$j\in S$,
\item~$\kappa_0$ is a~$\beta_0$-extension, if~$0\in S$.
\end{enumerate}
\end{corollary}

\begin{proof}
 The proof needs several steps.
 
 Step 1: We choose a minimal parahoric in~$\G_\beta\cap\M$ in the following way.  Let~$j$ be a non-negative element of~$S$. We denote by~$\G^j$ the image of~$\G$ under the projection onto~$\Aut_\D(\V^{(j)})$ and we choose an~$\o_{\E_j}$-$\o_\D$-lattice sequence~$\La^j_{m,j}$ such that~$j_{\beta_j}^{-1}(\Gamma_{\La^j_{m,j}})$ is an element in a chamber (under the weak simplicial structure) of~$\mf{B}((\G^{j})_{\beta_j})$ such that the closure of this chamber contains~$j_{\beta_j}^{-1}(\Gamma_{\La^j})$. We now take a self-dual~$\o_\E$-$\o_\D$-lattice sequence~$\La_m$ with self-dual lattice function
 \[(\bigoplus_{j\in S,~j\geq 0}\Gamma_{\La_{m,j}})\oplus (\bigoplus_{j\in S,~j>0}\Gamma_{\La_{m,j}}^{\#}).\]
Without loss of generality we can assume that~$\Gamma_{\La_m}$ is close enough to~$\Gamma_{\La}$ such that for every~$i\in\I$ the point~$j_{\beta_i}^{-1}(\Gamma_{\La^i_{m}})$ is an element of a facet whose closure contains~$j_{\beta_i}^{-1}(\Gamma_{\La^i})$. 
This implies that~$\tilde{\mf{b}}(\La_m)$ is a subset of~$\tilde{\mf{b}}(\La)$. 

 Step 2: The restriction of~$\kappa$ to~$\J^1_{\La_m,\La}$ is isomorphic to~$\eta_{\La_m,\La}$, by Proposition~\ref{propBetakappaetamixed}. Thus
 if we restrict the second isomorphism of Proposition~\ref{propStevens_5_12_5_13_etaP} to~$\J^1_{\La_m,\La}$ we obtain 
 \[\ind_{\J^1_{\La_m,\La,\P}}^{\J^1_{\La_m,\La}}(\kappa_\P|_{\J_{\La_m,\La,\P}^1})\cong\eta_{\La_m,\La}\]
 and Proposition~\ref{propetaPmixed} finishes the proof using the fact that~$\kappa_\P$ is an extension of~$\eta_\P$, i.e. we obtain that~$\kappa_\P$ is an extension of~$\eta_{\La_m,\La,\P}$. 
\end{proof}

\section{$\beta$-extensions with determinant of order twice a~$p$-power.}

We need to choose~$\beta$-extensions~$\kappa$ more carefully: they should have a big enough intertwining. For example those elements of~$\G_\b\cap\prod_{i\in\I}\mf{n}(\La_{\b_i})$ which do not permute the blocks of
$\timfb(\La)_0/\timfb(\La)_1$ should intertwine~$\kappa$. (Recall that~$\mf{n}(\La_{\b_i})$ is the normalizer of~$\La_{\b_i}$ in~$(\tiG_i)_{\b_i}$.) This is important for generalizing the proof of
\cite[Proposition 6.14]{stevens:08} to the quaternionic case. 
For that reason we consider $\beta$-extensions with determinant of order~$2p^s$,~$s\geq 0$.

\begin{definition}\label{def2pord}
 Let~$\kappa$ be a finite dimensional representation of a group~$\J$. We say that~$\kappa$ satisfies
 \textbf{(ORD)} if its determinant~$\det(\kappa)$ is a character of order dividing~$2p^s$, for some~$s\geq 0$.  
\end{definition}

\begin{notation}
 Given a pro-finite group~$\J$ we denote by~$\J^{(p)}$ the subgroup generated by all pro-$p$-Sylow subgroups of~$\J$ and we denote by~$\J^{(p,rad)}$ the pro-$p$-radical of~$\J$. 
\end{notation}

\begin{proposition}\label{prop2pord}
Let~$\Delta=[\Lambda,n,0,\beta]$ be a self-dual semisimple stratum
$\theta\in\C(\Delta)$ and~$\eta$ be the Heisenberg representation for~$\theta$. 
Let~$\Lambda_\M$ be a maximal~$\o_\E$-$\o_\D$-lattice sequence such that
$\timfb(\Lambda)\subseteq\timfb(\Lambda_\M)$. 
Then there exists a~$\beta$-extension~$\kappa$ $\in\bext_{\La_\M}(\La)$ satisfying~\textbf{(ORD)}. 
\end{proposition}

For the proof we need two lemmas.

\begin{lemma}\label{lemclassmaxORD}
 Under the assumptions of Proposition~\ref{prop2pord}, suppose~$\La_\M=\La$. Then there exists a unique~$\beta$-extension~$\kappa$ $\in\bext(\La)$ whose determinant has order a power of~$p$.
\end{lemma}

\begin{proof}
 Using Bezout's lemma, we can twist a~$\b$-extension~$\kappa\in\bext(\La)$ with a character such that its determinant has~$p$-power order, noting that the~$\CoefField$-dimension of~$\kappa$ is a power of~$p$. 
 For the uniqueness two~$\b$-extensions in~$\bext(\La)$ differ by a twist with a character~$\chi$ which is
 trivial on the subgroup~$\J_{\La}^{(p)}$  of~$\J_\La$, see~\cite[Lemma 3.10]{stevens:08}. The group~$\J_\La/\J_{\La}^{(p)}$ has no element of order~$p$. Thus~$\chi$ is trivial on~$\J_\La$ again by Bezout's lemma.
\end{proof}

\begin{lemma}\label{lemIndORD}
 Let~$\G_2$ be a finite group and~$\G_1$ be a subgroup of~$\G_2$ such that~$\G_2=\G_1\G_2^{(p,rad)}$. Suppose~$\kappa_1$ is a representation of~$\G_1$ of~$p$-power dimension and let~$\kappa_2$ be the representation of~$\G_2$ induced by~$\kappa_1$. Then
 $\kappa_2$ satisfies~\textbf{(ORD)} if and only if $\kappa_1$ satisfies~\textbf{(ORD)}.
\end{lemma}

\begin{proof}
 At first note that~$\kappa_2$ has~$p$-power dimension. We have for all~$g_1\in\G_1$ the following identity:
 \[\det(\kappa_2(g_1))=\prod_{[g]}(-1)^{l(g,g_1)+1}\det(\kappa_1(gg_1^{l(g,g_1)}g^{-1})),\]
 where~$[g]$ passes through a system of representatives for orbits of the left-action of~$\langle g_1\rangle$ on the set of~$\G_1$-right-cosets of~$\G_2$.  Here~$l(g,g_1)$ is the length of the orbit of~$\G_1g$ under this action. We obtain the if-part of the assertion, because~$\G_2$ is generated by its~$p$-radical and~$\G_1$.
 For the only-if part we choose a character~$\chi$ of~$\G_1$ of order prime to~$p$ such that the twist of~$\kappa_1$ with~$\chi$ satisfies~\textbf{(ORD)}. We need to show that~$\chi$ satisfies~\textbf{(ORD)}, i.e. that~$\chi$ is trivial or quadratic. 
 From
 \begin{equation}\label{eqGCommutatorGp}
   ([\G_2,\G_2]\G_2^{(p)})\cap\G_1=[\G_1,\G_1]\G_1^{(p)}
 \end{equation}
 follows that~$\chi$ inflates to~$\G_2$ and therefore the representation
 $\kappa_1\chi$ induces to~$\kappa_2\chi$ which satisfies~\textbf{(ORD)} by part one of the proof.  
 It follows that~$\chi$ satisfies~\textbf{(ORD)} because~$\kappa_2$ and~$\kappa_2\chi$ do. 
 To see~\eqref{eqGCommutatorGp} we note that the commutator subgroup of~$\G_2$ is contained in~$[\G_1,\G_1]\G_2^{(p,rad)}$ and that~$\G_2^{(p)}$ is contained in~$\G_1^{(p)}\G_2^{(p,rad)}$.
\end{proof}

\begin{proof}[Proof of Proposition~\ref{prop2pord}]
 The preceeding lemmas~\ref{lemclassmaxORD} and~\ref{lemIndORD} 
 imply the Proposition, because~$\bext_{\La_\M}(\La)$ is defined by~\eqref{eqk'k''} along a path from~$\La_\M$ to~$\La$ in~$\mf{B}(\G_\b)$. 
\end{proof}

An immediate consequence of Proposition~\ref{prop2pord} and Lemma~\ref{lemIndORD} is the following. 

\begin{proposition}\label{propJPORD}
 Suppose the conditions of Proposition~\ref{prop2pord}. Let~$(e_j)_{j\in \SS}$
 be a partition exactly subordinate to~$\Delta$. Consider a parabolic subgroup~$\P$ of~$\G$ with Levi subgroup~$\M$ as in~\S\ref{secIwahori}. 
 Let~$\kappa\in\bext_{\La_\M}(\La)$ be a~$\beta$-extension containing~$\theta$ such that~$\kappa$ satisfies~\textbf{(ORD)}. Then~$\kappa_\P$ on~$\J_\P$ satisfies~\textbf{(ORD)}.
\end{proposition}

Let~$\kappa_\P$ as in Proposition~\ref{propJPORD}.
We have the decomposition
\[\kappa_\P|_{\J\cap\M}\cong (\kappa_0,\J_{\La^0})\otimes\bigotimes_{j>0}(\tilde{\kappa}_j,\tilde{\J}_{\La^j}).\]
Then~$\kappa_0$ and~$\tilde{\kappa}_j,\ j>0$ satisfy~\textbf{(ORD)}.

In particular the case~$j>0$ is important. 

\begin{proposition}\label{propGLmaxORD}
 Let~$\Delta$ be a simple stratum with parameter~$r=0$, with maximal~$\Lambda_\b$,~$\tilde{\theta}\in\tilde{\C}(\Delta)$ and~$\tilde{\kappa}\in\beta-\widetilde{\ext}(\La)$ containing~$\tilde{\theta}$. Suppose~$\tilde{\kappa}$ satisfies~\textbf{(ORD)}. Then the~$\mf{n}(\La_\b)$ normalizes~$\tilde{\kappa}$. 
\end{proposition}

\begin{remark}\label{remtiKappanotnormalizedbyDtimes}
 If we skip the assumption~\textbf{(ORD)} in Proposition~\ref{propGLmaxORD} then the assertion doesn't hold. 
 Consider 
 \begin{itemize}
  \item a skewfield~$\D$ central and  of index~$2$ over~$\F$ and with residue field~$k_\D$,
  \item the group~$\tilde{G}=\D^\times$,
  \item the stratum~$\Delta=[\mathfrak{p}_\D^{\mathbb{Z}},0,0,0]$ with the trivial simple character
  $(\tilde{\theta},1+\mathfrak{p}_\D)$. 
 \end{itemize}
 Then the set~$\b$-$\widetilde{\ext}(\mathfrak{p}_\D^\mathbb{Z})$ of~$\beta$-extensions containing~$\tilde{\theta}$ is the set of inflations to~$\o_\D^\times$ of~$\CoefField$- valued characters of~$k_\D^\times$.
 The field~$\CoefField$ has characteristic different from~$p$ and we can find a group monomorphism~$\chi$ from 
 $k_\D^\times$ into~$\CoefField^\times$. There are elements in~$k_\D$ which are not fixed under the non-trivial automorphism of~$k_\D|k_\F$. Thus the inflation of~$\chi$ to~$\o_\D^\times$ is not normalized by any uniformizer of~$\D$. 
\end{remark}

\newcommand{\tikappa}{\tilde{\kappa}}
\newcommand{\tith}{\tilde{\theta}}

\begin{proof}
 Let~$g$ be an element of~$\mf{n}(\La_\b)$. 
 The representations~$\tilde{\kappa}$ and~$\tilde{\kappa}^g$ are extensions of the Heisenberg representation of~$\tith$. 
 The restriction of~$\tilde{\kappa}$ to any pro-$p$-Sylow subgroup of~$\tilde{\J}_\La$ is intertwined by~$\tiG_\b$. Thus the same holds for~$\tilde{\kappa}^g$ and therefore the latter is a~$\b$-extension for~$\tith$. Thus there is a character~$\chi$ of~$\tilde{\J}_\La$ trivial 
on~$\tilde{\J}_\La^1$ such that~$\tikappa^g$ is isomorphic to the twist of~$\tikappa$ with~$\chi$. 
 The restriction of~$\chi$ to any pro-$p$-Sylow subgroup of~$\tilde{\J}_\La$ is intertwined by~$\tiG_\b$ and therefore~$\chi$ is trivial on~$\tilde{\J}_\La^{(p)}$, see~\cite[Lemma 3.10]{stevens:08}, and therefore~$\chi$ is an inflation of a character of~$k_{\D_\b}^\times$. Thus~$\chi$ is trivial or the unique quadratic character of~$\tilde{\J}_\La$, because~$\chi$ satisfies~\textbf{(ORD)} as~$\tikappa$ and~$\tikappa^g$ do. Now
 ~$\det(\tikappa)$ and~$\det(\tikappa^g)$ have the same order, say~$2^\epsilon p^j$ for some~$\epsilon\in\{0,1\}$ and some non-negative integer~$j$. Thus the~$p^j$th power of both determinants agree. Therefore~$\chi$ is the trivial character. 
\end{proof}

\section{Standard~$\beta$-extensions}\label{secStandardbetaExt}
Given a full skew-semisimple stratum~$\Delta=[\Lambda,n,0,\beta]$ we need to fix a vertex in~$\mf{B}(\G_\beta)$ (with respect to the weak simplicial structure) to choose~$\beta$-extensions. Given that we have fixed the signed hermitian form~$h$, there is a canonical choice for this vertex. 
For the non-quaternionic case we refer to the remark in~\cite[\S4.2]{stevens:08}. 
At first we choose a standard vertex in the following way in several steps:
\begin{enumerate}
 \item[Step 1:] According to the associated splitting of~$\beta$ the lattice sequence~$\Lambda$ decomposes into~$\Lambda=\bigoplus_{i\in\I}\Lambda^i$ 
 \item[Step 2:] We choose a standard self-dual~$\o_{\D_{\beta}^i}$-lattice sequence~$\Theta^i$ in the affine class of
 $\Lambda_\beta^i$,~$i\in\I$. See~\ref{subsecCentralizer} for the definition of~$\Lambda_\beta^i$. 
 \item[Step 3:] The Bruhat-Tits building of~$\G_{\beta}^i$ contains in the facet of~$\Lambda_\beta^i$ the following vertex in the weak simplicial structure
 \[\Theta_{stmax}^i(2r+s):=\mf{p}_{\D_\beta^i}^r\Theta^i(s),\ s\in\{0,1\},\ r\in\mathbb{Z},\]
 see~\cite[Remark in~\S4.2]{stevens:08}.
 \item[Step 4:] Choose a standard self-dual $\o_{\E_i}$-$\o_\D$-lattice sequence~$\Lambda_{stmax}^i$ such that~$(\Lambda_{stmax}^i)_{\beta_i}$ and 
 $\Theta^i_{stmax}$ are in the same affine class,~$i\in\I$. We choose them such that all~$\Lambda^i_{stmax}$ have the same~$\o_\D$-period and such that this period is minimal.
 \item[Step 5:] We put
 \[\Lambda_{stmax}:=\bigoplus_{i\in\I}\Lambda_{stmax}^i\]
 and we call the vertex corresponding to~$(\Theta^i)_{i\in\I}$ in the Bruhat--Tits building of~$\G_\beta$ \emph{the standard vertex of~$\Lambda_\beta$}.
\end{enumerate}

Let~$\theta$ be a an element of~$\C(\Delta)$. 

\begin{definition} A~$\beta$-extension~$\kappa\in\beta-\ext_{\Lambda_{stmax}}(\Lambda)$ of~$\th$ is called~\emph{standard} if it satisfies~\textbf{(ORD)}.
\end{definition}

\section{Main theorems for the classification}
Given a cuspidal irreducible representation of~$\G$ then there is a semisimple character~$\theta\in\C(\Lambda,0,\b)$ contained in~$\pi$. 
Thus it contains the Heisenberg representation~$(\eta,\J^1)$ of~$\theta$ and there is an irreducible representation~$\rho$ of~$\J/\J^1$ and a~$\beta$-extension~$\kappa$ of
$\eta$ such that~$\kappa\otimes\rho$ is contained in~$\pi$. Now one has to prove:
\begin{theorem}[Exhaustion]\label{thmExhaustionG}
The representation~$\kappa\otimes\rho$ is a cuspidal type. In particular~$\ind_{\J}^\G(\kappa\otimes\rho)\cong\pi$. 
\end{theorem}
Note: The induction assertion is given by Theorem~\ref{thmCuspType}. 
The second main theorem is:
\begin{theorem}[Intertwining implies conjugacy, cf.~\cite{kurinczukSkodlerackStevens:20}~11.9 for~$\G_\L$]\label{thmIntConG}
 Suppose $(\lambda,\J)$ and~$(\lambda',\J')$ are two cuspidal types of~$\G$ which intertwine in~$\G$ (or equivalently which compactly induce 
equivalent representations of~$\G$.) Then there is an element~$g\in\G$ such that~$g\J g^{-1}=\J'$ and~$ ^g\lambda$ is equivalent to~$\lambda'$. 
\end{theorem}

The proofs of those two Theorems will occupy the next two sections.
\section{Exhaustion}
\subsection{Bushnell--Kutzko type theory for showing non-cuspidality}
There is a well--know procedure to show exhaustion results using the computational results in~\cite{bushnellKutzko:98}.
We give a remark for the modular case and we exclusively use their notation in this subsection. Let here~$\G$ be the set of~$\F$-rational points of a connected reductive group defined over~$\F$, and let~$\P=\M\N_u$ be a parabolic subgroup with opposite unipotent group~$\N_l$. The subscripts, see for example~$\J_u,\J_\M,\J_l$ below, are corresponding to the Iwahori decomposition with respect to~$\N_l\M\N_u$.
The statements and proofs of~\cite[6.8--7.9(i) and 7.9(ii injectivity)]{bushnellKutzko:98} carry over to the modular case if there~$\J_u$ and~$\J_l$ are supposed to be pro-$p$, but one has to undertake two modifications:
\begin{enumerate}
 \item Given a representation~$(\J,\tau)$ on a
 compact open subgroup of~$\G$ the formula for the convolution in the Hecke algebra~$\mathcal{H}(\G,\tau)$ is given by
 \[(\Phi\star\Psi)(x)=\sum_{[z]\in\J\backslash\G}\Phi(xz^{-1})\Psi(z),\]
 to avoid any Haar measure. 
 \item The formula in~\cite[6.8]{bushnellKutzko:98} needs to be replaced by:
 \[(\Phi\star\Psi)(x)=[\J_l:\J_l^w] \sum_{[y_\M]\in \J_\M \cap\J_\M^w \backslash\J_\M}\varphi(xy_\M^{-1}w^{-1})\psi(wy_\M) q(x,y_\M),\]
 for~$x\in\M$. 
\end{enumerate}
(Note that we did deliberately not include the surjectivity statement of~\cite[7.9(ii)]{bushnellKutzko:98}.)
The main ingredient for the exhaustion arguments is the following. 

\begin{theorem}[\cite{blondel:05}~(0.4),~\cite{bushnellKutzko:98}~7.9(ii injective) also for mod~$l$]\label{thmBKJacquet}
 Suppose~$(\tau,\J)$ is an irreducible representation (with coefficients in~$\CoefField$) which decomposes under the Iwahori decomposition~$\N_l\M\N_u$ and suppose that~$\J_u$ and~$\J_l$ are pro-$p$. Suppose further that there is a~$(\P,\J)$-strongly positive element~$\zeta$ in the center of~$\M$ such that there is an invertible element of the  Hecke algebra~$\mathcal{H}(\G,\tau)$ with support in~$\J\zeta\J$. Let~$(\pi,\G)$  be a smooth representation
 of~$\G$. Then the canonical map:
 \[\Hom_{\J}(\tau,\pi)\hookrightarrow\Hom_{\J_\M}(\tau,\pi_{\N_u}).\]
 is injective, where~$\pi_{\N_u}$ is the corresponding Jacquet module. In particular, if~$\P\neq\G$ and~$\tau$ is contained in~$\pi$ then~$\pi$ is not cuspidal.
\end{theorem}

\subsection{Skew characters and cuspidality}\label{secSkew}

In this section we prove:

\begin{theorem}[cf.~\cite{stevens:05}~Theorem~5.1]\label{thmCuspidalImplesSkew}
 Let~$\pi$ be a cuspidal irreducible representation of~$\G$ and $\theta\in\C(\Delta)$ with~$r=0$ be a self-dual semisimple character contained in~$\pi$. Then~$\Delta$ 
 is skew-semisimple, i.e. the adjoint involution of~$h$ acts trivially on the index set of~$\Delta$. 
\end{theorem}
Proof of Theorem~\ref{thmCuspidalImplesSkew}: 
Suppose for deriving a contradiction that~$\Delta$ is not skew-semi\-simple.
Consider the decomposition
\begin{equation}\label{eqDecompExhaust} \V=\V_+\oplus\V_0\oplus\V_-,\end{equation}
given by
$\V_\d:=\oplus_{i\in\I_\d}\V^i, ~\d\in\{+,0,-\}.$
Let~$\M$ be the Levi subgroup of~$\G$ defined over~$\F$ given by the stabilizer of the decomposition~\eqref{eqDecompExhaust}. 
The decomposition also defines unipotent subgroups:~$\N_+$, the unipotent radical of the stabilizer of the flag~$\V_+,\V_+\oplus\V_0,\V$ in~$\G$, and the 
opposite~$\N_-$. We have the Iwahori decompositions for~$\H^1$ and~$\J^1$ with respect to~$\N_-\M\N_+$, and we write~$\H^1_\d$ and
$\J^1_\d$ for the obvious intersections (e.g.~$\J^1_+:=\J^1\cap\N_+$) ,~$\d\in\{\pm,0\}$. 
Note that every irreducible representation of~$\K:=\H^1\J^1_+$ containing~$\theta$ is a character, because~$\K/\H^1$ is abelian and 
$\th$ admits an extension~$\xi$ to~$\K$ which is trivial on~$\J^1_+$. 

\begin{proposition}\label{propTrans}
 The group~$\J^1$ acts transitively on the set of characters of~$\K$ extending~$\th$.
\end{proposition}
\begin{proof}
 The group~$\K$ is normalized by~$\J^1$ because~$\J^1/\H^1$ is abelian. 
 A character of~$\K$ extending~$\th$ is contained in~$\ind_{\H^1}^{\J^1}\th$ and is therefore contained in~$\eta$. 
 Thus the~$\J^1$-action on the set of these characters must be transitive because~$\eta$ is irreducible. 
\end{proof}

For the notion of cover we refer to~\cite[\S8]{bushnellKutzko:98}. In fact we take the weaker version where we only want to consider strongly positive elements 
for the parabolic subgroup~$\Q:=\M\N_+$, i.e. not for other parabolic subgroups. This is enough for our purposes. 

\begin{proposition}\label{propCoverxi}[cf.\cite{stevens:05}~Proposition~4.6~and~\cite{bushnellKutzko:99}~Corollary~6.6]
There exists a strongly positive~$(\Q,\K)$-element~$\zeta$ of the centre of~$\M$, such that there is an invertible element of the Hecke algebra~$\mc{H}(\G,\xi)$ with support in~$\K\zeta\K$. 
\end{proposition}

\begin{proof}
The element~$\zeta$ defined as follows (see~\cite[4.5]{stevens:05})
\[\left(\begin{array}{ccc}
   \varpi_\F & 0 & 0 \\
   0 & 1 & 0 \\
   0 & 0 & \varpi_\F^{-1} \\
  \end{array}\right)
\]
is central in~$\M$ and a strongly~$(\P,\K)$-positive element. Note that~$\zeta$ and~$\zeta^{-1}$ intertwine 
~$\xi$, because~$(\xi,\K)$ respects the Iwahori-decomposition with respect to~$\N_-\M\N_+$.
Now the argument at the end of the proof in~\cite[6.6]{bushnellKutzko:99} finishes the proof, because 
\begin{eqnarray*}
 \I_{\N_+}(\xi) &\subseteq& \N_+\cap\I(\th)\\
 &=&\N_+\cap(\J^1\G_\b\J^1)\\
 &=&\N_+\cap\J^1.\\
\end{eqnarray*} and because the constant~$c$ there is equal to~$[\K:\K\cap \zeta^{-1}\K \zeta]$ and therefore not divisible by~$p_\CoefField$. Note that the condition~\cite[(7.15)]{bushnellKutzko:98} is satisfied in the modular case, because of~\cite[III~Proposition 2~(ii)]{blondel:05}.
\end{proof}

We now can finish the proof of Theorem~\ref{thmCuspidalImplesSkew}: As~$\pi$ contains~$\th$ it also must contain an extension of~$\th$ to~$\K$ and therefore, 
we obtain from Proposition~\ref{propTrans} that~$\pi$ contains~$\xi$. Thus~$\pi$ has a non-trivial Jacquet module for the parabolic~$\Q$, by Proposition~\ref{propCoverxi} and Theorem~\ref{thmBKJacquet}. A contradiction, because~$\M\neq\G$ and~$\pi$ is cuspidal. 
This finishes the proof of Theorem~\ref{thmCuspidalImplesSkew}.

%
%

\subsection{Exhaustion of cuspidal types (Proof of Theorem~\ref{thmExhaustionG})}\label{secExhaust}

The proof of exhaustion is mutatis mutandis to~\cite{stevens:08} (and to~\cite{kurinczukStevens:19} for mod-$p_\CoefField$), see also~\cite[3.3]{MiSt} for the final argument. 
We are going to give the outlook of the proof in this section and refer to the 
corresponding results in~\cite{stevens:08}. The referred statements of \cite[\S6 and 7]{stevens:08} are mutatis mutandis valid for the quaternionic case.
To start, let~$\pi$ be a cuspidal irreducible representation of~$\G$. Then, by Theorem~\ref{thmStDuke5.1ForQuatCase} and Theorem~\ref{thmCuspidalImplesSkew}, there exists a skew-semisimple character~$\th\in\C(\La,0,\b)$
such that:
\begin{equation}\label{eqTheta_in_pi}
 \theta\subseteq\pi.
\end{equation}
Let~$\La$ be chosen such that~$\timfb(\La)$ is minimal with respect to~\eqref{eqTheta_in_pi}. The aim is to show that~$\P^0(\La_\b)$ is maximal parahoric. Take any standard~$\b$-extension $(\kappa,\J_\La^0)$ of~$\th$, i.e.~$\kappa$ is a~$\b$-extension with respect to a lattice sequence~$\Lambda_{stmax}$, which  corresponds to a standard vertex of~$\Lambda_\beta$, and~$\kappa$ satisfies \textbf{(ORD)}, 
see~\S\ref{secStandardbetaExt}.
Then, by Lemma~\ref{lemTensProdOffEnd}\ref{lemTensProdOffEndPartii} there is an irreducible representation~$\rho$ of~$\P^0(\La_\b)/\P_1(\La_\b)$ such that~$\lambda:=\kappa\otimes\rho$ is contained in~$\pi$.
Note that~$\rho$ has to be cuspidal by the minimality of~$\timfb(\La)$, see~\cite[7.4]{stevens:08} and use Proposition~\ref{propExisteceOfbextLaM}\ref{propExisteceOfbextLaM-ii} instead of~\cite[Lemma 7.5]{stevens:08}. 

Further by the minimality condition on~$\timfb(\La)$ there is a tuple of idempotents~$(e_j)_{j\in S}$ exactly subordinate to~$\Delta=[\La,-,0,\b]$ such that 
~$\P^0(\La_{\b}^{(0)})$ is a maximal parahoric of~$\G^{(0)}_\b$ if~$0\in S$ is non-zero. (by~\cite[7.7]{stevens:08}; take quotient instead of subrepresentation in the definition of~\emph{lying over}). 
Indeed: There exists a tuple of lattice functions~$(\Gamma^{(i,0)})_{i\in\I}\in\prod_{i\in\I}\Latt_{h_i,\o_{\E_i},\o_\D}^1(\W^{(0)}\cap\V^i)$ such that~$\timfb(\Gamma^{(0)})\subsetneqq\timfb(\Lambda^{(0)})$ and~$\P^0(\Gamma_\b^{(0)})=\P^0(\Lambda^{(0)}_\beta)$, for $\Gamma^{(0)}:=\oplus_{i\in\I}\Gamma^{(i,0)}$, 
if~$\P^0(\La^{(0)}_\b)$ is not a maximal parahoric subgroup of~$\G_\b^{(0)}$, by Lemma~\ref{lemCompNonParahoric}.
Take a self-dual~$\o_\E$-$\o_\D$-lattice sequence~$\La'$ split by~$(e^{(j)})_{j\in S}$ such that~$\oplus_{j\neq 0}\La'^{(j)}$ is an affine translation of~$\oplus_{j\neq 0}\La^{(j)}$,~$\timfb(\La')$ is contained in~$\timfb(\La)$ and~$\timfb(\Gamma^{(0)})=\timfb(\Lambda'^{(0)})$. We then   apply~\cite[Lemma 7.7]{stevens:08} to conclude that the transfer of~$\th$ to~$\C(\La',0,\b)$ is contained in~$\pi$. A contradiction, because~$\timfb(\La'^{(0)})$ is properly contained in~$\timfb(\La^{(0)})$.

Assume~$\P^0(\La_\b)$ is not a maximal parahoric of~$\G_\b$ (Note that in our quaternionic case the center of~$\G_\b$ is compact, i.e.~$\SO(1,1)(\F)$
does not occur as a factor of~$\G_\b$). We then have~$m>0$, see~\eqref{eqS} for the choice of~$S$, i.e.~$(e_j)_{j\in S}$  has at least 2 idempotents.  
Let~$\M$ be the stabilizer in~$\G$ of the decomposition of~$\V$ given by~$(e_j)_{j\in S}$, and let~$\U$ be the set of upper unipotent elements of 
~$\G$ with respect to the latter decomposition. We put~$\P=\M\U$. Let~$\lambda_\P$ be the natural representation of~$\J^0_\P\:=\H^1_\La(\J^0_\La\cap\P)$ 
on the set of~$(\U\cap\J^0_\La)$-fixed points of~$\lambda.$ 

We need more notation corresponding to the partition~$(e_j)_{j\in S}$:
\begin{enumerate}
\item For every non-zero~$j\in S$ there is a unique~$i_j\in\I$ such that $e_j\beta_{i_j}\neq 0$.
Under the map~$j_\E$ the~$o_\E$-$\o_D$-lattice sequence corresponds to a tuple~$(\Lambda_\beta^i)_{i\in\I}$, with~$\Lambda_\beta^i$ a self-dual~$\o_{\D_\b^i}$-lattice sequence in~$\V_\b^i$. 
Further~$\V_\beta^i$ and~$\Lambda_\beta^i$ decompose under~$(e_j)_{j\in S}$:
\[\V_\beta^i=(\bigoplus_{j\neq 0,i_j=i}\V_\beta^{(j)})\oplus\V_\beta^{(i,0)},\ \La_\beta^i=(\bigoplus_{j\neq 0,i_j=i}\La_\beta^{(j)})\oplus\La_\beta^{(i,0)}\]
($\V_\beta^{(i,0)}=0$ is possible).
Then
\begin{equation}\label{eqexactsuborddecomp}\P^0(\La_\beta)/\P_1(\La_\beta)\simeq\bigtimes_{j=0}^m\P^0(\La_\beta^{(j)})/\P_1(\La_\beta^{(j)}),\end{equation}
because the~$\E\otimes_\F\D$-partition~$(e_j)_{j\in S}$ is exactly subordinate to~$\Delta$.
Without loss of generality we assume that~$\La_\beta^i$ is~\emph{standard}, i.e.
\[\ 2|e_i=e(\Lambda_\beta^{i}|o_{\D_\beta^i})\ \text{and }\Lambda_\beta^i(t)^{\#_i}=\Lambda_\beta^i(1-t),  \]
for all~$t\in\ZZ$, where~$\#_i$ is the duality operator defined by~$h_\beta^i$. 
\item For every~$j\neq 0$ there is a unique integer
$q_j$ satisfying~$-\frac{e_{i_j}}{2}<q_j<\frac{e_{i_j}}{2}$ such that
\[\Lambda_\beta^{(j)}(q_j)\supsetneqq\La_\beta^{(j)}(q_j+1),\]
and if we fix a total order on the index set~$\I$ of~$\beta$ then we choose the numbering of the idempotents~$e_j$ the way such that we have for non-zero~$j,k\in S$:
\begin{equation}\label{numbering}j<k \text{ if and only if }(i_j<i_k \text{ or }(i_j=i_k \text{ and }q_j<q_k)),\end{equation}
see \cite[\S6.2 and remark after Lemma 6.6]{stevens:08}. 
\item For every positive~$j\in S$ Stevens constructs two  Weyl group elements~$s_j$ and 
~$s_j^\varpi$ of~$\G_\beta$, which through conjugation swap the blocks corresponding to~$j$ and~$-j$ and act trivially on~$\W^{(k)}$ for~$k\neq \pm j$, see~\cite[\S 6.2]{stevens:08}. Now they are used to define an involution on~$\Aut_{\D}(\W^{(j)})$ via: 
\[\sigma_j(g_j)=s_j\sigma_h(g_j)^{-1}s_j^{-1}.\]
The representation~$\rho$ decomposes under~\ref{eqexactsuborddecomp} into
\[\rho\simeq\rho^{(0)}\otimes(\bigotimes_{j=1}^m\tilde{\rho}^{(j)}),\]
and one defines for negative~$j$:
\[\tilde{\rho}^{(-j)}=\tilde{\rho}^{(j)}\circ\sigma_j.\]
\end{enumerate}

We have now two cases to consider: 
\begin{enumerate}
 \item[Case 1:] There exists a positive~$j$ such that~$\rho^{(j)}\not\simeq\rho^{(-j)}$. In this case Stevens constructs in~\cite[7.2.1]{stevens:08} a decomposition~$\cY_{-1}\oplus\cY_0\oplus\cY_1$ with Levi~$\M'$ (the stabilizer of the decomposition) and 
 non-zero~$\cY_{-1}$,~$\cY_1$ such that the normalizer of~$\rho|_{\M\cap\P^0(\La_\b)}$ in~$\G_\b$ is contained in~$\M'$. 
 Then  by~\cite[Theorem 7.2]{bushnellKutzko:98} and~\cite[6.16]{stevens:08} (see~\cite[Lemma 9.8]{kurinczukStevens:19} for the modular case) the representation~$\lambda_\P$ satisfies the conditions of Theorem~\ref{thmBKJacquet},
 for a parabolic of~$\G$ with Levi~$\M'$. Here we use that~$\kappa_\P$ satisfies~\textbf{(ORD)}, see Proposition~\ref{propJPORD}, which allows to apply Proposition~\ref{propGLmaxORD} in the proof of~\cite[Lemma 6.14]{stevens:08}, a lemma used for~\cite[Lemma 6.16]{stevens:08}.
 \item[Case 2:] For all positive~$j$ we have~$\rho^{(j)}\simeq\rho^{(-j)}$. 
 Here let
 \[\cY_ {-1}:=e^{(-m)}\V,\ \cY_0:=(1-e^{(m)}-e^{(-m)})\V,\ \cY_ {1}:=e^{(m)}\V\]
 with stabilizer~$\M'$ in~$\G$ and (upper block triangular) parabolic~$\P'$. Here Stevens constructs in~\cite[7.2.2]{stevens:08} a strongly~$(\P',\J^0_\P)$-positive element~$\zeta$ of~$\G$
 in the centre of~$\M'$ such that there is an invertible element of~$\mc{H}(\G,\lambda_\P)$ 
 with support in~$\J^0_\P\zeta\J^0_\P$. The construction of~$\zeta$ carries mutatis mutandis over to the quaternionic case using the ordering~\eqref{numbering}. The elements~$s_m$ and~$s_m^\varpi$ in Section~\loccit\ are automatically in~$\G$ because 
 all isometries of~$h$ have reduced norm~$1$, so one does not need to consider~\loccit\ (7.2.2)((i) and~(ii)). Further, for the modular case, in the paragraph after~\cite[7.12]{stevens:08}, as indicated in~\cite[Theorem 9.9(ii)]{kurinczukStevens:19}, one needs to refer to the description of the Hecke algebra of a cuspidal representation on a maximal parahoric given by Geck--Hiss--Malle~\cite[4.2.12]{geckJacon:11}.
 Thus~$\lambda_\P$ satisfies the conditions of Theorem~\ref{thmBKJacquet} with respect to~$\P'$. 
\end{enumerate}
In either case Theorem~\ref{thmBKJacquet} and the fact that~$\M'\cap\G$ is a proper Levi subgroup of~$\G$ imply that~$\pi$
is not cuspidal. A contradiction.

\section{Conjugate cuspidal types (Proof of Theorem~\ref{thmIntConG})}\label{secConjCuspTypes}
In this section we finish the classification of cuspidal irreducible representations of~$\G$. 
%
%
 By the assumption of Theorem~\ref{thmIntConG} we are given two cuspidal types~$(\lambda,\J(\b,\La))$ and $(\lambda',\J(\b',\La'))$ 
 which induce equivalent representations of~$\G$. Let us denote the representation~$\ind_{\J}^\G\lambda$ by~$\pi$.  
 Let~$\theta\in\C(\La,0,\beta)$ and~$\theta'\in\C(\La',0,\beta')$ be the skew-semisimple characters used for the construction of~$\lambda$ 
 and~$\lambda'$. As~$\theta$ and~$\theta'$ are contained in the irreducible~$\pi$ we obtain that both have to intertwine by an element of~$\G$, say with matching~$\zeta:\I\rightarrow\I'$. 
 By~\bfII.6.10, cf.~\cite[Proposition~11.7]{kurinczukSkodlerackStevens:20}, we can assume without loss of generality that~$\beta$ and~$\beta'$ have the same characteristic polynomial and that~$\theta'$ is the transfer of~$\theta$ from~$(\b,\La)$
 to~$(\b',\La')$. One can now apply a~$\dag$-construction, see~\bfI.\S5.3~to transfer to the case where all~$\V^i$ and ~$\V^{i'}$ have the same $\D$-dimension, to then observe that the matching~$\zeta$ must fulfill that~$\b_i$ and~$\beta'_{\zeta(i)}$ have the same minimal polynomial by the unicity of the matching. We can therefore conjugate to the case~$\b=\b'$, by~\bfII.4.14. Now,~\cite[Theorem~12.3]{kurinczukStevens:19} is valid for the quaternionic case, see below. We conclude that there is an element~$g$ of~$\G$ such that~$g\J g^{-1}=\J'$ and~$ ^g\lambda\simeq\lambda'$. 
 
Let us outline~\loccit\ to show which of their statements and constructions are needed: 
Without loss of generality we can assume that~$\La$ and~$\La'$ are standard self-dual with the same~$o_\D$-period. We consider the associated self-dual~$o_\E$-$o_\D$-lattice functions~$\Gamma$ and~$\Gamma'$ respectively and, in the building of~$\G_\beta$, the lattice functions~$\Gamma_\beta=j_\beta^{-1}(\Gamma)$ and~$\Gamma'_\beta=j^{-1}_\beta(\Gamma')$. 
$\lambda$ is constructed using an irreducible representation~$\rho$ of~$\M(\Gamma_\b):=\bbP(\Gamma_\b)(\k_\F)$ with cuspidal restriction to~$\M^0(\Gamma_\b):=\bbP^0(\Gamma_\b)(\k_\F)$ and a~$\beta$-extension~$(\kappa,\J)$, 
~$\lambda=\kappa\otimes\rho$. Analogously we have~$\lambda'=\kappa'\otimes\rho'$ for respective~$\rho'$ and~$\kappa'$. We now have three pairs of functors: 
\begin{enumerate}
 \item $\R_\kappa:\ \mf{R}(\G)\ra \mf{R}(\M(\Gamma_\b))$ and~$\I_\kappa:\ \mf{R}(\M(\Gamma_\b))\ra \mf{R}(\G)$ defined via
 \[\R_\kappa(\omega):=\Hom_{\J^1}(\kappa,\omega),\ \I_\kappa(\varrho):=\ind_\J^\G(\kappa\otimes\varrho)\]
 \item $\R_{\Gamma_\b}:\ \mf{R}(\G_\b)\ra \mf{R}(\M(\Gamma_\b))$ and~$\I_{\Gamma_\b}:\ \mf{R}(\M(\Gamma_\b))\ra \mf{R}(\G_\b)$ defined via
 \[\R_{\Gamma_\b}(\omega):=\Hom_{\P_1(\Gamma_\b)}(1,\omega),\ \I_{\Gamma_\b}(\varrho):=\ind_{\P(\Gamma_\b)}^{\G_\b}(\varrho)\]
 \item $\R_{\Gamma_\b}^0:\ \mf{R}(\G_\b)\ra \mf{R}(\M^0(\Gamma_\b))$ and~$\I^0_{\Gamma_\b}:\ \mf{R}(\M^0(\Gamma_\b))\ra \mf{R}(\G_\b)$ defined via
 \[\R_{\Gamma_\b}^0(\omega):=\Hom_{\P_1(\Gamma_\b)}(1,\omega),\ \I_{\Gamma_\b}^0(\varrho):=\ind_{\P^0(\Gamma_\b)}^{\G_\b}(\varrho).\] 
\end{enumerate}

Before we start to explain their proof we want to remark that we use~\cite[8.5]{kurinczukStevens:19} which is valid for the case of~$\G$, because
the key is the exact diagram in the proof of~\cite[Lemma~5.2]{kurinczukStevens:19} which can be obtained for~$\G$ by taking~$\Gal(\L|\F)$-fixed points of the 
corresponding diagram for~$\G\otimes\L$. Now we come to their proof of~\cite[Theorem~12.3]{kurinczukStevens:19}. 
It contains two parts:

Part 1: 
The first part is to show that~$\Gamma$ and~$\Gamma'$ are~$\G_\b$-conjugate. It is implied as follows: 
We have that~$\I_\kappa(\rho)$ and~$\I_{\kappa'}(\rho')$ are isomorphic to~$\pi$, in particular isomorphic to each other, and therefore  
$\R_{\kappa}\circ\I_{\kappa'}(\rho')$ contains~$\rho$ and therefore is non-zero. 
Thus~$\R_{\Gamma_\b}\circ \I_{\Gamma'_\b}(\rho')$ is non-zero by~\cite[8.5(i)]{kurinczukStevens:19}.
Let~$\rho'^0$ be a cuspidal irreducible sub-representation of the restriction of~$\rho'$ to~$\M^0(\Gamma'_\b)$. 
Then~$\R_{\Gamma_\b}\circ \I^0_{\Gamma'_\b}(\rho'^0)$ is non-zero because it contains~$\R_{\Gamma_\b}\circ \I_{\Gamma'_\b}(\rho')$. 
Thus~$\R_{\Gamma_\b}^0\circ \I^0_{\Gamma'_\b}(\rho'^0)$ is non-zero and therefore~$\P^0(\Gamma_\b)$ and~$\P^0(\Gamma'_\b)$ are~$\G_\b$-conjugate 
by~\cite[7.2(ii)]{kurinczukStevens:19}. Therefore~$\Gamma_\b$ is~$\G_\b$-conjugate to~$\Gamma'_\b$, because~$\P^0(\Gamma_\b)$ and~$\P^0(\Gamma'_\b)$ are~$\G_\b$-conjugate maximal parahoric sub-groups of~$\G_\b$, see~\cite[after 5.2.6]{bruhatTitsII:84}. This finishes Part 1 and we can assume~$\La=\La'$ without loss of generality,
in particular~$\theta=\theta'$ and we have the same Heisenberg representation. 

Part 2 is for showing~$\lambda\simeq\lambda'$. Take a character~$\chi$ of~$\M(\Ga_\b)$ such that~$\kappa'=\kappa\otimes\chi$. 
Then we get~$\lambda'=\kappa\otimes (\chi\otimes\rho')$ and we get 
\[(\R_{\kappa}\circ\I_\kappa)(\chi\otimes\rho')\cong (\R_{\kappa}\circ\I_\kappa)(\rho)\]
where the left hand side contains~$\chi\otimes\rho'$ and the right hand side is equivalent to~$\rho$ by~\cite[8.5(ii)]{kurinczukStevens:19}.
Thus by irreducibility we obtain the existence of an isomorphism from~$\chi\otimes\rho'$ to~$\rho$, and therefore of an isomorphism from~$\lambda'$
to~$\lambda$. 

\appendix
\daniel{
\section{Erratum on semimsimple strata for~$p$-adic classical groups}\label{appErratum4p2}
This part of the appendix is a fix of~\cite[Proposition~4.2]{stevens:02}. As it is stated in  \loccit\ the proposition is false. This was pointed out by Blondel and Van-Dinh Ngo. The fix was provided by Stevens, author of~\cite{stevens:08}, in 2012, but until now not published. 
At first we need to set up the notation to discuss the proposition.
In Appendix~\ref{appErratum4p2} and~\ref{appThmB} we only work over~$\F$ (not~$\D$) and, as usual, with odd residual characteristic~$p$, and we are given an~$\epsilon$-hermitian form~$(h,\V)$ with respect to some at most quadratic extension~$\F|\F_0$ with Galois group~$\langle \bar{\ }\rangle$. We fix a uniformizer~$\varpi$ of~$\F$ with~$\varpi\in\F_0$ if~$\F|\F_0$ is unramified and~$\sigma_h(\varpi)=-\varpi$ if~$\F|\F_0$ is ramified. Further we will use a fixed uniformizer~$\varpi_0$ of~$\F_0$ which satisfies~$\varpi=\varpi_0$ if~$\F|\F_0$ is unramified and~$\varpi_0=\varpi^2$ if not.   Given a lattice sequence~$\La$ on~$\V$ Stevens defines a finite dimensional~$\k_\F$-vector space~$\ti{\La}$ via
\[\ti{\La}:=\tilde{\La}(0)\oplus\tilde{\La}(1)\oplus\ldots\oplus\tilde{\La}(e_0-1),\ \tilde{\La}(j)=\La(j)/\La(j+1),\ j\in\ZZ,\]
where~$e_0$ is the~$\F_0$-period of~$\La$ and considers its endomorphism algebra~$\End_{\k_\F}(\tilde{\La})$. The space~$\tilde{\La}(j)$ is identified with~$\tilde{\La}(j+ie_0)$ by multiplication with the~$i$th power of~$\varpi_0$, so instead of~$\tilde{\La}(j)$ we write~$\tilde{\La}(\tilde{j})$,~$\tilde{j}$ being the mod $e_0$ congruence class of~$j$. We consider a stratum~$\Delta=[\La,n,n-1,b]$. 
One defines an endomorphism~$\tilde{b}\in\End_{\k_\F}(\tilde{\La})$ by the maps\[~\tilde{\La}(\ti{j})\rightarrow\tilde{\La}(\ti{j}-\tilde{n}),~[v]\mapsto [bv],\ \tilde{j}\in\bbZ/e_0\bbZ.\]
 Now, Proposition~4.2 in \loccit\ states that if~$[\La,n,n-1,b]$ is a self-dual stratum then there is a self-dual stratum~$[\La',n',n'-1,b']$ with~$b'=b$ such that
 \[\frac{n'}{e(\La'|\F)}=\frac{n}{e(\La|\F)},\ \timfa_{1-n}(\La)\subseteq\timfa_{1-n'}(\La')\]
 with semisimple endomorphism~$\tilde{b'}$ in~$\End_{\k_\F}(\widetilde{\La'})$.
 Here is a counter example.
We take the symplectic group~$\Sp_6(\bbQ_3)$ with the usual anti-diagonal Gram matrix and we consider the self-dual lattice chain~$\La$ to the following hereditary order together with an element~$b$:
$$ 
\left(\begin{array}{cccccc}
\o &\p &\p &\o &\p &\p \\
\o &\o &\p &\o &\o &\p\\
\o &\o &\o &\o &\o &\o \\
\o &\p &\p &\o &\p &\p \\
\o &\o &\p &\o &\o &\p\\
\o &\o &\o &\o &\o &\o \\
\end{array}\right),\   
\left(\begin{array}{cccccc}
 & &-\varpi & & & \\
-1 & & & & &\\
 &-1 & & & & \\
 & & & & &\varpi \\
 & & & 1 & &\\
 & & & & 1 & \\
\end{array}\right)\varpi^{-1}, 
$$
e.g. take~$\varpi=3$. The stratum~$\Delta=[\La,2,1,b]$ is fundamental with characteristic polynomial~$\chi_\Delta(X)=(X-1)^3(X+1)^3.$
Note that~$\Lambda$ is a regular lattice chain of period~$3$. Now suppose~$\Delta'=[\La',n',n'-1,b'=b]$ is another stratum not equivalent to a null stratum such that~$\Delta$ has the same depth as~$\Delta'$, i.e.~$\frac{n}{e}=\frac{n'}{e'}$ for~$e=e(\La|\F)$ and~$e'=e(\La'|\F)$. 
The equality of depth and~$b=b'$ imply that~$\Delta$ and~$\Delta'$ are intertwining fundamental strata which share the characteristic polynomial. 
From~$\chi_{\Delta'}(0)=-1\in \k_\F^\times$ follows now that~$b'$ normalizes~$\La'$. Now if~$\tilde{b'}\in\End_{\k_\F}(\ti{\La'})$ is semisimple its minimal polynomial has to divide~$(X-1)(X+1)$ because its third power vanishes~$\widetilde{b'}$ (Note that~$3$ is the residual characteristic). In other words:~$X+1$ and~$X(X+1)$ define the same endomorphism for~$X=\ti{b'}$. Thus by homogeneity~$e'$ divides~$n'$ or~$2n'$, i.e. $\frac{2n'}{e'}$ is an integer, which is absurd, as~$e=3$ and~$n=2$.
We were able to exclude~$\Delta'$ equivalent to a null stratum immediately by~\cite[Proposition~6.9]{skodlerackStevens:18}.}
\daniel{
Given a lattice sequence~$\La$ we call another lattice sequence~$\La'$ a refinement of~$\La$ if there is positive integer~$m$ such that~$\La'(mj)=\La(j)$ for all~$j\in\bbZ$.
Now, Proposition~4.2 in~\loccit\ will be replaced by:
\begin{proposition}[S. Stevens,~2012]\label{propStevensStrataFix}
Let~$\Delta=[\La,n,n-1,b]$ be a self-dual stratum. Then there is a self-dual stratum~$\Delta'=[\La',n',n'-1,b'=b]$ such that:
\begin{enumerate}
 \item $\timfa_{1-n}(\La)\subseteq\timfa_{1-n'}(\La')$ and~$\La'$ is a refinement of~$\La$,
 \item $\frac{n}{e}=\frac{n'}{e'}$ with~$e=e(\La|\F)$ and~$e'=e(\La'|\F)$,
 \item\label{propStevensStrataFixiii} and if we define~$y=\varpi^{n/g}b^{e/g}$ with~$g=\gcd(e,n)$ we have~$\tilde{y}\in\End_{\k_\F}(\widetilde{\La'})$ is semisimple.
 \item\label{propStevensStrataFixiiiv} If~$\Delta$ is non-fundamental then we can choose~$\Delta'$ such that it further satisfies~$\ti{b}=0$~in~$\End_{\k_\F}(\widetilde{\La'})$, i.e. such that~$\Delta'$ is equivalent to a null stratum. 
 \end{enumerate}
\end{proposition}
}

\begin{proof}[Proof by S. Stevens, edited by the author] 
\daniel{In passing through the affine class of~$\La$ we can assume~$\La$ to be standard self-dual without loss of generality. Denote by~$\Phi(X)\in\o_\F[X]$ the characteristic polynomial of~$y$
and by~$\varphi(X)\in \k_\F[X]$ its reduction
modulo~$\p_\F$. Then by the definition of~$y$  there is a sign~$\eta$ which satisfies~$\sigma_h(y)=\eta y$ and~$\ov\Phi(\eta X)=\pm\Phi(X)$ and the same applies
to~$\varphi$. We write
\[
\varphi(X)=\psi_0(X)\psi_1(X),
\]
 with~$\psi_0(X)$ a power of~$X$ and coprime to~$\psi_1(X)$. By
 Hensel's Lemma, this lifts to a factorization
 \[
 \Phi(X)=\Psi_0(X)\Psi_1(X),
 \]
 such that~$\ov\Psi_i(\eta X)=\pm\Psi_i(X)$. For~$i=0,1$ we
 put~$\Y_i=\ker\Psi_i(y)$, so that~$\V=\Y_0\perp\Y_1$ and this
decomposition is stabilized by~$b$. Moreover,
putting~$\La_i(k)=\La(k)\cap\Y_i$, we
have~$\La(k)=\La_0(k)\oplus\La_1(k)$ for all~$k\in\bbZ$.
Now we pass to the graded~$\k_\F$-vector space~$\tilde{\La}$
write~$\Yy_i$ for the image in~$\ti{\La}$ of the space~$\Y_i$; that is
\[
\Yy_i=\bigoplus_{k=0}^{e_0-1}\(\La(k)\cap\Y_i\)/\(\La(k+1)\cap\Y_i\).
\]
Then we have an orthogonal decomposition~$\ti{\La}=\Yy_0\perp\Yy_1$,
with respect to the (graded) nondegenerate~$\e$-hermitian
pairing~$\ti{h}$ on~$\ti{\La}$, see~\cite[\S4]{stevens:02} where~$\tilde{h}$ is defined using that~$\La$ is standard self-dual.  
Both~$b$ and~$y$ induce homogeneous graded maps~$\ti{b}$ and~$\ti{y}$
in~$\End_{\k_\F}(\ti{\La})$, respectively of degree~$-n$ and~$0$~mod~$e_0\bbZ$. Then we can also interpret~$\Yy_i$ as the kernel of the
map~$\psi_i(\ti{y})$ (which is also homogeneous of degree~$0$) and~$b$
preserves the decomposition~$\ti{\La}=\Yy_0\perp\Yy_1$. The restriction
of~$\ti{b}$ to~$\Yy_0$ is nilpotent, since~$\ti{y}^2=\ti{b}^{2e/g}$ is
nilpotent on~$\Yy_0$. On the other hand, the restriction of~$\tilde{y}$
to~$\Yy_1$ is invertible and has a Jordan decomposition
\[
\tilde{y}|_{\Yy_1}=\tilde{y}_s + \tilde{y}_n,
\]
with both~$\tilde{y}_s,\tilde{y}_n$ homogeneous of degree~$0$. Note
that~$\tilde{y}|_{\Yy_1}$ can be written as a polynomial in~$\tilde{y}$ so the
same applies to~$\tilde{y}_n$ and~$\tilde{y}_s$; in particular, they commute
with~$\ti{b}$.
We pick an
odd integer~$m=2s-1$ such that both~$\tilde{y}_n^m=0$ and~$\tilde{b}^m|_{\Yy_0}=0$. Now we put
\[
\begin{array}{ll}
\Vv^i_{\tilde{k}} = \tilde{b}^{\,i}\(\ti{\La}(\tilde{k}+\tilde{\imath}\tilde{n})\cap\Yy_0\)
\perp \tilde{y}_n^{\,i}\(\ti{\La}(\tilde{k})\cap\Yy_1\),\quad
&\hbox{for }\tilde{k}\in\ZZ/e_0\ZZ,\ 0\le i\le m; 
\\[10pt]
\Ww^i_{\tilde{k}} = 
\displaystyle\bigcap_{q-p=2(s-1-i)}\(\Vv^p_{\tilde{k}}+\(\Vv^q_{-\tilde{k}}\)^\perp\),
&\hbox{for }\tilde{k}\in\ZZ/e_0\ZZ,\ 0\le i\le s-1; 
\\[10pt]
\Ww^i_{\tilde{k}} = \(\Ww^{m-i}_{-\tilde{k}}\)^\perp,
&\hbox{for }\tilde{k}\in\ZZ/e_0\ZZ,\ s\le i\le m.\\
\end{array}
\]
Note that
we have~$\Ww^i_{\tilde{k}}=\left(\Ww^i_{\tilde{k}}\cap\Yy_0\right)\perp \left(\Ww^i_{\tilde{k}}\cap\Yy_1\right)$ 
 and~$\tilde{b}\(\Ww^i_{\tilde{k}}\cap\Yy_0\right)\subseteq\left(\Ww^{i+1}_{\tilde{k}-\tilde{n}}\cap\Yy_0\right)$;
similarly,~$\tilde{y}_n\(\Ww^i_{\tilde{k}}\cap\Yy_1\right)\subseteq\(\Ww^{i+1}_{\tilde{k}}\cap\Yy_1\right)$. Moreover, 
\[
\ti{\La}(\tilde{k}) = \Ww^0_{\tilde{k}} \supseteq \Ww^1_{\tilde{k}} \supseteq
\cdots \supseteq \Ww^m_{\tilde{k}}=0, \quad\hbox{for }\tilde{k}\in\ZZ/e_0\ZZ,
\]
and~$\tilde{\w}_\F\Ww^i_{\tilde{k}} = \Ww^i_{\tilde{k}+\tilde{e}}$, for~$\tilde{
k}\in\ZZ/e_0\ZZ$, $0\leq i\leq m$. In particular, this gives rise to a
refinement~$\La'$ of~$\La$, with
\[
\La'(km+i)/\La(k+1)=\Ww_{\tilde{k}}^i,\qquad\hbox{for }k\in\ZZ,\ 0\le i\le m-1.
\]
We put~$n'=nm$. 
The proof is now the same as in~\cite[Proposition~4.2]{stevens:02}, but with showing that~$\tilde{y}$ is semisimple in~$\End_{\k_\F}(\widetilde{\La'})$ instead of~$\tilde{b}$ because of the change of statement in~\ref{propStevensStrataFix}\ref{propStevensStrataFixiii} compared to~\loccit. 
 Note that
the construction implies that~$y|_{\Y_0}$ induces the zero map
in~$\End_{\k_\F}(\widetilde{\La'})$ so one needs only to prove that~$y|_{\Y_1}$
induces a semisimple map.
The assertion~\ref{propStevensStrataFixiiiv} follows now from the construction, as~$\tilde{b}|_{\Yy_0}$ is null in~$\End_{\k_\F}(\widetilde{\La'})$.   
}
\end{proof}
\daniel{
Now we state Theorem~4.4 of \loccit,
We write~$\G$ for~$\U(h)$. For the notion of $\G$-split stratum see~\cite[\S2]{stevens:02}.
\begin{theorem}[cf. \cite{stevens:02}~Theorem~4.4]\label{thmStevens4p4Gsplit}
Suppose we are given a non-$\G$-split self-dual fundamental stratum~$\Delta=[\La,n,n-1,b]$ over~$\F$. Then there is a skew-semisimple stratum~$\Delta'=[\La',n',n'-1,\beta']$ of the same depth as~$\Delta$ such that 
 \begin{equation}\label{eqmacond}b+\timfa_{1-n}(\La)\subseteq\beta'+\timfa_{1-n'}(\La').\end{equation}
\end{theorem}
}

\begin{proof}[Proof by S. Stevens, edited by the author:]
 \daniel{ The theorem follows as in the proof given in \loccit\ with the following changes. The proof there uses the fact that,
a non-split fundamental stratum~$[\La,n,n-1,b]$ with~$\La$
strict is equivalent to a simple stratum if and only if~$\tilde{y}$ is
semisimple in~$\End_{\k_\F}(\ti{\La})$: this is true for strata
in~$\g$-standard form by~\cite[Proposition~2.5.8]{bushnellKutzko:93}, and follows for all
strata by~\cite[Proposition~2.5.11]{bushnellKutzko:93}. Moreover, the same result is
true for non-strict~$\La$ since one can replace~$\La$ by the
underlying lattice chain without changing the coset or the induced
map~$\tilde{y}$ (cf.~\cite[\S2.1]{broussous:99}). 
}
\end{proof}
\daniel{
This finishes the erratum.}
\daniel{
\section{The self-dual and the non-fundamental case}\label{appThmB}
The strategy of the proof of~\ref{thmStevens4p4Gsplit} carries to complementary  general cases. This all is not new, but we need it for the main part of the article and it follows directly from \S\ref{appErratum4p2}. So we state it and give the short argument. 
\begin{theorem}\label{thmStevens4p4ForNonGsplitAndAlsoForNonFundamental}
\begin{enumerate}
\item Suppose we are given a (self-dual) fundamental stratum~$\Delta=[\La,n,n-1,b]$ over~$\F$. Then there is a (self-dual) semisimple stratum~$\Delta'=[\La',n',n'-1,\beta']$ of the same depth as~$\Delta$ which satisfies   
\eqref{eqmacond}.
\item Suppose~$\Delta=[\La,n,n-1,b]$ is a non-fundamental (self-dual) stratum. Then there is a (self-dual) null-stratum~$\Delta'=[\La',n'-1,n'-1,0]$ such that
 \[b+\mf{a}_{1-n}(\La)\subseteq\mf{a}_{1-n'}(\La')\]
 and
 \[\frac{n}{e(\La|\F)}=\frac{n'}{e(\La'|\F)}.\]
\end{enumerate}
\end{theorem}
}

\begin{proof}
For the general fundamental case the proof in Theorem~\ref{thmStevens4p4Gsplit} simplifies because one does not need orthogonal sums and does not need to consider~\cite[1.10]{stevens:01}.
So let us consider the self-dual fundamental case. By the general fundamental case we can find a semisimple stratum~$\Delta'$ of the same depth as~$\Delta$ satisfying the condition~\eqref{eqmacond}, but with the self-dual lattice sequence~$\La'$ constructed in Proposition~\ref{propStevensStrataFix}. Now~$\Delta'$ is in addition self-dual because~$b$ is. It is equivalent to a self-dual semisimple stratum by~\cite[Proposition~A.9]{kurinczukSkodlerackStevens:20}.

In the second part we use that in the proof of~\ref{propStevensStrataFix} the lattice sequence~$\La'$ is constructed such that~$\widetilde{b}\in\End_{\k_\F}(\widetilde{\La'})$ is~$0$. This finishes the proof.
\end{proof}

\section{Non-parahoric subgroups on maximal self-dual orders}
\label{appMaxSelfdualOrderNonParahoric}

This is about a technical step for going from a vertex in the weak structure of the Bruhat-Tits building of a classical group to a vertex which supports a maximal parahoric subgroup. Here we also consider non-quaternionic classical groups. 
Let~$h:\V\times\V\rightarrow\D$ be an~$\epsilon$-hermitian form with respect to~$(\D,(\bar{\ }))$, where~$\D$ is a skew-field of degree at most~$2$ over~$\F$ and~$(\bar{\ })$ is an anti-involution on~$\D$. We write~$\G$ for the set of isometries of~$h$ whom if~$h$ is orthogonal we further require to have reduced norm~$1$. 
We call a hereditary order of~$\End_\D(\V)$ self-dual if it is stable under the adjoint anti-involution~$\sigma_h$ of~$h$, and those self-dual hereditary orders which are among all self-dual hereditary orders maximal under inclusion are called~\emph{maximal self-dual orders}.  
We fix a uniformizer~$\varpi_\D$ of~$\D$ such that~$\bar{\varpi}_\D=\pm\varpi_\D$.

\begin{lemma}\label{lemCompNonParahoric}
 Let~$\Gamma\in\Latt^1_{h}(\V)$ be a self-dual~$\o_\D$-lattice function such that~$\timfa(\Gamma)$ is a maximal self-dual order  and~$\P^0(\Gamma)$ is not a maximal parahoric subgroup of~$\G$. Then, there exists a self-dual~$\o_{\D}$-lattice function~$\Gamma'$, such that~$\P^0(\Gamma')=\P^0(\Gamma)$ and~$\timfa(\Gamma')\subsetneqq\timfa(\Gamma)$.
\end{lemma}

\begin{proof}
 Instead with~$\Gamma$ we work with a self-dual lattice chain~$\Lambda=\Lambda_{\Gamma}$ associated to~$\Gamma$. 
 Because of the assumptions on~$\Gamma$, after possibly multiplying~$h$ with~$\varpi_\D^{-1}$, we can assume without loss of generality, 
 \begin{itemize}
  \item $\La_0^\#=\La_1$
  \item $(\bar{\ })$ is trivial on~$\k_\D$, and
  \item that the~$\k_\D$-form~$(\bar{h},\La_0/\La_1)$, given by~$\bar{h}(\bar{v},\bar{w}):=\overline{h(v,w)}$, is orthogonal of the form~$\O(1,1)$. 
 \end{itemize}
Now we choose a lattice~$\La_{\frac{1}{2}}$ in between~$\La_0$ and~$\La_1$ such that~$\La_{\frac12}/\La_1$ is isotropic. There are exactly~$2$ choices. 
Now we add all the~$\D^\times$-translates of~$\Lambda_{\frac12}$ to~$\Lambda$ to obtain a new self-dual lattice sequence~$\La'$ which satisfies~$\timfa(\La')\subsetneqq\timfa(\La)$. We claim that~$\P^0(\La')$ and~$\P^0(\La)$ coincide. 

Proof of the claim in several steps: 

1) At first we show that~$\P^0(\La)$ is contained in~$\P(\La')$. Take an element~$g\in\P^0(\La)$. Then its reduction~$\bar{g}$ modulo~$\timfa_1(\La)$ projects to an element of~$\SU(\bar{h})$. Thus~$g\La_{\frac12}=\La_{\frac12}$, and thus~$g$ is an element of~$\P(\La')$. 

2) We show~$\P_1(\La)=\P_1(\La')$. The group~$\P_1(\La)$ is contained in~$\P_1(\La')$ because the image of~$\La'$ contains the image of~$\La$. To prove the other inclusion take an element~$g=1+x\in\P_1(\La')$. It acts as the identity in the isotropic space~$\La_{\frac12}/\La_1$. The only element of~$\O(1,1)(\k_\D)$ which coincides with the identity on one isotropic space is the identity itself. Therefore~$x\La_0\subseteq\La_1$ and~$g$ has to be an element of~$\P_1(\La)$.   

3) We show~$\P^0(\La)$ is contained in~$\P^0(\La')$.
The quotient map~$\phi$ from~$\timfa(\La')/\timfa_1(\La)$ to~$\timfa(\La')/\timfa_1(\La')$ is a~$\k_\D$-algebra map which over the algebraic closure of~$\k_\D$ maps the group~$\bbP^0(\La)$ into~$\bbP(\La')$. Thus~$\bbP^0(\La')$ contains~$\phi(\bbP^0(\La))$ and we obtain
\begin{eqnarray*}
\P^0(\La)/\P_1(\La')&=&\phi(\P^0(\La)/\P_1(\La))\\
&=&\phi(\bbP^0(\La)(\k_\D))\\
&=&\phi(\bbP^0(\La))(\k_\D)\\
&\subseteq&\bbP^0(\La')(\k_\D)\\
&=&\P^0(\La')/\P_1(\La').\\
\end{eqnarray*}
The 3rd equality arises, because~$k_\D$ is perfect: use that~$\phi$ is defined over~$k_\D$, pass to the algebraic closure and use the Galois action of~$\Gal(\bar{\k}_\D|\k_\D)$. Now, the desired containment follows from 2). 

4) It remains to prove:~$\P^0(\La')$ is contained in~$\P^0(\La)$. 
From 3) we get the following for algebraic groups over~$\k_\D$:
\[\bbP(\La)\supseteq\bbU(\timfa(\La')/\timfa_1(\La),\bar{\sigma}_h)\stackrel{\phi\times\bar{\k}_\D}{\longrightarrow}\bbU(\timfa(\La')/\timfa_1(\La'),\bar{\sigma}_h)=\bbP(\La'),\]
and the derivative of~$\phi$ at the identity is an isomorphism, because~$\mfa_1(\La')=\mfa_1(\La)$. Thus the kernel of~$\phi\times\bar{\k}_\D$ is finite, and therefore~$\bbP^0(\La),\bbP^0(\La')$ and~$\phi(\bbP^0(\La))$ have the same dimension, i.e. the latter two algebraic groups coincide. This finishes the proof. 
\end{proof}

\def\Circlearrowleft{\ensuremath{%
  \rotatebox[origin=c]{180}{$\circlearrowleft$}}}

\bibliographystyle{plain}
\bibliography{./bibliography}

\begin{thebibliography}{10}

\bibitem{abramenkoNebe:02}
Peter Abramenko and Gabriele Nebe.
\newblock Lattice chain models for affine buildings of classical type.
\newblock {\em Math. Ann.}, 322(3):537--562, 2002.

\bibitem{blondel:05}
Corinne Blondel.
\newblock Quelques propri\'{e}t\'{e}s des paires couvrantes.
\newblock {\em Math. Ann.}, 331(2):243--257, 2005.

\bibitem{broussousLemaire:02}
P.~Broussous and B.~Lemaire.
\newblock Building of {${\rm GL}(m,D)$} and centralizers.
\newblock {\em Transform. Groups}, 7(1):15--50, 2002.

\bibitem{broussousSecherreStevens:12}
P.~Broussous, V.~S{\'e}cherre, and S.~Stevens.
\newblock Smooth representations of {${\rm GL}_m(D)$} {V}: {E}ndo-classes.
\newblock {\em Doc. Math.}, 17:23--77, 2012.

\bibitem{broussousStevens:09}
P.~Broussous and S.~Stevens.
\newblock Buildings of classical groups and centralizers of {L}ie algebra
  elements.
\newblock {\em J. Lie Theory}, 19(1):55--78, 2009.

\bibitem{broussous:99}
Paul Broussous.
\newblock Minimal strata for {${\rm GL}(m,D)$}.
\newblock {\em J. Reine Angew. Math.}, 514:199--236, 1999.

\bibitem{bruhatTitsII:84}
F.~Bruhat and J.~Tits.
\newblock Groupes r\'eductifs sur un corps local. {II}. {S}ch\'emas en groupes.
  {E}xistence d'une donn\'ee radicielle valu\'ee.
\newblock {\em Inst. Hautes \'Etudes Sci. Publ. Math.}, (60):197--376, 1984.

\bibitem{bruhatTitsIII:84}
F.~Bruhat and J.~Tits.
\newblock Sch\'emas en groupes et immeubles des groupes classiques sur un corps
  local.
\newblock {\em Bull. Soc. Math. France}, 112(2):259--301, 1984.

\bibitem{bruhatTitsIV:87}
F.~Bruhat and J.~Tits.
\newblock Sch\'emas en groupes et immeubles des groupes classiques sur un corps
  local. {II}. {G}roupes unitaires.
\newblock {\em Bull. Soc. Math. France}, 115(2):141--195, 1987.

\bibitem{bushnellFroehlich:83}
C.J. Bushnell and A.~Fr{\"o}hlich.
\newblock {\em Gauss sums and {$p$}-adic division algebras}, volume 987.
\newblock Springer-Verlag, Berlin, 1983.

\bibitem{bushnellKutzko:93}
C.J. Bushnell and P.-C. Kutzko.
\newblock {\em The admissible dual of {${\rm GL}(N)$} via compact open
  subgroups}, volume 129.
\newblock Princeton University Press, Princeton, NJ, 1993.

\bibitem{bushnellKutzko:99}
C.J. Bushnell and P.C. Kutzko.
\newblock Semisimple types in {${\rm GL}_n$}.
\newblock {\em Compositio Math.}, 119(1):53--97, 1999.

\bibitem{bushnellKutzko:98}
Colin~J. Bushnell and Philip~C. Kutzko.
\newblock Smooth representations of reductive {$p$}-adic groups: structure
  theory via types.
\newblock {\em Proc. London Math. Soc. (3)}, 77(3):582--634, 1998.

\bibitem{dat:09}
J.-F. Dat.
\newblock Finitude pour les repr\'esentations lisses de groupes {$p$}-adiques.
\newblock {\em J. Inst. Math. Jussieu}, 8(2):261--333, 2009.

\bibitem{fintzen:18}
J.~Fintzen.
\newblock Types for tame p-adic groups.
\newblock {\em arXiv:1810.04198}, 2018.

\bibitem{fintzen:19}
J.~Fintzen.
\newblock Tame cuspidal representations in non-defining characteristics.
\newblock {\em arxiv:1905.06374}, 2019.

\bibitem{geckJacon:11}
Meinolf Geck and Nicolas Jacon.
\newblock {\em Representations of {H}ecke algebras at roots of unity},
  volume~15 of {\em Algebra and Applications}.
\newblock Springer-Verlag London, Ltd., London, 2011.

\bibitem{glauberman:68}
G.~Glauberman.
\newblock Correspondences of characters for relatively prime operator groups.
\newblock {\em Canad. J. Math.}, 20:1465--1488, 1968.

\bibitem{kimJL:07}
J.L. Kim.
\newblock Supercuspidal representations: an exhaustion theorem.
\newblock {\em J. Amer. Math. Soc.}, 20(2):273--320 (electronic), 2007.

\bibitem{kurinczukSkodlerackStevens:20}
Robert Kurinczuk, Daniel Skodlerack, and Shaun Stevens.
\newblock Endo-parameters for {$p$}-adic classical groups.
\newblock {\em Inventiones Mathematicae}, pages 1297--1432, 2020.

\bibitem{kurinczukStevens:19}
Robert Kurinczuk and Shaun Stevens.
\newblock Cuspidal {$\ell$}-modular representations of p-adic classical groups.
\newblock {\em Journal f{\"u}r die reine und angewandte Mathematik}, (0), 2019.

\bibitem{lemaire:09}
B.~Lemaire.
\newblock Comparison of lattice filtrations and {M}oy-{P}rasad filtrations for
  classical groups.
\newblock {\em J. Lie Theory}, 19(1):29--54, 2009.

\bibitem{minguezSecherre:14}
Alberto M\'{\i}nguez and Vincent S\'{e}cherre.
\newblock Types modulo {$\ell$} pour les formes int\'{e}rieures de {${\rm
  GL}_n$} sur un corps local non archim\'{e}dien.
\newblock {\em Proc. Lond. Math. Soc. (3)}, 109(4):823--891, 2014.
\newblock With an appendix by Vincent S\'{e}cherre et Shaun Stevens.

\bibitem{MiSt}
Michitaka Miyauchi and Shaun Stevens.
\newblock Semisimple types for {$p$}-adic classical groups.
\newblock {\em Math. Ann.}, 358(1-2):257--288, 2014.

\bibitem{moeglinVignerasWaldspurger:87}
Colette M{\oe}glin, Marie-France Vign\'{e}ras, and Jean-Loup Waldspurger.
\newblock {\em Correspondances de {H}owe sur un corps {$p$}-adique}, volume
  1291 of {\em Lecture Notes in Mathematics}.
\newblock Springer-Verlag, Berlin, 1987.

\bibitem{moyPrasad:94}
A.~Moy and G.~Prasad.
\newblock Unrefined minimal {$K$}-types for {$p$}-adic groups.
\newblock {\em Invent. Math.}, 116(1-3):393--408, 1994.

\bibitem{moyPrasad:96}
A.~Moy and G.~Prasad.
\newblock Jacquet functors and unrefined minimal {$K$}-types.
\newblock {\em Comment. Math. Helv.}, 71(1):98--121, 1996.

\bibitem{navarro:98}
G.~Navarro.
\newblock {\em Characters and blocks of finite groups}, volume 250 of {\em
  London Mathematical Society Lecture Note Series}.
\newblock Cambridge University Press, Cambridge, 1998.

\bibitem{secherreI:04}
V.~S{\'e}cherre.
\newblock Repr\'esentations lisses de {${\rm GL}(m,D)$}. {I}. {C}aract\`eres
  simples.
\newblock {\em Bull. Soc. Math. France}, 132(3):327--396, 2004.

\bibitem{secherreII:05}
V.~S{\'e}cherre.
\newblock Repr\'esentations lisses de {${\rm GL}(m,D)$}. {II}.
  {$\beta$}-extensions.
\newblock {\em Compos. Math.}, 141(6):1531--1550, 2005.

\bibitem{secherreStevensIV:08}
V.~S{\'e}cherre and S.~Stevens.
\newblock Repr\'esentations lisses de {${\rm GL}_m(D)$}. {IV}.
  {R}epr\'esentations supercuspidales.
\newblock {\em J. Inst. Math. Jussieu}, 7(3):527--574, 2008.

\bibitem{secherreStevensVI:12}
V.~S{\'e}cherre and S.~Stevens.
\newblock Smooth representations of {$\GL_m(D)$} {VI}: semisimple types.
\newblock {\em Int. Math. Res. Not. IMRN}, (13):2994--3039, 2012.

\bibitem{skodlerack:13}
D.~Skodlerack.
\newblock The centralizer of a classical group and {B}ruhat-{T}its buildings.
\newblock {\em Annales de l'Institut Fourier}, 63(2):515--546, 2013.

\bibitem{skodlerack:14}
D.~Skodlerack.
\newblock Field embeddings which are conjugate under a {$p$}-adic classical
  group.
\newblock {\em Manuscripta Mathematica}, pages 1--25, 2013.

\bibitem{skodlerack:17-1}
D.~Skodlerack.
\newblock Semisimple characters for inner forms {I}: {$\GL_m(D)$}.
\newblock {\em arxiv:1703.04904}, pages 1--37, Feb 2017.

\bibitem{skodlerackStevens:18}
D.~Skodlerack and S.~Stevens.
\newblock Intertwining semisimple characters for $p$-adic classical groups.
\newblock {\em Nagoya Mathematical Journal}, page 1–69, 2018.

\bibitem{skodlerack:20}
Daniel Skodlerack.
\newblock Semisimple characters for inner forms {II}: {Q}uaternionic forms of
  {$p$}-adic classical groups ({$p$} odd).
\newblock {\em Represent. Theory}, 24:323--359, 2020.

\bibitem{stevens:01}
S.~Stevens.
\newblock Double coset decompositions and intertwining.
\newblock {\em Manuscripta Mathematica}, 106(3):349--364, 2001.

\bibitem{stevens:01-2}
S.~Stevens.
\newblock Intertwining and supercuspidal types for {$p$}-adic classical groups.
\newblock {\em Proc. London Math. Soc. (3)}, 83(1):120--140, 2001.

\bibitem{stevens:02}
S.~Stevens.
\newblock Semisimple strata for {$p$}-adic classical groups.
\newblock {\em Ann. Sci. \'Ecole Norm. Sup. (4)}, 35(3):423--435, 2002.

\bibitem{stevens:05}
S.~Stevens.
\newblock Semisimple characters for {$p$}-adic classical groups.
\newblock {\em Duke Math. J.}, 127(1):123--173, 2005.

\bibitem{stevens:08}
S.~Stevens.
\newblock The supercuspidal representations of {$p$}-adic classical groups.
\newblock {\em Invent. Math.}, 172(2):289--352, 2008.

\bibitem{corvallisTits:79}
J.~Tits.
\newblock Reductive groups over local fields.
\newblock {\em Proc. of Symp. in Pure Math.}, 33:29--69, 1979.

\bibitem{vigneras:96}
Marie-France Vign\'{e}ras.
\newblock {\em Repr\'{e}sentations {$l$}-modulaires d'un groupe r\'{e}ductif
  {$p$}-adique avec {$l\ne p$}}, volume 137 of {\em Progress in Mathematics}.
\newblock Birkh\"{a}user Boston, Inc., Boston, MA, 1996.

\bibitem{vigneras:01}
Marie-France Vign\'{e}ras.
\newblock Irreducible modular representations of a reductive {$p$}-adic group
  and simple modules for {H}ecke algebras.
\newblock In {\em European {C}ongress of {M}athematics, {V}ol. {I}
  ({B}arcelona, 2000)}, volume 201 of {\em Progr. Math.}, pages 117--133.
  Birkh\"{a}user, Basel, 2001.

\bibitem{YuJK:01}
J.-K. Yu.
\newblock Construction of tame supercuspidal representations.
\newblock {\em J. Amer. Math. Soc.}, 14(3):579--622 (electronic), 2001.

\end{thebibliography}

\end{document}